\documentclass[a4paper, twoside, 11pt, leqno]{amsart}

\usepackage[subnum]{cases}

\usepackage{geometry}		
\usepackage{xcolor}
\usepackage{bbm}
\usepackage{dsfont}
\usepackage{todonotes}
\usepackage[normalem]{ulem} 

{%
\setcounter{enumi}{0}

\begin{enumerate}}%
{\end{enumerate} }

\definecolor{airforceblue}{rgb}{0.36, 0.54, 0.66}
\definecolor{britishracinggreen}{rgb}{0.0, 0.26, 0.15}
\definecolor{cadmiumgreen}{rgb}{0.0, 0.42, 0.24}
\definecolor{amber(sae/ece)}{rgb}{1.0, 0.49, 0.0}

\newcommand{\ee}{\normalcolor}
\newcommand{\nc}{\normalcolor}

\def\cdott{\!\cdot\!}

\usepackage{amsmath, amsmath, amsthm,amssymb}
\usepackage{graphicx}
\usepackage{subcaption} 

\usepackage{enumerate} 

\allowdisplaybreaks

\usepackage{color}
\usepackage{hyperref}
\hypersetup{
    colorlinks=true,
    linktoc=all,
    linkcolor=blue,
}

\newcommand{\R}{\mathbb{R}}
\newcommand{\rd}{\mathbb{R}^d}

\newtheorem{thm}{Theorem}[section]

\newtheorem{lemma}[thm]{Lemma}
\newtheorem{proposition}[thm]{Proposition}

\theoremstyle{remark}
\newtheorem{remark}[thm]{Remark}

\theoremstyle{Theorem}
\newtheorem{definition}[thm]{Definition}

\newenvironment{description*}%
  {\begin{description}
    \setlength{\itemsep}{0.33em}
  }
  {\end{description}}
   
\makeatletter%
\let\orgdescriptionlabel\descriptionlabel
\renewcommand*{\descriptionlabel}[1]{%
  \let\orglabel\label
  \let\label\@gobble
  \phantomsection
  \edef\@currentlabel{#1}%
  \let\label\orglabel
  \orgdescriptionlabel{#1}%
}
\makeatother

\newcommand{\vep}{\varepsilon}

\newcommand{\vfi}{\varphi}

\newcommand{\beq}{\begin{equation}}
\newcommand{\eeq}{\end{equation}}
\newcommand{\rife}[1]{{(\ref{#1})}}
\newcommand{\dive}{{\rm div}}

\newcommand{\into}{{\int_{\R^d}}}

\def\de{\delta}
\def\cI{\mathcal{L} }
\def\dschi{\mathds{1}}

\def\vep{\varepsilon}

\def\vfi{\varphi}
 \def\la{\lambda}

 \def\cL{\mathcal{L}}
  \def\cP{\mathcal{P}}
    \def\cM{\mathcal{M}}
\def\rife#1{(\ref{#1})}

\begin{document}
\title[Long time behaviour]{Long time behaviour of Mean Field Games with fractional diffusion}

 \author{Olav Ersland}
\address{Department of Mathematical Sciences, Norwegian University of Science and Technology NTNU, Norway, Email address: olave@met.no} 
\thanks{O. Ersland was supported by the Research Council of Norway through the Toppforsk (research excellence) grant agreement no. 250070 {\em Waves and Nonlinear Phenomena (WaNP)}}

\author{Espen R. Jakobsen}\address{Department of Mathematical Sciences, Norwegian University of Science and Technology NTNU, Norway
Email address: espen.jakobsen@ntnu.no
}
\thanks{E. R. Jakobsen was supported by the Research Council of Norway through the agreement no. 325114 {\em IMod. Partial differential equations, statistics and data: An interdisciplinary approach to data-based modelling}}

  \author{Alessio Porretta}
\address{Department of Mathematics, University of Rome Tor Vergata, Via della Ricerca Scientifica 1, 00133 Roma, email: porretta@mat.uniroma2.it.}
\thanks{
A. Porretta was supported by Indam Gnampa projects, by the Excellence Project MatMod@TOV of the Department of Mathematics of the University of Rome Tor Vergata,  by Kaust project ORA-2021-CRG10-4674.5 (\lq\lq Mean-field games; models, theory, and computational aspects\rq\rq) and by   Italian (EU Next Gen) PRIN project 2022W58BJ5 ({\it PDEs and optimal control methods in mean field games, population dynamics and multi-agent models}, CUP E53D23005910006).}

\maketitle


\begin{abstract}
    In this paper we study the long time behaviour of  mean field games systems with fractional diffusion, modeling the case that the  individual dynamics of the players is driven by independent jump processes and controlled through the drift term, while being confined by an external field in order to guarantee ergodicity. In the case of globally Lipschitz, locally uniformly convex Hamiltonian, and weakly coupled costs satisfying the Lasry-Lions monotonicity condition, we prove that there is a unique solution $(u_T,m_T)$ to the mean field game problem in $(0,T)$ and we show that, if $T$ is sufficiently large, $(u_T,m_T)$ satisfies the so-called turnpike property, namely it is exponentially close to the (unique) stationary ergodic state for any proportionally long intermediate time.
\end{abstract}

\tableofcontents

\section{Introduction}

In this paper we analyze a class of forward-backward systems of partial differential equations (PDEs) arising in the theory of Mean Field Games (MFGs).  This theory was initiated by J.-M. Lasry and P.-L. Lions \cite{LL06cr1,LL06cr2,lasry2007mean}, and extensively developed by P.-L. Lions in several courses held at Coll\`ege de France \cite{L-college}. Similar ideas were introduced at the same time by P. Caines, M. Huang and R. Malham\'e \cite{HCM} from the perspective of stochastic control of McKean-Vlasov equations. The purpose of this theory is to describe Nash equilibria in large populations of  rational agents, by deriving a mean-field limit of Nash equilibria in $N$-players differential games, under suitable symmetry and weak interaction assumptions. The macroscopic description proposed in the limit can be summarized by a system of PDEs coupling the equation of the value function of the generic agent with the equation of the distribution density of the population. Usually the agents are described as dynamical states,   often perturbed by (Gaussian) noise, and    the density equation takes the form of a continuity, or Fokker-Planck equation, for the corresponding law. 

In the present paper, we consider a system of this kind   involving fractional diffusions coming from non-Gaussian individual noises of Poissonian/Levy 
type,  
\begin{align}
  \begin{cases}
    \begin{aligned}
       &  -\partial_t u - \mathcal{L} u (x ) + H ( x, Du ) +  b(x) \cdott Du      = F ( x, m ( t ) ) && \text{ in } ( 0,T ) \times \R^d, \\[0.2cm]
       &  \partial_t m - \mathcal{L}^* m (x) - div \big( m \big( b ( x )
        +D_p H ( x, Du  )\big)  \big)  = 0 && \text{ in } ( 0,T ) \times \R^d, \\[0.2cm]
     &  m ( 0) = m_0  ,\qquad u (T )  = u_T,
      \end{aligned}
  \end{cases}
  \label{eqn:parabolic_MFG}
\end{align}
  where $d\geq1$ and 
  the diffusion operator $\mathcal{L}$ is the infinitesimal generator of a pure jump L\'evy process in $\R^d$,   a nonlocal 
operator defined on smooth functions $\vfi$ as 
\begin{align}
    \mathcal{L}  \varphi(x) = \int_{\R^d} \left\{\varphi ( x+z ) - \varphi ( x ) -  D \varphi ( x )\cdott  z   \dschi_{| z | \leq 1} \right\} \nu ( dz ),  
    \label{def:nonlocal_diffusion}
\end{align}
where $\nu$ is a nonnegative Radon measure satisfying the  L\'evy condition $\int (1\wedge |z|^2)\nu(dz)<\infty$.   A prototypical example is 
the fractional Laplacian $\mathcal L=-(-\Delta)^{\frac\sigma2}$ for $\sigma\in (0,2)$,   the infinitesimal generator of the $\alpha$-stable processes with $\alpha=\sigma$ and $\nu(dz)=|z|^{-d-\sigma}\,dz$. This type of nonlocal diffusions
has many applications in science and economics \cite{Wo01,MK:04,CT:Book,Ru:Book,Sc:Book} and has generated a lot of activity in different areas of mathematics like nonlocal PDEs \cite{CS:07,BG:Book,BV:Book}, stochastic processes \cite{Apple,Be:Book,BS:Book,Sa:Book}, control theory \cite{OS:Book,Ha:Book}, and nonlocal operator and potential theory \cite{BS:Book,Gr:Book,Ja:Book}, to mention some.   

We will here restrict our analysis to the diffusive   or sub-critical   case where $\sigma>1$. 

The vector field $b(x)$ in \rife{eqn:parabolic_MFG}
  is assumed to be   a confining Ornstein-Uhlenbeck (OU) drift term,
whose role is to center the mass of the system, in order to give compactness in the long time. The functions $H(x,Du)$, 
$F(x,m)$,   and $u_T(x)$   are related to the underlying control problem, induced by the dynamics and by the costs 
of the agents' optimization problem. The typical interpretation of \rife{eqn:parabolic_MFG} in   MFG   
theory is the following. A generic agent in a (large) population is represented through a dynamical state   $X_\tau$   given by a controlled L\'evy process, and faces the standard   finite horizon   stochastic control problem \cite{OS:Book,Ha:Book},
\beq\label{dyncon}
\begin{split}
& \begin{cases}
  dX_\tau= \alpha_\tau\, d\tau + b(X_\tau)d\tau + dP_\tau &  \\
X_t= x &   
\end{cases}
\\
& \qquad\qquad\qquad 
\rightsquigarrow \quad u(t,x)= \inf_{
\alpha_\tau\in \mathcal A
} {\mathbb E} \left\{ {\small \int_t^T [L(X_\tau, \alpha_\tau)+ F(X_\tau,m_\tau)] + u_T(X_T)} \right\},
\end{split}
\eeq
where $P_\tau$ is a pure jump   Levy   process   with infinitesimal generator $\mathcal L$ given by \rife{def:nonlocal_diffusion} 
\cite{Apple,Sa:Book}, $\mathcal A$ is the set of admissible $A$-valued control processes $\{\alpha_\tau\}$, and $\{m_\tau\}$ is a family of measures given exogeneously and represents 
the anticipated distribution of agents.   The Hamiltonian $H$ is defined as usual as
$$
H(x,p)= \inf_{\alpha \in A} [L(x,\alpha)+ \alpha\cdott p],
$$ 
and the functions $L,F,u_T$ represent 
  running and terminal costs, possibly depending  on 
$\{\alpha_\tau\}$ and  $\{m_\tau\}$,  with $L(x,\alpha)$ convex with respect to $\alpha$.  
The coupling function $F(x,m)$ is assumed to be   smoothing/non-local and   (at least) continuous on $\R^d\times \cP(\R^d)$, where $\cP(\R^d)$ is the space of probability measures on $\R^d$. 
By dynamic programming, the value function $u$ satisfies a 
Hamilton-Jacobi-Bellman (HJB) equation corresponding to the first equation in \rife{eqn:parabolic_MFG}, and the optimal feedback control   is   given 
by $D_pH(x,Du(t,x))$. Assuming consistency at the Nash equilibrium means that the anticipated distribution of agents coincides with the real law of the controlled process $X_\tau$, which satisfies a 
Fokker-Planck equation   corresponding to the second equation   in \rife{eqn:parabolic_MFG}   with initial distribution 
of agents given by 
$m_0$.   This explains (at macroscopic level) the above forward-backward system of PDEs.

Many results exist nowadays for MFG systems from both analytical and probabilistic perspective (PDE or stochastic control approach), and we refer the interested reader to the monographs \cite{CP-CIME,CaDe}   and references therein for an overview.
However, MFGs with non-Gaussian noise and nonlocal diffusion operators $\mathcal L$ have been much less studied until recent times. 
 Well-posedness of classical solutions was obtained in 
\cite{cesaroni+al} for stationary 
MFG systems on the torus with $\mathcal L=(-\Delta)^{\sigma/2}$ for $\sigma\in (1,2)$, and extended to time-dependent problems and $\sigma \in (0,2)$ in \cite{CiGo} 
(and weaker type of solutions for $\alpha\in (0,1]$). Using the moment-free tightness/compactness approach of \cite{ChJaKr24}, well-posedness was obtained  in \cite{ersland2020classical} for problems on the whole space for more general nonlocal operators $\mathcal L$ of order greater than $1$. 
This latter result covers the setting we work in here, in addition to operators 
like $(-\partial^2_{x_1})^{\sigma_1/2}+(-\partial^2_{x_2})^{\sigma_2/2} +\ldots +(-\partial^2_{x_d})^{\sigma_d/2}$,  tempered and non-symmetric operators from Finance \cite{CT:Book,Sc:Book,Ru:Book}, and operators generating processes with no finite moments. See \cite{JR25} for extensions to less regular data and mixed local-nonlocal diffusions.
Problems with non-separable Hamiltonians have been considered in \cite{YZL22}, 
mixed local-nonlocal Bertrand and Cournot MFGs in \cite{GIN21}, MFGs of controls with fractional Laplacians in \cite{GMR26}, and MFGs with controlled nonlocal (and local) diffusions in \cite{ChJaKr24,ChJaKr25}. Numerical approximations were studied in \cite{CEJ23,CJLU26} and a  fractional Master equation in \cite{JR25}. Recently
Finally, there are some results for MFGs with jump-diffusion processes from a probabilistic point of view in \cite{BCD,CaDe,Ca22,KoT}. 


\vskip1em
In this work, our focus is on the long-term and ergodic behavior of the system, namely on the analysis of   \eqref{eqn:parabolic_MFG} as $T \to +\infty$. In a regime of stability, it is expected that the solutions of the MFG problem in a long horizon $(0,T)$ should spend most of the intermediate time in equilibrium,  say close to a stationary ergodic solution. This kind of behavior is known under the name of {\it turnpike property}, in the terminology coined by the economist P.A. Samuelson  for optimal control problems \cite{DSS}. In the context of MFGs, 
this property has been investigated in a 
series of papers; in particular, under the 
Lasry-Lions monotonicity condition ($F(x,m)$ is monotone in 
$m$), the turnpike property was proved to hold for finite state space \cite{gomes2010discrete} and  for periodic solutions of second order problems \cite{cardaliaguet2012long,cardaliaguet2013nonlocal}, with corresponding quantitative estimates. The exponential rate of turnpike was explained in terms of nondegenerate linearization \cite{P-MTA} and the  ergodic behavior of the master equation  \cite{CaPo}. Outside the case of the Lasry-Lions condition, very little is known, but the exponential turnpike was still proved to hold under some smallness condition on the coupling terms, using either the individual noise of the agents \cite{CiPo} or some favorable condition of the optimal control problem behind \cite{Conforti+al}. Even if some weak form of ergodic behavior and average convergence has been proved for deterministic MFG systems \cite{bardi-kou,Cannarsa+al,cardaliaguet2013long,masoero2019long}, very few  quantitative results are known outside the case of second order problems. However, in  the recent preprint \cite{CiMe},  new turnpike results have been obtained for both diffusive and deterministic MFG systems, replacing the classical Lasry-Lions assumption with displacement monotonicity conditions.    

The goal of this article is to provide an exponential turnpike estimate for the solution $(u,m)$ of \eqref{eqn:parabolic_MFG}, thus extending previous results that were so far limited to the case of second order problems and mostly restricted to the case of periodic solutions. 
We will show, indeed, that for long time horizons the solution of \eqref{eqn:parabolic_MFG} will  stabilize (exponentially fast) around the solution of the following stationary ergodic MFG system 
\begin{align}
  \begin{cases}
    \begin{aligned}
       &  \lambda - \mathcal{L} \bar{u}(x) + H ( x,D \bar{u}  ) + b ( x ) \cdott D \bar{u}       = F ( x, \bar{m} ) && \text{ in }  \R^d, \\[0.2cm]
      &   - \mathcal{L}^*  \bar m (x) - div \big( \bar m \big( b ( x )
        +D_p H ( x, D \bar{u} )\big)  \big)  = 0 && \text{ in } \R^d, \\[0.2cm]
       &  \int_{\R^d } \bar{m} = 1 , \qquad  \bar{u} ( 0 ) = 0\,
      \end{aligned}
  \end{cases}
  \label{eqn:ergodic_MFG}
\end{align}
where $\lambda\in \R$ is the so-called ergodic constant.  
%
As usual this  will take the form of a  quantitative estimate stating that $(Du_T,m_T)$ is approximately $(D\bar u, \bar m)$ for any large intermediate period $(\de T, (1-\de)T)$. Namely, 
for every $t\in (0,T)$, 
\beq\label{turnest0}
\|m(t)-\bar m\|_{TV_{k}}+ [u(t)-\bar u]_{\langle x^k\rangle}  +  \|Du(t)-D\bar u \|_{L^\infty(\langle x\rangle^{-k})} \leq M (e^{-\omega t}+ e^{-\omega (T-t)}) 
\eeq
where $ \| \cdot \|_{TV_{k}}$ is the total variation norm weighted by $|x|^k$ (see \rife{tvk}),  $\|\cdot \|_{L^\infty(\langle x\rangle^{-k})}$ is the $L^\infty$ norm weighted by $|x|^{-k}$, and  $[\cdot ]_{\langle x^k\rangle} $ the  
oscillation seminorm weighted by $|x|^{-k}$ (see \eqref{weisem}).    
We refer to the next section for the precise definitions and a proper statement.  
 
\vskip0.4em
Let us stress that   the drift term $b$ plays a key role here to guarantee some confining property of the underlying dynamics, which is needed to invoke some form of ergodicity. This condition is of course not required in the  case of compact state space which  
 is the usual setting 
in the literature, e.g. when studying periodic solutions. 

In establishing \rife{turnest0}, we will assume that the Hamiltonian $H(x,p)$ is locally uniformly convex and globally Lipschitz continuous in $p$, which  includes  the case that the set of controls $A$ in \rife{dyncon} is   compact.  The coupling cost $F(x,m)$ will be assumed to be smooth and monotone on the set of probability measures, i.e. satisfying the Lasry-Lions monotonicity condition. Those assumptions will be made precise in the next Section. 
Even if a  possible generalization  of this setting could be considered (see also Remark \ref{extensions}), along the lines of other papers mentioned above, we tried to keep here the assumptions on $H$ and $F$ to their simplest version, focusing on the main novelty which is  non-Gaussian driving noise and corresponding PDEs with fractional diffusion operators. 
To our knowledge, this is the first exponential turnpike result for   MFGs driven by  
Levy processes with no Brownian component.  

 \vskip0.4em
To prove 
estimate \eqref{turnest0}, we will need to show a number of   auxiliary 
results 
both for the individual 
HJB and Fokker-Planck equations and for the full systems   appearing in \eqref{eqn:parabolic_MFG} and \eqref{eqn:ergodic_MFG}.  
In Section \ref{HJB} we show
uniform in time global Lipschitz and local second-derivative bounds on the  solution $u$ of the HJB equation, 
 bounds that will be transferred to the solution of the stationary ergodic HJB equation as well. 
Such estimates are quite classical for nonlocal Hamilton-Jacobi equations 
(see e.g. \cite{barles+al1,barles+al2}), although we will need to review and refine some of the arguments to take care of long time robustness of the bounds. 
 Many  of the technical tools needed here are confined to the Appendix. 
 The analysis of the nonlocal Fokker-Planck equation is given in Section \ref{FP-sec} and mostly relies  on results from 
\cite{porretta2024decay} which guarantee the well-posedness of a suitable dual formulation. 
 A fixed point argument and the results of the previous sections are then used in Section \ref{sec:evol} to show well-posedness for the two MFG systems \eqref{eqn:parabolic_MFG} and \eqref{eqn:ergodic_MFG}.  

While Section \ref{HJB} and Section 
\ref{FP-sec} contain some  preliminary  material  needed for the well-posedness of the MFG systems, the main core of the paper is the proof of the  estimate \rife{turnest0} contained in Section \ref{turnpi}. Here we need to carefully mix  the exponential stabilization of the two single equations and the contractivity induced by the Lasry-Lions monotonicity condition.
A nontrivial role is played by the choice of norms in which we can quantitatively measure the above properties.
To this extent, we rely on the recent approach introduced in \cite{porretta2024decay} for the decay  of solutions of nonlocal Fokker-Planck equations, quantified in weighted total variation norm, and obtained as dual to the (global in time) weighted oscillation estimates for HJB equations. However, in order to tackle the difficulty of the forward-backward coupling in the   MFG system, we will need  both to upgrade the oscillation bounds of $u$ into  truly gradient estimates and to combine some Duhamel's type estimates for the nonhomogeneous equations (see Proposition \ref{lem72} and Proposition \ref{estm}) with the usual time-contractivity of the system coming from the monotonicity of $m\mapsto F(x,m)$. 



\subsection*{Notation}
We denote by $\R^d$, $d\geq 1$, the $d$-dimensional Euclidean space (with Lebesgue measure $dx$), and by $I_{d}$  the $d \times d $ identity matrix. $B_r$ denotes the ball centered at the origin with radius $r$. For $x,y\in \R^d$, $x\cdott y$ denotes the scalar product. 

For $x \in  \R^d $, we set $ \langle x  \rangle := \sqrt{1+ | x |^{2} }$, and if needed we equally use $ \langle R  \rangle = \sqrt{1+ R^{2} }$  for a real number $R>0$. We denote by $\mathds{1}_E$ the characteristic function of a set $E\subset \R^d$. Given real numbers $a,b$, we use the standard notation $ a\vee b= \max(a,b)$ and $a\wedge b= \min(a,b)$. 

We use
$C_b(U)$ for  continuous bounded functions on a set $U\subseteq \R^d$, and $C_c(U)$ for  continuous compactly supported functions. The space $C_b^{m} (\rd)$ is the subset of $C_b (\rd)$ with bounded and continuous derivatives of order $m$. Given $T>0$, we set $Q_T := (0,T) \times \rd$ (and $\overline Q_T=  [0,T] \times \rd$); the space $C_b^{l,k} (Q_T)$ is the subset of $C_b (Q_T)$, with $l$ bounded and continuous derivatives in time,
and $k$ bounded and continuous derivatives in space. If continuity holds up to $t=0, T$, we use $C_b (\overline Q_T)$. 

For any $p\in[1,\infty]$,  measurable sets $\Omega$, and nonnegative measurable function $\phi:\Omega\to\R$, we define the following weighted $L^p$-spaces:
\begin{align}\label{wlp}
   L^{p} (\Omega, \phi ) :=  \Big\{ g: \Omega \to \R \text{ measurable} : \|g \phi \|_{L^p(\Omega)} < \infty  \Big\}.
\end{align}
When $\Omega=\R^d$, we simply write $L^{p} (\phi ):=L^{p} (\R^d, \phi )$. 

The space $W^{1,\infty}(\R^d)$ (with its usual norm) denotes bounded functions with bounded (Sobolev) derivatives, which coincides with the space of Lipschitz continuous functions in $\R^d$. Similarly, $W^{k,\infty}(\R^d)$ is the space with $k-$th bounded (weak) derivatives. Eventually, we write $u\in W^{k,\infty}_{loc}(\R^d)$ if $u$ belongs to $W^{k,\infty}(B)$ for any ball $B$ in $\rd$.  As usual, the space $L^\infty(0,T;W^{k,\infty}(\R^d))$ denotes (Bochner measurable) time-space dependent functions $u(t,x)$ such that $\|u(t)\|_{W^{k,\infty}(\rd)}$ is bounded in $(0,T)$.

For $u\in C(\R^d)$ and $\theta\in [0,1]$, we denote by $[u]_\theta$ the seminorm 
$$
[u]_{\theta}:=  \sup_{x,y\in \R^{d}}\,\, \frac{|u(x)-u(y)|}{|x-y|^\theta}
$$
and we write $C^{0,\theta}(\rd)$ for the space of H\"older continuous functions.
We also use the weighted oscillation seminorm:
\beq\label{weisem}
[u]_{\langle x\rangle^k}:=  \sup_{x,y\in \R^{d}}\,\, \frac{|u(x)-u(y)|}{\langle x\rangle^k+ \langle y\rangle^k}\,.
\eeq
We let $\cP (\rd)$ denote the space of Borel probability measures on $\rd$. This 
space is equipped with the Kantorovich-Rubinstein (or bounded-Lipschitz) distance  $d_0$, defined as
\begin{align}
    d_0 (\mu_1 , \mu_2 ) := \sup_{f \in \text{Lip}_{1,1} (\rd)} \bigg\{ \int_{\rd} f (x) d (\mu_1 - \mu_2 ) (x) \bigg\}, \qquad \mu_1,\mu_2 \in \cP (\rd),
    \label{kantorovich-rubinstein}
\end{align}
where
$\text{Lip}_{1,1} (\rd) = \Big\{ f: f \text{ is Lipschitz continuous and } \|f \|_{\infty},\|Df \|_{\infty} \leq 1 \Big\} $. Convergence in $d_0$ is equivalent to weak convergence of measures (weak* convergence in $(C_b)^*$). When considering possibly signed (finite Borel) measures, we use the space $\cM(\rd)$  with the total variation norm $\| m\|_{TV}= |m|(\R^d)$.  We denote by $\cP_k(\rd)$ (or $\cM_k$ for signed measures) the subset of   $m$ with finite $k$ moments, i.e. $\langle x  \rangle^k\in L^1(\R^d, d|m|)$. In this case we also use the weighted total variation norm
\beq\label{tvk}
\|m\|_{TV_k}:= \into \langle x  \rangle^k\, d|m|\,.
\eeq

\section{Assumptions and main result}

Let us now make precise the set of assumptions under which we are going to  study problem \rife{eqn:parabolic_MFG}.
As far as the nonlocal operator $\cL$ is concerned, we assume that it is defined as in  \rife{def:nonlocal_diffusion} for some  measure $\nu$ satisfying 
\begin{description*}
    \item[($\nu$)\label{nu'}] (L\'evy measure) $\nu$ is a positive absolutely continuous measure such that, for
    
    \ some $\sigma \in ( 1,2 )$ and $c_{\sigma,d}>0$,
    $$
    \frac 1{c_{\sigma,d}} | z |^{-d-\sigma} \leq \frac{d\nu}{dz} \leq c_{\sigma,d} | z |^{-d-\sigma}\qquad \forall z\in \R^d.
    $$  
    \end{description*}
This is a (non-degenerate) ellipticity assumption for $\cL$, and implies that $\cL$ is a fractional or pseudodifferential operator of order $\sigma$. We notice that no symmetry assumption is made on the measure $\nu(\cdot)$, so that the operator $\cL$ is not necessarily variational. However, the main example is given of course by the fractional Laplace operator.

Assuming $\sigma\in (1,2)$ means that the operator is still in a diffusive regime, and we will indeed use  some regularizing effect due to $\sigma>1$ (see e.g. Proposition \ref{regeff}). We actually believe that considering the less diffusive case $\sigma\in (0,1]$ should also be possible, but this would require a significant additional effort in the technical part that we develop later. This is why we chose not to dig into that range in the present paper, confining ourselves to the range $\sigma\in (1,2)$.

We now proceed with the assumptions required on the first order terms of equation \rife{eqn:parabolic_MFG}.
The vector field $b:\R^d\to \R^d$ is assumed   to   
satisfy
  \begin{description*}
    \item[(B1)\label{B}] (Confining drift) $b:\R^d\to \R^d$ is continuous and   there is
    $\alpha,\beta > 0$ such that 
        $$(b ( x ) - b ( y ) )\cdot  (x-y)   \geq \alpha |x-y|^{2}-\beta|x-y| \qquad \forall \, x,y \in \mathbb{R}^d.$$
    \item[(B2)\label{B2}] (Regularity of drift) $b$ is locally Lipschitz in $\R^d$.
    \end{description*}
 Of course the prototype example  is given by the Ornstein-Ulhenbeck  drift   $b(x)= \alpha\, x$.
 \vskip0.4em
 
 \begin{remark}\label{rem_cond} (a)   Assumption \ref{B}, {  together with the local boundedness of $b$},  imply that there is $R>0$ {  (depending on $\alpha, \beta$ and  e.g. on  $|b(0)|$)} such that
\begin{align}
 \label{consb1}&b(x)\cdott x\geq   \tfrac12   \alpha|x|^2 \qquad \text{for}\qquad |x|>R,\\  
\label{consb2}& (b(x)-b(y))\cdott (x-y)\geq-\beta|x-y|\qquad \forall\,\,  x,y\in\R^d.   
\end{align}
These two conditions correspond  to assumptions (2.12) and (2.13) (with $\gamma=2$) in \cite{porretta2024decay}; later on, we will use several results of this latter paper in dealing with the Fokker-Planck equation. Condition \rife{consb1} gives the confining property of the drift $b$  (obtained from \ref{B}  which implies $b(x)\cdot x\geq \alpha |x|^2- (\beta+ |b(0)|) |x|$). Condition \rife{consb2},  which is  already contained in \ref{B},  includes  a sort of one-sided \lq\lq regularity\rq\rq  assumption. 
\end{remark}

 \vskip1em
 
 We now come at the assumptions on the Hamiltonian function $H:\R^d\times \R^d\to \R$, that we suppose to enjoy the following properties:
    
    \begin{description*}
    \item[(H1)\label{H1}] (Growth) $H$ is continuous and there is  $C_{H} > 0$ such that for 
        $x,p \in \R^d $, 
        $$|H ( x,p ) | \leq C_{H} ( 1+ |p| ).$$
    \item[(H2)\label{H2}] ($p$-Lipschitz) There is  $L_{H} >0$ such that for 
        $x,p,q \in \R^d $,
        \begin{align*}
            | H ( x,p ) - H ( x,q ) | \leq L_{H} | p-q |.
        \end{align*}
    \item[(H3)\label{H3}] ($x$-Lipschitz) There is $C>0$ 
        such that for $x,y,p \in \R^d $,
        \begin{align*}
            | H ( x,p ) - H ( y,p ) | \leq C ( 1+| p |  )| x-y |. 
        \end{align*}
    \item[(H4)\label{H5}] (Convex and $C^2$) $H \in C^{2} ( \R^d \times \R^d  )$,
        and for every $K>0$ there 
        
        \quad is $\alpha_K, C_K >0$ such that for $|p|\leq K$, 
        \begin{align*}
            \alpha_K I_{d} \leq  H_{pp} ( x,p ) \leq C_K I_{d}.
        \end{align*}
        \end{description*}

        \vskip1em
        Finally, we come at the coupling cost function $F:\R^d\times \cP(\R^d)\to \R$, which is assumed to be continuous and satisfy the following:
        \begin{description*}
        
    \item[(F1)\label{F1}] (Monotonicity) 
        \begin{align*}
            &\qquad\int_{\R^d } ( F ( x,m_{1} ) - F ( x,m_{2} ) )\, d ( m_{1} - m_{2} ) \geq 0 \qquad \text{ for all} \,\, m_{1}, m_{2} \in  \cP ( \R^d  ) ,
        \end{align*}
        
    \item[(F2)\label{G}] (Bounded Lipschitz)
       there exists $C_{F} > 0$ such that 
        \begin{align*}
          |F( x, \mu ) | + \frac{| F( x, \mu ) - F ( y, \nu )|}{| x-y | + d_{0} ( \mu, \nu )}  \leq C_{F},
        \end{align*}
       for $x\in\R^d$, $\mu, \nu\in \cP(\R^d)$, where $d_0$ is defined in \eqref{kantorovich-rubinstein}.

  \item[(F3)\label{F3}]   For some $k\in(0,\sigma)$, there exists $C > 0$ such that for all $ \nu\in \cP(\rd)$ and any $\mu\in \cM(\R^d)$ with zero average, 
$$ 
[F(\cdot,\nu+\mu)-F(\cdot ,\nu)]_{\langle x^k\rangle}  +\|D_xF(\cdot,\nu+\mu)-D_xF(\cdot,\nu)\|_{L^\infty(\langle x\rangle^{-k})}\leq C  \,\|\mu\|_{TV_k}.
$$
\end{description*}

These assumptions are fairly standard in the context of MFG, except maybe for the use of $d_0$ instead of the slightly stronger Wasserstein $1$ metric $d_1$. Note that we are dealing with nonlocal/smoothing couplings $F$ and $G$, which are different from local/nonsmoothing coupling considered e.g. in \cite{cardaliaguet2012long}, \cite{CiPo}. In general, local couplings are considered to be harder in terms of technical tools.
\vskip0.4em
We can now make precise the statement of the turnpike result. This is the main result of the paper.

 \begin{thm}\label{main0} Assume hypotheses \ref{nu'}, \ref{B}--\ref{B2}, \ref{H1}--\ref{H5} and \ref{F1}--\ref{F3}. Assume that $m_0\in \cP_k$ for some $k\in (0,\sigma)$, and $u_T\in W^{1,\infty}(\rd)$. Let $(u,m)$ be the solution of \rife{eqn:parabolic_MFG} and let $(\lambda, \bar u, \bar m)$ be the solution of  \rife{eqn:ergodic_MFG}, such that  $\bar m \in \cP_{2k}$. Then there exist   $\omega, M>0$ (independent of $T$) such that 
\beq\label{turnest1}
\|m(t)-\bar m\|_{TV_{k}}+ [u(t)-\bar u]_{\langle x^k\rangle}  +  \|Du(t)-D\bar u \|_{L^\infty(\langle x\rangle^{-k})} \leq M (e^{-\omega t}+ e^{-\omega (T-t)}) 
\eeq
for every $t\in (0,T) $.
\end{thm}

\begin{remark}\label{extensions} We comment here on possible generalizations that could be considered, but certainly at the expense of much more technicality, once the  result and the strategy of proof of   Theorem \ref{main0} be understood. 

First we point out  that \ref{nu'} could  be generalized to include much more general nonlocal operators as in \cite{ersland2020classical,JR25}, in particular there is no strict reason (out of simplicity of presentation) that the L\'evy measure $\nu(dz)$ have similar polynomial behavior around the singularity $z=0$ and at infinity. In our arguments, the \lq\lq diffusive condition\rq\rq\  \ref{nu'} with $\sigma>1$ is essentially  used  for $|z|$ small,  while the tail at infinity could be different (Gaussian, slowly decaying, etc...) up to modifying the Lyapunov function used in our weighted norms. 

As in \cite{JR25},  also the superposition of second order and fractional diffusions would be possible, in our approach. Indeed, both the Lipschitz bounds and the decay estimates developed in \cite{porretta2024decay} can be proved in that setting.

As for the drift term, 
we only consider here the case of strongly confining vector fields $b$,
where  exponential convergence in time to the stationary solution is guaranteed for the Fokker-Planck equation, but one could possibly investigate as well the case of slowly confining drifts, which is well established at least for the single continuity equation, giving polynomial decay in long time (see \cite{porretta2024decay}). We expect however the full MFG system in that case to be much more delicate to handle. 

Of course, a further extension would be to consider a  coupled pay-off at final time, say  $u(T)=G(x,m(T))$ where $G$ satisfies similar conditions as $F$. However, it is understood that the turnpike property is quite indifferent to the choice of the final value $u_T$, so this generalization seems to be quite straightforward. 
\end{remark}


\section{Well-posedness, approximation, and regularity of fractional HJB equations}\label{HJB}
 In this section we give well-posedness results  for
parabolic and ergodic HJB equations and uniform Lipschitz estimates for the corresponding solutions. These results and arguments are similar to the ones in \cite{chasseigne2019priori, nguyen2019large}. We also give local bounds on the time derivative  and second-derivatives in space.  

\subsection{The time-dependent HJB equation}
The parabolic HJB equation is given by 
\begin{align}
    \begin{cases}
        \partial_t u - \mathcal{L} u (x) + H ( x,Du ) +  b(x) \cdott Du = f ( t, x ) , &\text{in } ( 0,T ) \times \R^d,  \\
        u ( 0, \cdot ) = u_{0}, & \text{in } \R^d,  
    \end{cases}
    \label{eqn:parabolic_HJB}
\end{align}
which is given here   forward in time to ease presentation.
We recall that assumption \ref{nu'} is in force. 
The nonlocal operator in \eqref{def:nonlocal_diffusion} can be defined as the sum of the operators
\begin{align*}
    &\mathcal{L}_{r} ( x, u,p ) := \int_{| z | \leq r} [ u ( x+z ) - u ( x ) -   p \cdott z  \,  \dschi_{ | z | \leq 1} ( z ) ] \nu ( dz ) ,
   \end{align*}
   and
   \begin{align*}
    &\mathcal{L}^{r}  ( x, u,p ) := \int_{| z | > r} [ u ( x+z ) - u ( x ) -  p \cdott z  \,  \dschi_{ | z | \leq 1} ( z ) ] \nu ( dz ),
\end{align*}
and we note that   $ \mathcal{L}^{1} ( x, g, D g (x) ) $ is finite  for all $g \in C (\rd) \cap L^{\infty}(\langle x\rangle^{-k})$ (defined in \eqref{wlp}) provided $k < \sigma$.   We use the notation
\beq\label{notL}
\mathcal{L}  ( x, u,p ) := \mathcal{L}_{r}  ( x, u,p ) + \mathcal{L}^{r}  ( x, u,p ) 
\eeq
whereas $\mathcal{L}  ( x, u(x),Du(x) ) $ is simply denoted as $\cL u(x)$ as in \eqref{def:nonlocal_diffusion}. 
\vskip0.5em
For  now we use the viscosity solution concept for 
\eqref{eqn:parabolic_HJB}, but we will later prove more regularity so that equation \eqref{eqn:parabolic_HJB} will hold pointwise.
Viscosity solutions for integro-differential parabolic equations are by now well established, we mostly refer to \cite{barles2008second}, \cite{JaKa05}, \cite{JaKa06}. As usual, $USC(Q_T)$ and $LSC(Q_T)$ denote upper semi-continuous, respectively lower semi-continuous, functions on $Q_T$.

\begin{definition}[Viscosity solutions] \label{def:visc_soln_parabolic}
    Assume \ref{nu'}.\smallskip\\ 
    \ (i) \ A 
  function $u \in  \text{USC} ( Q_T )\cap L^{\infty} ( Q_T)$  
    is a viscosity subsolution of \eqref{eqn:parabolic_HJB}, if for 
    any $0 < r \leq 1$,   $\phi \in C^{2} ( Q_T )$  
    and maximum point $(\hat{t},\hat{x})$ of $u - \phi$ in $(\frac{\hat{t}}{2}, \frac{3}{2} \hat{t} \wedge T) \times B_{r} ( \hat{x} )$,
    \begin{align*}
        \partial_t \phi ( \hat{t},\hat{x}) - \mathcal{L}_{r} ( \hat{x}, \phi, p ) 
        - \mathcal{L}^{r} ( \hat{x}, u, p ) 
        + b ( \hat{x} ) \cdott p  + H ( \hat{x},p ) \leq f ( \hat{t},\hat{x} ), \quad  p = D \phi ( \hat{t},\hat{x} ).
    \end{align*}\smallskip
    \noindent(ii) \ A 
   function   $u \in \text{LSC} ( Q_T ) \cap L^{\infty} (Q_T)$ 
    is a viscosity supersolution of \eqref{eqn:parabolic_HJB}, if for 
    any $0 < r \leq 1$,   $\phi \in C^{2} ( Q_T )
    $, 
    and minimum point $(\hat{t},\hat{x})$ of $u - \phi$ in $(\frac{\hat{t}}{2}, \frac{3}{2} \hat{t} \wedge T) \times B_{r} ( \hat{x} )$,
    \begin{align*}
        \partial_t \phi ( \hat{t},\hat{x}) - \mathcal{L}_{r} ( \hat{x}, \phi, p ) 
        - \mathcal{L}^{r} ( \hat{x}, u, p ) 
        +  b ( \hat{x} ) \cdott p  + H ( \hat{x},p ) \geq f ( \hat{t},\hat{x} ),\quad p = D \phi (\hat{t}, \hat{x} ).
    \end{align*}
   
    \noindent (iii) $u$ is a viscosity solution if it is both   (viscosity)  subsolution and
    supersolution.
\end{definition}

We have the following   well-posedness and comparison   results for \eqref{eqn:parabolic_HJB}. This is a fairly standard result, so we only sketch the proof.

\begin{thm} \label{thm:well_posedness_HJB_parabolic}
Assume \ref{nu'}, \ref{B}--\ref{B2}, \ref{H1} --\ref{H3}, $f \in C_b (Q_T)$,
$u_{0} \in W^{1,\infty}( \R^d )$.

\medskip
    \noindent (a)
    There exists a unique viscosity solution 
    $u \in C_b ( \bar{Q}_{T} )$
    of \eqref{eqn:parabolic_HJB}.
\medskip

\noindent (b)
Let
    $u_1 \in \text{USC} ( \bar{Q}_{T}   ) \cap L^{\infty} (Q_T)$ 
    and
    $u_2 \in \text{LSC} ( \bar{Q}_{T}   ) \cap L^{\infty} (Q_T)$ 
    be
    viscosity sub- and supersolutions of \eqref{eqn:parabolic_HJB}
    with $f_{i} \in C_b ( Q_{T})$ and $u_{0,i}\in W^{1,\infty}(\R^d)$ for $i=1,2$. If either $u_1$ or  $u_2$ belongs to $L^\infty(0,T;W^{1,\infty}(\R^d))$, 
    then for all $ ( x,t ) \in \bar{Q}_{T} $,
    \begin{align*}
        u_1 ( x,t ) - u_2 ( x,t ) \leq \| (u_{0,1} - u_{0,2})^+ \|_{C_{b}} + t \| ( f_{1} - f_{2} )^{+}  \|_{C_{b} ( Q_{T}  )}  
    \end{align*}


\end{thm}

\begin{proof}
\noindent (a) The proof is close to the proof of \cite[Theorem 2.1]{nguyen2019large}, where 
the corresponding stationary HJB equation (with exponential growth assumptions on the data) is considered. The idea is to truncate the   data   
to get a problem covered by standard results \cite{JaKa05,JaKa06, barles2008second}, and then use the uniform Lipschitz estimate in  Proposition \ref{lipest0} below   and the Arzela-Ascoli theorem to pass to the limit. By stability of viscosity solutions \cite{JaKa05,barles2008second}, the limit will be a viscosity solution. Since our proof would be fairly similar, we omit it. Moreover, since any solution is Lipschitz by  Proposition \ref{lipest0}, it is 
unique by part (b). 
\medskip

\noindent (b) This is a standard variant of the comparison theorem in the viscosity theory for nonlocal equations, see e.g. \cite{JaKa06,JaKa05,barles2008second}. \nc We refer also to \cite[Proposition 1]{nguyen2019large}, where they give a detailed proof for stationary HJB equations and a sketch for the parabolic case,  both in our setting except for more growth on the data.  Only small modifications are needed, mainly due to a different penalization of infinity, which is given for us using the Lyapunov function $\langle x\rangle^k$ (cf. Lemma \ref{lemma:supersolution_property}). 
\end{proof}

A key result in our paper is the following global (in time) Lipschitz estimate.

\begin{proposition}[Global Lipschitz bound]\label{lipest0} Assume \ref{nu'}, \ref{B}--\ref{B2}, \ref{H1}, $f \in C_b (Q_T)$, and  let $u \in  C_b ( \overline Q_{T} )$  
be  a viscosity solution of \eqref{eqn:parabolic_HJB}. If $u_0 \in W^{1,\infty}(\R^d)$, then  
there is $C>0$ such that
\begin{align*}
        | u  ( t,x ) - u  ( t,y ) | 
        \leq C | x-y | \qquad\text{for}\,\,\, x,y \in \R^{d},\ t \in  [ 0,T ),
\end{align*}
where $C$ depends on $d, \sigma, \alpha, \beta, |b(0)|, C_H$ and on $[u_0]_1$ and $\sup_{[0,T]}[f(t)]_0$.
\end{proposition}

We stress that the constant $C$ above does not depend on $T$ other than through the global bound of $[f(t)]_0$ in $(0,T)$. This result is proved in Appendix \ref{app:Lipschitz}, see  Proposition \ref{lipest}. 
We  also need a local version of this Lipschitz result.

\begin{lemma}[Local Lipschitz bound]\label{thm:loc_lip_R}
Assume \ref{nu'}, \ref{H1},  $b\in C(\R^d)$, $f\in C(Q_T)$,   
$u_0\in W^{1,\infty}_{\textup{loc}}(\rd)$, and let $u\in  C_b ( \overline Q_{T} )$ be  a viscosity solution of \eqref{eqn:parabolic_HJB}. Then for $r >1 $, 
there exists $C_r\geq 0$ such that
\beq\label{lipB2}
       \sup_{t\in(0,T)} | u  ( t,x ) - u  ( t,y ) | \leq C_r(1+\omega_{r,T}(u)) | x-y | \qquad \hbox{for} \,\,\, x,y \in B_r:=\{|x|\leq r\},
\eeq
where $C_r$ depends on $d,\sigma,   \|f\|_{L^\infty([0,T]\times B_{2r})},   \|b\|_{L^\infty( B_{2r})} 
, C_H, Lip_{ B_{2r} }(u_0)$, and 
\begin{align}
\label{w_r}
\omega_{r,T}(u):= \sup_{\scriptsize\begin{array}{c}s\in[0,T],x,y \in  B_{r+2},\\  |x-y|< 1 \end{array}}\!\!\!\!  |u  (s,x) - u  (s,y)|\,.
\end{align}
\end{lemma}
The result is a corollary of Theorem \ref{thm:parabolic_local_Lipschitz_cutoff}\footnote{Just apply   Theorem \ref{thm:parabolic_local_Lipschitz_cutoff} with $\hat H(x,p):=H(x,p)+b(x)p   -f(t,x) $, $B=B_r$,   $\Omega=B_{2r}$, and recall that $r>1$.} in Appendix \ref{app:Lipschitz}. This local regularity result 
can be used  to show that solutions belong to $W^{2,\infty}_{\textup{loc}}(\rd)$ in space and hence are classical. 

\begin{thm} \label{thm:uniform_second_derivative_HJB} Assume $r > 0$,  
    \ref{nu'}, \ref{B}--\ref{B2}, \ref{H1}--\ref{H3}, 
    $f \in  C_b (Q_T)\cap L^\infty(0,T;W^{1,\infty}(\R^d))$, $u_{0} \in W^{2,\infty}(\R^d)$, 
    and let $u$ be the viscosity solution of \eqref{eqn:parabolic_HJB}.
\smallskip

\noindent (a) There is $C_r>0$, independent of $T$, such that for $t, t+s \in (0, T)$, $x \in B_r$, and $|s|,|y|\leq1$, 
\begin{align*}
  \frac{|u(t+s,x)-u(t,x)|}{|s|} + \frac{|u (t,x+y) - 2 u (t,x) + u (t,x-y) |}{|y|^2} \leq C_{r}.
\end{align*}

\noindent (b) $\partial_t u,\mathcal Lu\in C(Q_T)$ and $u$ is a classical solution of \eqref{eqn:parabolic_HJB}.
\end{thm}

The proof of Theorem \ref{thm:uniform_second_derivative_HJB} is given in the next subsection.

\subsection{A truncated smoothed 
problem and the proof of Theorem \ref{thm:uniform_second_derivative_HJB}.}
\  The plan is to  truncate the drift $b$  and smooth the problem, 
use  existence results of classical solutions for problems with bounded coefficients
from  \cite{ersland2020classical}, prove uniform $W^{2,\infty}_{\textup{loc}}$ bounds, and  use compactness and stability to pass to the limit and get Theorem \ref{thm:uniform_second_derivative_HJB}.  

Let $R  > 1 $,  $\vep\in(0,1)$,   and consider the truncated  smoothed   version of \eqref{eqn:parabolic_HJB},
%
\begin{align}
    \begin{cases}
        \partial_t u - \mathcal{L} u (x) +  b_{\vep ,R} (x) \cdott Du + H_{\vep} ( x,Du )  = f_{\vep} ( t, x ) , &\text{in } ( 0,T ) \times \R^d,  \\[0.2cm]
        u ( 0, \cdot ) = u_{0,\vep}, & \text{in } \R^d,  
    \end{cases}
    \label{eqn:parabolic_HJB_cutoff}
\end{align}
where  $b_{\vep,R } := b_\vep \cdot\chi_{R}$, $b_\vep:=b*\rho_\vep$, $H_\vep:=H*\hat \rho_\vep$, $f_\vep(t,x):=(f(t,\cdot)*\rho_\vep)(x)$, $u_{0,\vep}:=u_0*\rho_\vep$, $\rho_\vep:\R^d\to[0,\infty)$ and $\hat\rho_\vep=\rho_\vep\otimes\rho_\vep:\R^{2d}\to[0,\infty)$ are standard mollifiers,\footnote{\label{fn:mollify}The standard mollifier is defined as $\rho_\vep(x)=\frac1{\vep^d}\rho(\frac x\vep)$ for a function $0\leq\rho\in C_c^\infty(\R^d)$ which is radially symmetric with $\mathrm{supp}(\rho)\subset B(0,1)$ and  $\int \rho= 1$.}    and $\chi_{R} ( x )$  is a cut-off function defined by
\beq\label{chir}
\chi_R(x)=\chi\left(\frac{\langle x\rangle  
 }R\right)\,,\quad\chi \in C_{c}^{\infty}(\R)\,,\qquad \dschi_{B_{1}}(r) \leq \chi  ( r ) \leq \dschi_{B_{2}} ( r ), \qquad \|\chi\|_{C^1_b}\leq 2.
 \eeq
%
These functions are $C^\infty$ in $(x,p)$ with derivatives bounded uniformly in $x$, locally in $p$.

\begin{remark}\label{eps-indep-bnds}
(a) \ By the properties of mollifiers it follows that if $b$ satisfy \ref{B}--\ref{B2}, then so does $b_\vep$, with the same constants and bounds. Similarly, if $H$ satisfy \ref{H1}--
\ref{H3} then so does $H_\vep$, even with the same constants if the factor $(1+|p|)$ is replaced by $(1+|p|+|\vep|)\leq (2+|p|)$ in \ref{H1} and \ref{H3}. 
\smallskip

\noindent (b) \ Moreover,  $\text{supp}\, \chi_R  \subset   \{x: |x|\leq 2R\}$, $D\chi_R(x)=\frac1R\chi'(\frac{ \langle x\rangle }R)\frac x{ \langle x\rangle }$,  and $b_{\vep,R}\in C_b^\infty$.   
%
\end{remark}


 \begin{proposition}
 \label{thm:HJB_well_posedness}
 Assume   $\vep\in(0,1)$, $R>1$,   \ref{nu'}, \ref{B}--\ref{B2}, \ref{H1}--\ref{H3}, 
    $f \in C_b (Q_T)$, and $u_0 \in C_b (\rd)$. \medskip
%
    
    \noindent (a) 
    There exists a unique
 classical   
    solution $u_{ \vep , R}\in C_b^{ 1,3 }(Q_T)$ of
  \eqref{eqn:parabolic_HJB_cutoff}. 
Moreover,
  $$\|u_{ \vep , R}\|_\infty \leq \|u_0\|_\infty + C_0T\qquad \text{where}
  \qquad C_0 := \| H ( \cdot,0 )  \|_\infty + \| f  \|_\infty.$$

  \noindent (b) Assume in addition
 that $u_{0} \in W^{1,\infty}( \R^d )$.
  Then  for $\gamma 
  \in (1,\sigma)$,   there exist $C>0$ (independent of $ \vep,  R, T$)  and $c_T>0$ (independent of $ \vep,   R$)    such that
\begin{align*}
 |  u_{ \vep , R} ( t,x ) - u_{ \vep , R}  ( t,y ) |   \leq C | x-y | + \frac{c_T}{\langle R \rangle^{  \gamma }} ( \langle x \rangle^{  \gamma\nc} + \langle y \rangle^{
  \gamma } )\quad \text{for} \,\, x ,y \in B_{R},\ t \in  [ 0,T ).
    \end{align*}
  
  \noindent (c) 
  Assume in addition 
  $f \in  
  C_b(0,T;W^{1,\infty}(\R^d))$,  
  and $u_{0} \in W^{2,\infty}(\R^d)$.
 Then  
    for any $r<R-2$ 
    there exists $C_{r}>0$  (independent of $\vep, R,T$) 
    such that for all $x\in B_{r}$, $t \in  (0,T )$, 
    \begin{align}\label{d-bnds}
     | Du_{ \vep,  R} ( t,x ) | +   | D^2 u_{ \vep,  R} ( t,x ) | +|\partial_tu_{ \vep,  R} ( t,x )|
     \leq C_{r} + o_R (1).
    \end{align}   
Moreover, 
$C_r$ only depends on $f$ and $u_0$ through their $L^\infty(0,T;W^{1,\infty}(\R^d))$ and  $W^{2,\infty}(\R^d)$ norms respectively.  
 \end{proposition}
\nc 
 
\begin{remark}
(i) Note that the bound on $\|u_{ \vep,  R}\|_\infty$ 
is independent of $ \vep,   R$ but not on $T$.
\smallskip

\noindent (ii) We will use part (b) to control the oscillation of $u_{ \vep,  R}$ and to obtain a global Lipschitz bound in the limit as $R\to\infty$  and $\vep\to0$.  
\smallskip
%
\end{remark}

 \begin{proof}
 \noindent (a)  In view of the regularity and boundedness of the data and nonlinearity, the result follows from Theorem 5.3 and 5.5 in \cite{ersland2020classical}.  
 %
  \medskip
  
    \noindent (b)  The result is proved in Appendix \ref{app:Lipschitz}, see Theorem \ref{thm:parabolic_Lipschitz_cutoff}.  The bounds are independent of $\vep$ by Remark \ref{eps-indep-bnds}.  
 \medskip
  
    \noindent (c)  Let $ 0<r<R-2$. 
  We   
    bootstrap the estimates  in the following order.
\vskip0.4em

\noindent  Step 1 (First derivatives).  Using part (b), for $r+2<R$ we have 
\begin{align*}
    |u_{ \vep,  R} (t,x) - u_{ \vep,  R} (t,y)| \leq   \tilde C_r   |x-y| + o_R (1) \qquad\text{for}\,\,
x,y\in B_{r+2}, \ t\in(0,T),
\end{align*}
where   $\tilde C_r>0$ is independent of $ \vep $, $R$, and $T$.   It follows that $\omega_{r,T} (u_{ \vep,  R}) \leq  \tilde C_r  + o_R (1)$ 
(cf. \eqref{w_r}
), and then by Lemma \ref{thm:loc_lip_R}  and Remark \ref{eps-indep-bnds}   there is $C_{r} > 0$ independent of $ \vep ,T,R$ such that  
\beq\label{locgra}
    |  Du_{ \vep,  R} ( t,x ) | 
        \leq C_{r} + o_R (1) \qquad\text{for}\,\,   x \in B_r ,\ t \in (0,T).
\eeq

 \noindent Step 2 (Second derivatives).
Differentiating \eqref{eqn:parabolic_HJB_cutoff} with respect to $x_i$ we find that $w_{ \vep , R}:= \partial_{x_{i}} u_{ \vep,  R}$ is a solution of 
\begin{align}
    \partial_t w_{ \vep , R} - \mathcal{L} w_{\vep, R} +  b_{\vep,R} ( x ) \cdot Dw_{ \vep , R} + \tilde{H} ( x,Dw_{ \vep , R} ) = \tilde{f} ( t,x ),  
    \label{eqn:differentiated_HJB}
\end{align}
where
$\tilde{H} ( x,Dw_{ \vep , R} ) =   (D_p H_{\vep})  ( x,Du_{ \vep,  R}) \cdot Dw_{ \vep , R} $ 
and 
\begin{align}\label{tilde-f}
  \tilde{f}(t,x) = -\partial_{x_{i} } b_{  \vep,  R} ( x ) \cdot Du_{ \vep,  R} -   (\partial_{x_{i} }H_{\vep}) ( x,Du_{ \vep,  R} ) + \partial_{x_{i}} f_{ \vep} (t,x ).  
\end{align}

By \ref{H2} and Remark \ref{eps-indep-bnds}, $\tilde H$ satisfies \ref{H1} with   $C_{\tilde H}=L_H$. Moreover, by the derivative bound \rife{locgra}, we get a bound on $\omega_{r,T}(w_{ \vep , R})$. Using also \ref{H3}, \ref{B2}, the regularity of $f$,  and Remark \ref{eps-indep-bnds},  we also have 
$$
\tilde f\in C_b(Q_T)\qquad\text{and}\qquad |\tilde f(x,t)|\leq C_r+o_R(1)\qquad\text{for}\,\, x\in B_r,\ t\in(0,T),
$$ 
for some $C_r>0$ independent of $ \vep,  R,T$.
Hence, we can apply once more Lemma \ref{thm:loc_lip_R} to $w_{ \vep , R}$, and we find that  there is $C_{r} > 0$ (independent of $ \vep,  T,R$) such that
\begin{align*}
    |  \partial_{x_i}Du_{ \vep,  R} ( t,x ) |=|  Dw_{\vep, R} ( t,x ) | 
        \leq C_{r} + o_R (1) \qquad \text{for}\,\, x \in B_r,\ t \in (0,T).
\end{align*} 

\noindent Step 3 (Time derivative).  By  definition of $\mathcal{L}$ in \eqref{def:nonlocal_diffusion}, using \ref{nu'} and the estimates on $u_{ \vep,  R} $ and $Du_{ \vep,  R} $ obtained above,  for $x\in B_r$ we have 
\begin{align*}|
\mathcal{L} u_{ \vep,  R}  (t,x)|\leq \ & \frac12\sup_{y\in x+B_{1}}|D^2u_{ \vep,  R} (t,y)|\int_{|z|<1}|z|^2\;\nu(dz)\\
& +2\sup_{y\in x+B_{\ln R-|x|}}|Du_{ \vep,  R} (t,y)|\int_{1<|z|<\ln R-|x|}|z|\;\nu(dz) \\
&+ 2\int_{|z|>\ln R-|x|}(|u_{ \vep,  R} (t,x+z)|+|u_{ \vep,  R} (t,x)|) \;\nu(dz)\\
\leq & \ c\|D^2u_{ \vep,  R} (t,\cdot)\|_{\infty, B_{r+1}} + c\|Du_{ \vep,  R} (t,\cdot)\|_{\infty, B_{\ln R}} +  2 
  \|u_{\vep, R}\|_\infty
\int_{B_{\ln R}^C} 
\nu(dz) 
\end{align*}
Now, on the right-side the 3rd term is bounded by $o_R(1)$ by the $L^\infty$ control on $u_{\vep,R}$ and \ref{nu'}, 
the 2nd term is bounded by $C+o_R(1)$ by Lemma \ref{thm:loc_lip_R}
and  $\frac{\langle \ln R \rangle^{\gamma}}{\langle R \rangle^{\gamma}} =o_R(1)$  
(cf. Step 1), and the first term is bounded by $C_r+o_R(1)$ from Step 2. Therefore, overall we conclude that
$$
|\mathcal{L} u_{ \vep,  R} (t,x)|\leq C_r + o_R(1).
$$
Then by equation \eqref{eqn:parabolic_HJB_cutoff}, the definition of $b_{  \vep,  R})$, the assumptions on the data (including \ref{B}, \ref{B2}, \ref{H1}), and the bounds in Step 1 and 2, we get  
\begin{align*}
    |\partial_t u_{ \vep,  R} (t,x)| &\leq C_r (|\mathcal{L} u_{ \vep,  R}  (t,x)| +  | Du_{ \vep,  R} (t,x)|+1) \leq \tilde C_r + o_R(1)  \ \text{for}\,\, x \in B_r,\ t \in (0,T),
\end{align*}
where $C_r,\tilde C_r$ do not depend on $R,T$.

This concludes  the proof of  estimate \eqref{d-bnds}, where we observe that   the constant $C_r$  only depends on $f$ and $u_0$ through their $L^\infty(0,T;W^{1,\infty}(\R^d))$ and  $W^{2,\infty}(\R^d)$ norms respectively. The reason is that, roughly speaking,  we  need to apply the Lipschitz estimate  of part (b) and Lemma \ref{thm:loc_lip_R}  after at most one spatial differentiation of the equation. 
 \end{proof}

\begin{proof}[Proof of Theorem \ref{thm:uniform_second_derivative_HJB}]
    We first show that $u_{ \vep,  R}\to u$ locally uniformly. From Proposition  \ref{thm:HJB_well_posedness},  (a) and (b),  $u_{ \vep,  R}$ is bounded   in $Q_T$ \ee 
    uniformly in $R$, and is locally equi-Lipschitz. By  the Arzela-Ascoli theorem, $u_{ \vep,  R}$ converges locally uniformly along a subsequence to some limit $u^*$. Since $b_{  \vep,  R} (x) \to b (x)$ locally uniformly, stability of viscosity solutions under local uniform convergence 
then implies that $u^*$ is a viscosity solution of \eqref{eqn:parabolic_HJB}. 
Freezing $x,y$ and sending $R\to\infty$ in 
the (approximate) Lipschitz bound of Proposition  \ref{thm:HJB_well_posedness}   (b),   shows that $u^*$ is Lipschitz in $x$. Then by uniqueness of Lipschitz solutions of \eqref{eqn:parabolic_HJB} (Theorem \ref{thm:well_posedness_HJB_parabolic}), a standard argument shows that the whole sequence $u_{ \vep,  R}$ will converge (locally uniformly) to $u^*=u$, the unique viscosity solution of \eqref{eqn:parabolic_HJB}.

 For the rest of the proof we note that we may assume without loss of generality that $f \in W^{1,\infty}(Q_T)\cap L^\infty(0,T;W^{2,\infty}(\R^d))$ (or smoother if needed), and similarly for $u_0$. The general case follows by an approximation argument, $C_b(Q_T)$ compactness of approximate solutions, and passing to the limit in the corresponding pointwise bounds of Proposition \ref{thm:HJB_well_posedness} (c).  
\medskip

\noindent (a) Let $t \in (0, T)$, $x \in B(0,r)$, and $|y|\leq 1$. By   the fundamental theorem   and Proposition \ref{thm:HJB_well_posedness} (c) with $r+1$ instead of $r$, 
\begin{align*}
    | u_{R} (t, x+y ) &- 2 u_{ \vep,  R} (t, x ) + u_{ \vep,  R} (t, x-y ) | \\
    &= \bigg| \int_{0}^{1} Du_{ \vep,  R}(t, x + \tau y ) \cdott y \,d\tau - \int_{0}^{1} Du_{ \vep,  R}(t, x-\tau y ) \cdott y \,d\tau \bigg| \\
    &=   \bigg| \int_{0}^{1}\int_{0}^{1}   2\tau\, D^2u_{ \vep,  R} \big(t, x + \tau y -s (2\tau y)\big) y\cdot y  \,ds \,d\tau \bigg|\\
    &\leq   \big(C_{r+1} +o_R (1)\big)| y |^{2} .   
\end{align*}
Sending $R \to \infty$, we then get that
\begin{align*}
     |u (t, x+y ) - 2 u (t, x ) + u (t, x-y ) | \leq C_{r+1} |y|^2,
\end{align*}
for a constant  $C_{r+1}>0$ depending on $r$ but not $T$. The time regularity estimate follows in a similar way. 
\medskip

\noindent (b) 
By interpolation between Lipschitz and $W^{2,\infty}_{\textup{loc}}$ functions, it follows that $\mathcal Lu$ is (locally) $\beta$-H\"older continuous in $x$ for $\beta\in(0,2-\sigma)$, cf.  Proposition 4.18 in \cite{ChJaKr24} for a similar result. This estimate is uniform in $t\in(0,T)$. Since $u\in C_b(Q_T)\cap L^\infty(0,T;W^{2,\infty}_{\textup{loc}})$, 
it follows that $\mathcal L^r u\in C(Q_T)$ and $\mathcal L_r u = O(r^{2-\sigma})$, and hence $\mathcal L u\in C(Q_T)$\footnote{To estimate $|\mathcal L u(x)-\mathcal L u(y)|$, use first $\mathcal L=\mathcal L_r +\mathcal L^r$ with $r$ small and then take $|x-y|$ small.}. Similarly, using a difference approximation, 
we can show that $Du\in C(Q_T)$.

To show that $\partial_t u \in C(Q_T)$, we first show that $u$ is differentiable for a.e. $(t,x)$ and satisfy equation \eqref{eqn:parabolic_HJB} in every point of differentiability. By the continuity of $\mathcal Lu, Du$, this means that $\partial_t u $ equals a continuous function a.e. and hence is continuous. 
This is (relatively) classical and explained in detail in  \cite[Theorem 4.23]{ChJaKr24} for a similar case. Here we just comment on some of the steps: (i) a.e. $t$-differentiability follows from the Rademacher theorem since $u$ is locally Lipschitz. (ii)  differentiability at a point implies that there is a test function $\phi$ such that at this point:  $u-\phi$ has a local strict max (or min),  $\partial_t\phi=\partial_t u $, and $D\phi=Du$.
This test function can be modified outside a neighborhood of the max (or min) point to be bounded and have global optimum.
The viscosity sub- and supersolution inequalities (see Definition \ref{def:visc_soln_parabolic}) then imply that equation \eqref{eqn:parabolic_HJB} holds pointwisely at this point after sending $r\to 0$.
\end{proof}

\nc
\subsection{The stationary HJB equation}

The stationary ergodic HJB equation is given by
\begin{align}
    \lambda - \mathcal{L} \bar{u}  (x) + H ( x,D \bar{u}  ) + b ( x ) \cdott D \bar{u} = f ( x ) \qquad \text{ in }  \R^d
    \label{eqn:ergodic_HJB}
\end{align}
where $\lambda \in \R$ is the so-called ergodic constant. We have existence and uniqueness (up to  addition of  constants) of viscosity solutions of this problem.

\begin{thm}\label{thm:well_posedness_HJB_ergodic}
 Assume \ref{nu'}, \ref{B}--\ref{B2}, \ref{H1},   
 and $f  \in W^{1,\infty} ( \R^d  )$. \smallskip
 
\noindent (a)
There exists a viscosity solution 
$ ( \lambda, \bar{u}  ) \in \R \times   C_b ( \R^d  )$ of \eqref{eqn:ergodic_HJB}. Moreover, there is $C>0$ independent of $\lambda$ such that 
 \begin{align*}
     |   \bar{u} ( x ) - \bar{u} ( y ) | \leq C | x-y |\qquad \text{for}\qquad x,y\in\R^d.    
\end{align*}

\noindent (b) Let
$ ( \lambda_{1}, u_{1} ), ( \lambda_{2}, u_{2} ) \in \R \times C_{b} ( \R^d  )$
be viscosity sub and supersolutions of \eqref{eqn:ergodic_HJB}. Then $\lambda_{1} \leq \lambda_{2}$, and if $\lambda_{1} = \lambda_{2}$, then $u_{1} - u_{2} = C$ in $\R^d $ for some constant $C \in \R$.
\end{thm}




\begin{proof}   
    (a) We refer to \cite[Theorem 3.1]{nguyen2019large}, where
    the solution $( \lambda, \bar{u})$ is constructed as the limit of $(\delta u_\delta, u_\delta-u_\delta(0))$ as $\delta \to 0$, where $u_\delta$ is the solution of the 
    stationary HJB equation 
    \beq
    \delta u^{\delta} - \mathcal{L} u^{\delta} (x) + b ( x ) \cdott Du^{\delta}  + H ( x, Du^{\delta} )  = f ( x ), \quad   x \in  \R^d.
  \label{delta_eq}
\eeq
    The solution $\bar{u}$ then inherits the $\delta$-independent Lipschitz bound on $u^{\delta}$, which can be proved as in  the time-dependent case considered in   Proposition  \ref{lipest0} (cf. \ee Proposition \ref{lipest} in Appendix \ref{app:Lipschitz}). In fact, time-independent bounds in the parabolic case translate  into $\delta$-independent bounds for  the solution of \rife{delta_eq}.
    Another very similar proof, this time for the ergodic stationary case, is given in \cite[Theorem 2.1 (ii)]{chasseigne2019priori}. We therefore omit the details.
    \medskip
    
    \noindent (b) The proof is essentially the same as the proof of \cite[Theorem 3.2]{nguyen2019large}.
\end{proof}

Since the parabolic regularity results of Theorem \ref{thm:uniform_second_derivative_HJB} (a) are independent of $T$, it follows that the  stationary ergodic problem has classical solutions.
\begin{thm}\label{thm:ergodic_double_derivative} 
Assume $r > 0$,  
    \ref{nu'}, \ref{B}--\ref{B2}, \ref{H1}--\ref{H3}, 
    $f \in   W^{1,\infty}(\R^d)$, and $(\lambda,\bar u)$ is a viscosity solution of \eqref{eqn:ergodic_HJB}.
\medskip
    
\noindent (a) There is a constant $C$ such that for $x,y \in \rd$,  
\begin{align*}
    |\bar u (x) - \bar u (y)| \leq C |x-y|.
\end{align*}

\noindent (b) There is a constant $C_r$  such that for $x \in B(0,r)$, and $|y|\leq1$, 
\begin{align*}
   \frac{|\bar u (t,x+y) - 2 \bar u (t,x) + \bar u (t,x-y) |}{|y|^2} \leq C_{r} .
\end{align*}

\noindent (c) $\mathcal L\bar u,D\bar u\in C(\rd)$, and $(\lambda,\bar u)$ is classical solution of  \eqref{eqn:ergodic_HJB}.
\end{thm}

\begin{proof}
The proofs of (a) and (b) are similar. We only show (b). Let  $u_{0} \in W^{2,\infty} (\R^d )$ and $u$ be the solution of \eqref{eqn:parabolic_HJB}.  By \cite[Theorem 3.4]{nguyen2019large} 
there exists a constant $a \in \mathbb{R}$, such that 
\begin{align*}
   \lim_{t \to \infty} \max_{B(0,R)} |u (t,z) - (ct + \bar{u} (z) +a)| = 0 \quad \text{for all}\quad R >0.
\end{align*}
Thus  we get
\begin{align*}
    &|\bar u (x+y) - 2 \bar u (x) + \bar u (x-y) | \leq
    |u (t,x+y) - 2 u (t,x) + u (t,x-y) |\\ 
    &\quad + |u (t,x+y) - (ct + \bar u (x+y) +a)|+ 2|u (t,x) - (ct + \bar u (x) +a)| \\
    &\qquad+ |u (t,x-y) - (ct + \bar u (x-y) +a)|.
\end{align*}
The result now follows from Theorem \ref{thm:uniform_second_derivative_HJB} (a)   passing to the limit as $t \to \infty$.
\medskip

\noindent (c) The proof is similar to the proof of Theorem \ref{thm:uniform_second_derivative_HJB} (b).
\end{proof}


\section{Well-posedness and approximation of fractional Fokker-Planck equations}\label{FP-sec}













Let us collect from  \cite{porretta2024decay}  some result  on the 
fractional Fokker-Planck (FFP) equation with confining drift. 
The time-dependent equation is given by
\begin{align}
    \begin{cases}
        \partial_{t} m - \mathcal{L}^*  m (x) - \text{div} ( m {\mathcal B}  (t,x) ) = 0  & \text{in} \quad (0,T) \times \R^d,\\[0.2cm]
    m ( 0 ) = m_{0} & \text{in} \quad \R^d\,,
    \end{cases}
    \label{eqn:FFP}
\end{align}
where $\mathcal{L}$ is the nonlocal diffusion operator defined in \eqref{def:nonlocal_diffusion} and $\mathcal B(t,x)$ is a vector field in $Q_T$ which is continuous and satisfies
\beq\label{bdiss}
 {\mathcal B}(t,x)\cdott x \geq  \alpha_0 \, |x |^{2}   \qquad \forall x  \in \R^d:\, |x|\geq R\,,\, \forall t>0
\eeq
\beq\label{disso}
 ({\mathcal B}(t,x)-{\mathcal B}(t,y))\cdott (x-y) \geq  - \beta_0 |x-y|    \qquad \forall x,y  \in \R^d\,, \,\forall t>0\,
 \eeq
for some $\alpha_0, \beta_0, R>0$. Let us notice that the Fokker-Planck equation in the MFG system \rife{eqn:parabolic_MFG} corresponds to the choice ${\mathcal B}(t,x)= b(x)+ D_pH(x,Du(t,x))$, which satisfies \rife{bdiss}-\rife{disso} due to \ref{B}, \ref{B2}, \ref{H1}, \ref{H2}, and the  $C^1_b$ regularity of $u$.

\vskip1em

Solutions of \eqref{eqn:FFP} will be defined as measure valued solutions tested against viscosity solutions of the dual problem
\begin{align}\label{fpdual}
\begin{cases}
- \partial_s \phi - \mathcal L \phi + {\mathcal B}(x,s)\cdott D\phi = f & \text{in} \quad (0,t) \times \R^d,\\[0.2cm]
\phi(t,x) = \xi(x) & \text{in} \quad \R^d.
\end{cases}
\end{align}
We denote by $C([0, T ); \cM (\rd)^*)$   the functions   from $[0,T]$ into $\cM (\rd)$  which are continuous with respect to the weak-$*$ topology of measures.

\begin{definition}\label{def:dsol}
Let $m_0 \in \mathcal M (\rd)$. A function $m \in C([0, T ); \cM (\rd)^*)$ is a solution of \eqref{eqn:FFP} if
\begin{align}\label{d-soln}
&\int_{\R^d}\xi dm(t) + \int_0^t \int_{\R^d} f dm(s) ds = \int_{\R^d} \phi(0, x) dm_0,
\end{align}
for all $t \in (0, T )$, $\xi \in C_b(\R^d)$, $f \in C_b(Q_t)$,  and viscosity solutions $\phi \in C_b(\overline Q_t)  $ of \eqref{fpdual}.
\end{definition}

\begin{remark}
Since classical solutions are viscosity solutions, any $m$ satisfying Definition \ref{def:dsol} is also a very weak solution in the distributional sense, 
    \begin{align}\label{vwsol}
        &\int_{\rd} \phi (t,x) dm(t)   = \int_{\rd} \phi (0,x) \,dm_0   \\ 
        &\qquad+ 
        \int_0^t \int_{\rd} \Big(\partial_t\phi  + \mathcal{L} \phi(x) - {\mathcal B}\cdott D\phi (x,r)\Big) dm(r)\quad \forall \phi\in  C^2_c 
        (\overline Q_t). \nonumber
    \end{align}
 \end{remark}

Existence and uniqueness of solutions follow as a special case of Theorem 2.6 in \cite{porretta2024decay}.

\begin{proposition}\label{ex-FP}
Assume $0<k<\sigma$, \ref{nu'}, \rife{bdiss}, \rife{disso}, $m_0\in \mathcal P_k(\R^d)$. Then there exists a unique solution $m \in C([0, T ); \mathcal P_k (\rd))$ of \eqref{eqn:FFP} according to Definition \ref{def:dsol}.
\end{proposition}

Let us point out that, once $m\in C([0, T ); \mathcal P_k (\rd))$, then it is possible to use unbounded functions $\xi, f$ in Definition \ref{def:dsol}, precisely one can take continuous $\xi, f$ such that $\xi\in L^\infty(\langle x\rangle^{-k}), f\in L^\infty((0,t);L^\infty(\langle x\rangle^{-k}) )$ (cf. \cite{porretta2024decay}).

For a given time-independent drift ${\mathcal B}(x)$, the corresponding stationary equation is given by
\begin{align}
   -\mathcal{L}^*  \bar{m} (x) - \text{div} ( \bar{m} {\mathcal B} (x) ) = 0.
    \label{eqn:FFP_stationary}
\end{align}
Solutions of this problem are defined by duality in a similar way as in Definition \ref{def:dsol}, and existence 
is given by Theorem 5.7 in \cite{porretta2024decay}, along with exponential convergence $m(t)\to \bar m$ as $t\to \infty$. 

\begin{remark}\label{form_FP-stat}  A  measure $\bar m$ can be defined to be a solution  of  \rife{eqn:FFP_stationary} if it is a stationary solution, in the sense of Definition \ref{def:dsol}, of problem \rife{eqn:FFP} with $m_0=\bar m$. It is however equivalent to give a duality formulation with only space-dependent test functions, namely to say that $\bar m$ solves \rife{eqn:FFP_stationary} if and only if 
\beq\label{duality-stat}
\int_{\R^d} f\, d\bar m= c \qquad \forall (f,c)\,:\, \,\, \exists \,\phi\,\, \hbox{viscosity sol. of}\,\, c- \mathcal L \phi + {\mathcal B}(x)\cdott D\phi = f\,.
\eeq
In fact,  \rife{duality-stat} easily follows from Definition \ref{def:dsol}, by using  $\phi (x) + c(t-s)$ in  \rife{fpdual}, dividing  \eqref{d-soln} by $t$, and sending $t\to\infty$. Conversely, if $\phi$ solves \rife{fpdual} with ${\mathcal B}$ time independent, integrating we have that $\int_0^t \phi(x,s)ds$ is a solution of a stationary problem with right hand side $\int_0^t f(x)ds+ \phi(x,t)-\phi(x,0)$, so that \rife{duality-stat} implies
$$
\int_{\R^d}\left\{ \int_0^t f(x)ds+ \phi(x,t)-\phi(x,0)\right\} d\bar m = 0 
$$
which means that $\bar m$ satisfies Definition   \ref{def:dsol}.
\end{remark}

\begin{proposition}\label{ex-FP-stat}
Assume $0<k<\sigma$, \ref{nu'},  and ${\mathcal B}$ independent of $t$ satisfying \rife{bdiss}--\rife{disso}. 
Then there exists a unique 
solution $\bar m \in \mathcal P_k (\rd)$ of \eqref{eqn:FFP_stationary} according to Definition \ref{def:dsol}.
\end{proposition}
\begin{proof}
Note that the assumptions of Theorem 5.7 in \cite{porretta2024decay} are satisfied (we are in the case where the Lyapunov function $\langle x\rangle^k$ satisfies the condition  for exponential convergence to the stationary problem, due to Lemma \ref{lemma:supersolution_property}).
Starting from  some $m_0\in \mathcal P_k(\R^d)$, the unique solutions $m(t)$ of \eqref{eqn:FFP} converge in $\cP_k(\R^d)$ to a unique limit $\bar m$ as $t\to\infty$. 
\end{proof}


In the next section we will prove compactness of solutions of a family of truncated  and smoothed  MFGs. To do this, we need a compactness result for the family of solutions of a truncated version of the Fokker-Planck equation \eqref{eqn:FFP}, namely
\begin{align}
    \begin{cases}
        \partial_{t} m_{\vep,R} - \mathcal{L}^*  m_{\vep,R} (x) - \text{div} ( m_{\vep,R} \,{\mathcal B}_{\vep,R} (t,x)) = 0, \\[0.2cm]
    m ( 0 ) = m_{0,\vep},
    \end{cases}
    \label{eqn:FFPt}
\end{align}
where ${\mathcal B}_{\vep,R}$ satisfies
\beq\label{tb}
 \begin{split}
 & {\mathcal B}_{\vep,R} := b_{1,\vep,R}+b_{2,\vep,R}\,,\quad b_{1,\vep,R}, b_{2,\vep,R}\in C_b(Q_T)  \\[0.2cm]
 & \, b_{1, \vep,R}:=b_{1,\vep}
 \cdott \chi_{R}\,,  \ \ \|b_{2,\vep,R}\|_{\infty}\leq C, \ \ \text{and}\ \ b_{2,\vep,R}  
 \underset{\substack{R\to \infty\\ \vep\to0}}\to
   b_2 \ \text{locally uniformly}, 
\end{split}
\eeq
for some continuous functions $b_1, b_2$, where $\chi_R$  is defined in \eqref{chir},  $m_{0,\vep}=(m_0)*\rho_\vep$, $b_{1,\vep}= (b_1)*\rho_\vep$, $\rho_\vep$ is defined in footnote \ref{fn:mollify}, and $C$ is independent of $R$. 

In the next section we will take $b_1=b$ and $b_{2,\vep,R}=D_p H_\vep (x,Du_{\vep, R}(t,x))$, for $u_{\vep, R}$ solution of a corresponding truncated HJB.
In this case we easily have  $b_{2,\vep,R}$  bounded by \ref{H2}, while $b_1$ satisfies \ref{B}. The compactness result follows from an   argument similar to  Lemma 6.3 in \cite{porretta2024decay}, specialized to our setting. 

\begin{lemma} 
\label{lem:comp}
Assume  \ref{nu'},  and that  \rife{tb} holds for some vector field $b_1$ satisfying condition \ref{B}.  
Let $m_0\in \mathcal P_k(\rd)$ for some $0<k<\sigma$, and $m_R\in C([0,T];\mathcal P_k(\rd))$
be  a solution of \eqref{eqn:FFPt} according to Definition \ref{def:dsol}.  Then we have:
\medskip

\noindent (a) There is $C_T>0$, independent of $\vep\in(0,1)$ and $R>0$, such that
\beq\label{momkt}
    \int_{\R^d }   \langle x \rangle^{k} dm_{\vep,R}(t)  \leq C_T\int_{\R^d }   \langle x \rangle^{k}  dm_{0} \qquad\forall t\in(0,T).
\eeq

\noindent (b) For every $\eta>0$  there is $\delta>0$ (independent of $\vep\in(0,1)$ and $R>0$) such that
$$
d_0(m_{\vep,R}(t),m_{\vep,R}(s)) \leq  \eta, \qquad \forall t,s\in(0,T),\quad |t-s|<\delta.
$$

\noindent (c) There is  $m\in C([0,T];\cP_k(\R^d))$ such that $m_{\vep,R}\to m$ in $C([0,T];\cP_k(\R^d))$. Moreover, $m$ is the unique solution of \eqref{eqn:FFP} with ${\mathcal B}=b_1+b_2$.
\end{lemma}

\begin{remark}\label{rem:unifbnd}
According to the proof below, the constants $\delta$ and $C_T$ above only depend  on  $k$, $\|m_0\|_{\cP_k}$,  $c_{\sigma,d}$, $\beta$ (from \ref{nu'},  \ref{B}), local bounds of $b_1$, and the constant $C$ in \rife{tb}.
\medskip
\end{remark}

\begin{proof}
(a) \ In Definition \ref{def:dsol}, written for $m_{\vep,R}$, we take $\phi(s,x)= e^{\lambda(t-s)} \langle x\rangle^k$ obtaining
\begin{align*}
 & \int_{\R^d }  \langle x \rangle^{k} dm_{\vep,R}(t)  - e^{\lambda t}\int_{\R^d }   \langle x \rangle^{k}  dm_{0, \vep}  = - \int_0^t \int_{\R^d } [- \partial_s \phi - \mathcal L \phi + {\mathcal B}_{\vep,R} (x,s)\cdott D\phi ] dm_{\vep,R}
\\ & \qquad  = - \int_0^t \int_{\R^d } e^{\lambda(t-s) }[ \lambda  \langle x\rangle^k  - \mathcal L  \langle x\rangle^k + k  \langle x\rangle^{k-2} {\mathcal B}_{\vep,R} (x,s)\cdott x ] dm_{\vep,R}
\,.
\end{align*}
Using \rife{tb} we have
\begin{align*}
{\mathcal B}_{\vep,R} (x,s)\cdott x & = b_{1, \vep}(x) \cdott x\, \chi_R(x) +b_{2, \vep, R}\cdott x  
\geq  - (|b_1(0)|+   \beta)  |x| - C\, |x|
\end{align*}
on account of condition  \ref{B} satisfied by $b_1$. Hence we get
\beq\label{momt}
\begin{split}
&\int_{\R^d }   \langle x \rangle^{k} dm_{\vep,R}(t)- e^{\lambda t}\int_{\R^d }    \langle x \rangle^{k}  dm_{0, \vep}\\
& \leq - \int_0^t \int_{\R^d } e^{\lambda(t-s) }[ \lambda  \langle x\rangle^k  - \mathcal L  \langle x\rangle^k - k  \langle x\rangle^{k-1} (|b_1(0)|+   \beta+ C)  ] dm_{\vep,R} 
\,.
\end{split}
\eeq
Since $\mathcal L  \langle x\rangle^k \leq c  \langle x\rangle^{k-1}$ (for some $c$ depending on $\nu, k, \sigma, d$, see also Lemma \ref{lemma:supersolution_property}), 
we have
\begin{align*}
 & \int_0^t \int_{\R^d } e^{\lambda(t-s) }\Big[ \lambda  \langle x\rangle^k  - \mathcal L  \langle x\rangle^k - k  \langle x\rangle^{k-1} (|b_1(0)|+   \beta+ C)  \Big] dm_{\vep,R} \\ & \qquad \geq 
\int_0^t \int_{\R^d } e^{\lambda(t-s) }\Big[ \lambda  \langle x\rangle^k  - \  \langle x\rangle^{k-1} \left[c + k (|b_1(0)|+   \beta+ C) \right]  \Big] dm_{\vep,R}
\end{align*}
and we can choose $\lambda $ sufficiently large (depending on $k, |b_1(0)|, \beta, C, \nu, \sigma, d$) so that this integral is positive. 
Since $m_{0, \vep}= (m_0)*\rho_\vep$, statement (a) then follows from \rife{momt}. \smallskip

\noindent (b) \
Once we have the uniform estimate on the moments of $m_{\vep,R}(t)$, part (b) follows exactly as in Lemma 6.3 of \cite{porretta2024decay}, where the local compactness only uses local bounds of ${\mathcal B}_{\vep,R}$.  This means that $\delta $ depends on $\eta$ only through the constant $C_T$ of \rife{momkt},  $\|m_{0}\|_{\cP_k}$,  and the bound of ${\mathcal B}_{\vep,R}$ on a compact set $K= K( \eta, C_T, \|m_0\|_{\cP_k})$. 

%
\smallskip

\noindent (c) \ Part (c) then follows from the Arzela-Ascoli theorem since the sequence is equicontinuous by (b) and pointwise precompact by (a) and Prokhorov's theorem. This type of convergence is enough to pass to the limit and show that any limit point is a solution of \eqref{eqn:FFP} (see below). Since \eqref{eqn:FFP} has a unique solution (\cite[Thm 2.6]{porretta2024decay}), this then implies full convergence of the sequence.

We are only left to check that the limit $m$ is a solution according to Definition \ref{def:dsol}. To this purpose,   take $f\in C_b(Q_t), \xi \in C_b(\R^d)$,  and consider that since ${\mathcal B}_{\vep,R}$ is bounded and continuous,  by standard results
%
(see e.g. \cite[Theorem 5.1]{ersland2020classical}),  there is a unique viscosity solution of
\begin{align}\label{fpdualR}
\begin{cases}
- \partial_t \phi_{\vep,R} - \mathcal L \phi_{\vep,R} + {\mathcal B}_{\vep,R}(s,x)D\phi_{\vep,R} = f & \text{in} \quad (0,t) \times \R^d,\\[0.2cm]
\phi_{\vep,R}(t,x) = \xi(x) & \text{in} \quad \R^d
\end{cases}
\end{align}
which satisfies  $\|\phi_{\vep,R}\|_\infty \leq \|\xi\|_\infty +  \|f\|_\infty\,  t$.  Applying Proposition \ref{regeff}, we have that the sequence $\{\phi_{\vep,R}\}_{\vep,R>0}$ is locally equi-continuous in $x$. 
%
%
By Lemma 6.2 in \cite{porretta2024decay}, applied to $ \phi_{\vep,R}( \cdott-s)$ (for all fixed $s\in(0,t)$), we deduce that $\phi_{\vep,R}$ is also equi-continuous in time (locally in space). This is actually true in the whole time interval $(0,t)$,  since $\xi$ is uniformly continuous on any compact set; namely, by  \cite[Lemma 6.2]{porretta2024decay}, for any $\eta>0$ and any compact set $K\subset \R^d$ there is $\delta>0$ such 
that $\|\phi_{\vep,R}(t-\tau)-\xi\|_{L^\infty(K)}\leq \eta$ for all $\tau<\de$, where $\de$ is independent of   $\vep\in(0,1)$ and $R>0$.  
Therefore, by the Arzela-Ascoli theorem, $\{\phi_{\vep,R}\}_{\vep,R>0}$ is (locally in space) compact in the uniform topology and converges uniformly, up to a subsequence, to some function $\phi\in C_b(Q_t)$ such that $\phi(t)=\xi$.  

By the local uniform convergence of ${\mathcal B}_{\vep,R}$ and $\phi_{\vep,R}$, and the stability of the notion of viscosity solutions, we can then pass to the limit in \eqref{fpdualR} and conclude that any limit point $\phi$ is a viscosity solution of \eqref{fpdual}. But problem \eqref{fpdual} has a unique bounded viscosity solution, which therefore coincides with the limit of $\phi_{\vep,R}$.  
%

The next step is to note that, since $m_R$ solves \eqref{eqn:FFPt} according to Definition \ref{def:dsol}, 
\begin{align*}
&\int_{\R^d}\xi dm_{\vep,R}(t) + \int_0^t \int_{\R^d} f dm_{\vep,R}(s) ds = \int_{\R^d} \phi_{\vep,R}(0, x) dm_{0,\vep}.
\end{align*}
In view of the convergence properties of $m_{\vep,R}$ and $\phi_{\vep,R}$, we send $R\to\infty$ and conclude that 
\begin{align*}
&\int_{\R^d}\xi dm(t) + \int_0^t \int_{\R^d} f dm(s) ds = \int_{\R^d} \phi(0, x) dm_0.
\end{align*}
This shows that $m$ is a solution of \eqref{eqn:FFP} according to Definition \ref{def:dsol}.
\end{proof}

\section{Well-posedness for 
time-dependent and 
stationary fractional MFGs}
\label{sec:evol}

In this section we will prove the well-posedness of the evolutionary and stationary MFGs \eqref{eqn:parabolic_MFG} and \eqref{eqn:ergodic_MFG} with unbounded confining drifts $b$. The strategy is to use known results for MFGs with bounded drift from \cite{ersland2020classical}, a truncation procedure for $b$, compactness arguments for the truncated equations, and stability results to conclude that the limits are solutions of the sought after MFGs.

\subsection{Well-posedness for 
 time-dependent MFG systems}

Let us begin by making precise the notion  of solution to \eqref{eqn:parabolic_MFG}.

\begin{definition}\label{def:MFG_soln} Assume \ref{nu'}, \ref{B}--\ref{B2}, \ref{H1}--\ref{H3}, and 
\ref{G}. Let $u_T\in W^{1,\infty}(\R^d)$, and $m_0 \in   L^1(\R^d)\cap \cP_k(\R^d)$, for some $0<k<\sigma$.  

\noindent A couple $(u,m)\in C_b(\overline Q_T)\times C([0,T];\cP_k(\R^d))$ is a solution of \eqref{eqn:parabolic_MFG} if  
\vskip0.2em
(i) $u$ is a  viscosity solution (in the sense of Definition \ref{def:visc_soln_parabolic}) of the problem 
\begin{align}
    \begin{cases}
        - \partial_t u - \mathcal{L} u (x) + H ( x,Du ) +  b(x) \cdott Du =  F(x,m(t)) , &\text{in } ( 0,T ) \times \R^d,  \\
        u ( T, \cdot ) = u_T, & \text{in } \R^d,  
    \end{cases}
\label{mfg_HJB}
\end{align}
\vskip0.4em
(ii) $m$ is  a solution (in the sense of Definition \ref{def:dsol}) of the problem
\begin{align}
    \begin{cases}
        \partial_t m - \mathcal{L}^* m  -\dive(m\, (b(x) + D_pH(x,Du)) )  = 0 , &\text{in } ( 0,T ) \times \R^d,  \\
       m ( 0, \cdot ) = m_{0}, & \text{in } \R^d\,.  
    \end{cases}
    \label{mfg_FP}
\end{align}
\end{definition}

Let us stress that the above definition entirely relies on the results of the previous sections. Indeed, due to \ref{G}, the term $F(x,m(t))$ belongs to $C_b(Q_T) \cap L^\infty((0,T); W^{1,\infty}(\R^d))$, so $u$ is the unique solution of the equation \rife{mfg_HJB}, and $u$ is $C^1$ in the space variable due to Theorem \ref{thm:uniform_second_derivative_HJB}. In particular, $D_pH(x,Du)$ is well defined and problem \rife{mfg_FP} falls under the conditions of Proposition \ref{ex-FP}, hence $m$ is the unique solution of this equation.
\vskip0.5em

We now prove the existence of a solution $(u,m)$ of the MFG system. To this purpose, we proceed by approximation. 
For $\vep, R>0$, let $b_{\vep,R}$ be  defined as in \rife{eqn:parabolic_HJB_cutoff}, i.e. $b_{\vep, R} ( x ) := b_\vep ( x ) \chi_{R} ( x )$ where $\chi_{R}$ is the cut-off function defined in \eqref{chir} and  $b_\vep=b*\rho_\vep$ (see footnote \ref{fn:mollify}).  Since $\text{supp}\, \chi_R = \{x: |x|\leq 2R\}$  and $|D\chi_R|\leq \frac2R$, $b_{\vep, R}$ is bounded and Lipschitz.  We also define
$$F_{\vep}(x,m_{\vep, R} ( t )) = F ( \cdot, m_{\vep, R} ( t ) ) *\rho_\vep(x),$$
and  
consider the truncated and  smoothed  MFG system
\begin{align}
  \begin{cases}
        -\partial_t u_{ \vep,  R} - \mathcal{L} u_{ \vep,  R} (x) +  H_\vep  ( x,Du_{ \vep,  R} ) +  b_{ \vep,  R}(x) \cdott Du_{ \vep,  R}   =  F_{ \vep} ( x, m_{\vep, R} ( t ) )  
        & \text{ in } ( 0,T ) \times \R^d, \\[0.2cm]
        \partial_t m_{ \vep,  R} - \mathcal{L}^* m_{ \vep,  R} (x) - div \Big( m_{ \vep,  R} \big( b_{ \vep,  R} ( x )  +D_p  H_\vep  ( x, Du_{ \vep,  R}  )\big)  \Big) = 0 & \text{ in } ( 0,T ) \times \R^d, \\[0.2cm]
      m_{ \vep,  R} ( 0,x ) = m_{0,\vep} (x),\qquad u_{ \vep,  R} ( x,T ) = u_{T,\vep}. &
  \end{cases}
  \label{eqn:parabolic_MFG_truncated}
\end{align}
We will prove existence of solutions $(u_{ \vep,  R},m_{ \vep,  R})$ of  this problem, compactness for the sequence of solutions $\{(u_{ \vep,  R},m_{ \vep,  R})\}_{\vep,R>0}$, and then pass to the limit as $R\to \infty$, $\vep\to 0$ to show that any limit point $(u,m)$ solves \eqref{eqn:parabolic_MFG}.

\begin{thm}\label{ex-MFG}
Assume
\ref{nu'}, \ref{B}--\ref{B2}, \ref{H1}--\ref{H5}, \ref{G}. Let $u_T\in W^{1,\infty}(\R^d)$, and $m_0 \in   L^1(\R^d)\cap \cP_k(\R^d)$, for some $0<k<\sigma$. 
Then there exists  a solution $(u,m)$ of \rife{eqn:parabolic_MFG}. In addition, if we assume also  \ref{F1},  then  $(u,m)$ is the unique solution of \rife{eqn:parabolic_MFG}.
\end{thm}

\begin{proof}
The proof can be summarized in the following  steps. \\[0.2cm]
{\it Step 1.}\quad Since 
$b_{ \vep,  R}(x)$, $F_{\vep} ( x, m_{\vep, R} ( t ) )$, and the data are bounded and smooth satisfying \ref{B}, \ref{B2}, and \ref{G} (but not necessarily \ref{F1}), as a consequence of the results in \cite{ersland2020classical} the truncated smoothed  MFG problem \rife{eqn:parabolic_MFG_truncated} admits a (classical,  not necessarily unique) solution $(u_{ \vep,  R}, m_{ \vep,  R})$. 
We then use Proposition \ref{thm:HJB_well_posedness} to get estimates for $u_{ \vep,  R}, Du_{ \vep,  R}$; in particular, due to \ref{H1} and \ref{G}, we have that $u_{ \vep,  R}$ is uniformly bounded in $Q_T$ and satisfies
$$
 |  u_{ \vep,  R}( t,x ) - u_{ \vep,  R} ( t,y ) |   \leq C | x-y | + \frac{c_T}{\langle R \rangle^{\gamma}} ( \langle x \rangle^{\gamma} + \langle y \rangle^{\gamma} )\quad \text{for} \quad x ,y \in B_{R},\ t \in  [ 0,T ).
 $$
In addition, using the local Lipschitz bound of $F(x,m_{\vep,R})$ (from \ref{G}) and $u_T$ we have that $Du_{ \vep,  R}$ is locally uniformly bounded in  $Q_T$. By a standard argument in viscosity solutions' theory (see e.g. Lemma 4 in \cite{BBL}), the local Lipschitz bound in the $x$ variable also implies local equi-continuity in the time variable, and eventually the local compactness of $u_{ \vep,  R}(t,x)$ in the  uniform topology, namely $u_{ \vep,  R}$ is relatively compact in $C_b([0,T]\times U)$ for all compact subsets $U\subset \rd$.  
By a usual diagonal argument, we can construct a subsequence (not relabeled), such that  $u_{ \vep,  R}\to u$ locally uniformly in $\overline Q_T$, for some $u\in C_b(Q_T)$.  Reasoning in the same way,  using the equation satisfied by $Du_{ \vep,  R}$ and $W^{2,\infty}_{loc}$ estimates in space, we also obtain that $Du_{ \vep,  R}\to Du$ locally uniformly in $\overline Q_T$.
\\[0.2cm]
{\it Step 2.}\quad By Step 1, $b_{ \vep,  R} ( x )  +D_p H ( x, Du_{ \vep,  R}  )$ satisfies condition \rife{tb} with $b_1=b$ and $b_{2,\vep, R}=D_pH(x,Du_{ \vep,  R}(t,x))\to D_p H(x, Du(t,x))$ (locally uniformly). Hence, from Lemma \ref{lem:comp} we infer the compactness of $m_{ \vep,  R}$ and the stability of the FP equation; namely, there exists $m\in C([0,T];\cP_k(\R^d))$ such that $m_{ \vep,  R}\to m$ in $C([0,T];\cP_k(\R^d))$ and $m$ solves the FP  equation with drift $b+D_p H(x, Du)$.
\\[0.2cm]
{\it Step 3.}\quad To conclude, we observe that the convergence of $m_{ \vep,  R}(t)$ to $m(t)$ and assumption \ref{G} imply that $F(x,m_{ \vep,  R}(t))$ uniformly converges to $F(x,m(t))$.  By the stability of  viscosity solutions, we deduce that $u$ is a solution of \rife{mfg_HJB}.
\\[0.2cm]
{\it Step 4.}\quad Assuming \ref{F1}, we now prove the uniqueness of solutions of \rife{eqn:parabolic_MFG}. We follow the classical argument introduced in  \cite[Theorem 2.4]{lasry2007mean}. 
Let $ ( u_{1}, m_{1} )$ and $ ( u_{2}, m_{2} )$ be two solutions of \eqref{eqn:parabolic_MFG}.
We write the HJB equations for $u_{1}$ and $u_{2}$ and subtract them; in particular, $u_1-u_2$ satisfies
\begin{align*}
 - &  \partial_t (u_1-u_2) - \mathcal{L} (u_1-u_2)    + (b(x)+ D_pH(x,Du_1)) \cdott D(u_1-u_2) \\ 
 &  =   F(x,m_1(t))-F(x,m_2(t))+ D_pH(x,Du_1) \cdott D(u_1-u_2)- (H(x, Du_1)- H(x,Du_2)),
\end{align*}
where we observe that the right-hand side is continuous and bounded, due to \ref{G}, \ref{H1}--\ref{H3} 
and the regularity of $u_1, u_2$. Therefore, we can use the equation of $m_1$ according to Definition \ref{def:dsol}, obtaining
\begin{align*}
&   \int_{\R^d} (u_1(t)-u_2(t)) dm_1(t)\\
&= \int_{\R^d} (u_1(0)-u_2(0))  dm_0+ \int_0^t \int_{\R^d} [F(x,m_1)-F(x,m_2)]dm_1ds\\
& \quad +\int_0^t \int_{\R^d}\Big[ D_pH(x,Du_1) \cdott D(u_1-u_2)- (H(x, Du_1)- H(x,Du_2))\Big] dm_1  ds. 
\end{align*}
Notice that the left-hand side vanishes for $t=T$, since there we have $u_1=u_2=u_T$. 
Next, we argue similarly for the equation of $m_2$, we subtract the results and we get
\begin{align*}
& \int_0^T \int_{\R^d} [F(x,m_1)-F(x,m_2)](dm_1-dm_2)ds\\
&  +\int_0^T \int_{\R^d}\Big[ D_pH(x,Du_1) \cdott D(u_1-u_2)- (H(x, Du_1)- H(x,Du_2))\Big] dm_1 ds 
\\ & -\int_0^T \int_{\R^d}\Big[ D_pH(x,Du_2) \cdott D(u_1-u_2)- (H(x, Du_1)- H(x,Du_2))\Big] dm_2 \, ds=0
\end{align*}
The convexity of $H$ and assumption \ref{F1} imply that the three above integrals are all equal to zero, and the strict convexity coming from \ref{H5} implies that $Du_1=Du_2$ both $m_1$ and $m_2$ almost everywhere. This means that $m_1$ and $m_2$ are solutions of the same Fokker-Planck equation, hence $m_1= m_2$ by Proposition \ref{ex-FP}.  As a consequence, $F(x,m_1)= F(x,m_2)$, and $u_1=u_2$ by uniqueness of viscosity solutions of the first equation.
\end{proof}

\subsection{Well-poseness for stationary ergodic MFG systems}

Our next step consists in showing the existence of a stationary solution  to the ergodic system.

\begin{thm}\label{exergo}
Assume \ref{nu'}, \ref{B}--\ref{B2}, \ref{H1}--\ref{H5}, \ref{G}.  Then there exists a triple $ (\lambda, \bar u, \bar m) \in \mathbb{R} \times  W^{1,\infty} (\rd) \times \cP _k(\rd)$, for $k \in (0, \sigma)$, which is a solution of  the stationary ergodic MFG system
\eqref{eqn:ergodic_MFG}. Moreover $\bar u\in C^1(\R^d)\cap W^{2,\infty}_{loc}(\R^d)$ and $\bar m\in L^1(\R^d)$, $\bar m>0$.
\end{thm}

\begin{proof}
For a constant $C > 0$ to be defined later, we set
 \begin{align}
     \mathcal{C}:= \{ \mu \in \cP ( \R^d  )   :
     \int_{\R^d } \langle x  \rangle^{k} d\mu   \leq C\} 
     \label{def:set_C_fixed_point_stationary}
\end{align}
for $k \in ( 0, \sigma )$. 
The set $\cP ( \R^d  )$ is equipped with the metric $d_{0}$. The set $\mathcal{C}$ is a closed and convex subset of $\cP(\R^d)$. It is compact by the Prokhorov theorem, using 
the tightness of $\mu$.
We define a fixed point map from $\mathcal{C}$ into itself. Let $\mu \in \mathcal{C}$, and let $ ( \lambda,u )$  be the unique solution of
\begin{align*}
\begin{cases}
 \lambda - \mathcal{L} u + b ( x ) \cdott Du + H ( x,Du ) = F ( x, \mu ) , \\ 
  u ( 0 ) = 0
\end{cases}
\end{align*}
which exists by  Theorem \ref{thm:well_posedness_HJB_ergodic}.
By Theorem \ref{thm:ergodic_double_derivative} this solution is smooth and satifies $Du \in L^{\infty} (\rd) $ and 
$\|D^2 u \|_{L^{\infty} (B_r) } \leq C_r$.  
Let then $m = T ( \mu )$ be defined as the unique solution of the Fokker-Planck equation
\beq\label{mfg_FP_stat}
    - \mathcal{L}^* m  - \text{div} \Big( m (     b ( x )+ D_{p} H ( x,Du )  ) \Big) = 0. 
\eeq
given by Proposition \ref{ex-FP-stat}. 
The proof is complete if we can show that $T$ has a fixed point. This follows from Schauder's fixed point argument if (i) $T$ maps $\mathcal{C}$ into itself, and (ii) $T$ is continuous on $\mathcal{C}$. We conclude the proof by checking these properties. 
\medskip

\noindent (i) \ {\em $T$ is well-defined and $\mathcal{C}$ is invariant.}  \
According to the notion of solution introduced in Definition \ref{def:dsol}, for any 
$\phi\in C^2_c(\R^d)$,  
$$
 0 = \int_{\rd}    \Big( - \mathcal{L}\phi + b \cdott D\phi + D_p H (x,Du) \cdott D\phi \Big)\,dm.
 $$
 Since $m\in \cP_k$, by approximation we can take possibly unbounded $\phi$ such that the above integrand belongs to $L^\infty(\langle x\rangle^{-k})$. We take $\phi (x) = \langle x \rangle^{k}$, $k\in (0,\sigma)$; since $\| D_p H (Du) \|_{\infty} \leq L_H$ by \ref{H2}, we can use 
Lemma \ref{lemma:supersolution_property} and we deduce that 
\begin{align*}
    0 = \int_{\rd}   ( - \mathcal{L} \phi+ b \cdott D\phi + D_p H (x,D u) \cdott D\phi )dm \geq \int_{\rd}   (\omega_0\phi (x) - K) dm, 
\end{align*}
for some $\omega_0, K>0$.  This yields  $\int_{\rd}  \langle x\rangle^k dm \leq K/{\omega_0}$. 
Thus the map $T: \mathcal{C} \to \mathcal{C}$ is well-defined for $C= \frac K{\omega_0}$. 
\medskip

\noindent (ii) \ {\em $T$ is continuous.}  \
Let $\mu_{n} \to \mu \in \mathcal{C}$ be a convergent sequence. By \ref{G}, $F(x,\mu_n)\to F(x,\mu)$ uniformly and is bounded  in $W^{1,\infty}$. Then there is a unique solution $( \lambda_{n}, u_{n} )$ 
of the corresponding HJB equation by Theorem \ref{thm:well_posedness_HJB_ergodic}. 
By Theorem  \ref{thm:ergodic_double_derivative} we also have uniform bounds  
\begin{align*}
    | \lambda_{n} | \leq C \qquad\text{and}\qquad \| u_{n}  \|_{W^{2,\infty} ( B_r )} \leq C_r,\quad r>0.
\end{align*}
Thus $\lambda_{n} \subset \R$ is compact, and $u_{n} $ is compact in $C^{1, \delta}_{\textup{loc}}$ for every $\delta\in [0,1)$.
Let $ ( \lambda_{n_{k}} , u_{n_{k}} )$ be a convergent subsequence. Then, $Du_{n_{k}} \to Du$ locally uniformly. 
Furthermore, writing $\mathcal{L} u_{n_{k}} (x) = \mathcal{L}_1 u_{n_{k}} (x) + \mathcal{L}^1 u_{n_{k}} (x) $, we see that
$\mathcal{L}_1 u_{n_{k}} (x) \to \mathcal{L}_1 u (x)$ by uniform convergence in $C^{1, \beta} (B_1)$, while $\mathcal{L}^1 u_{n_{k}} (x) \to 
\mathcal{L}^1 u (x)$ by dominated convergence.
Thus, for every fixed $x \in \R^d $, the equality
\begin{align*}
    \lambda_{n_{k}} - \mathcal{L} u_{n_{k}}  + b ( x ) \cdott Du_{n_{k}} + H ( x,Du_{n_{k}} ) = F ( x, \mu_{n_k} )
\end{align*}
yields, as $n_k\to \infty$,
\begin{align*}
    \lambda - \mathcal{L} u   + b ( x ) \cdott Du + H ( x,Du ) = F ( x, \mu ) .
\end{align*}
In addition $u_{n_{k}} ( 0 ) \to u ( 0 ) = 0$. Since the solution $ ( \lambda,u )$ is unique,
we get full convergence $ (\lambda_n , u_n) \to (\lambda , u)$.

By definition of   $\mathcal{C} \subset \cP (\rd)$, $m_n= T(\mu_n)$ is tight, and therefore it is 
relatively compact in $\cP(\R^d)$ by the Prokhorov theorem.
Thus, there exists a subsequence $ ( m_{n_{k}} )$ and $m\in \mathcal{C} $ such that
\begin{align*}
    \int_{\R^d }  \xi \, dm_{n_k}  \to \int_{\R^d }   \xi \, dm \qquad \text{for all } \xi \in C_b ( \R^d  ).
\end{align*}
In order to show that $m$ is a solution according to our definition, we proceed as in Lemma \ref{lem:comp}. Namely, for  $f\in C_b(Q_t) $,  we consider the viscosity solution of
\begin{align}\label{fpdualR2}
 c_n - \mathcal L \phi_n  +   (b(x)+ D_pH(x,Du_n)) \cdott D \phi_n  = f & \quad \text{in} \,\,   \R^d, 
\end{align}
which exists for some $c_n\in \R$, by Theorem  \ref{thm:well_posedness_HJB_ergodic}.  Hence we have (see Remark \ref{form_FP-stat})
\begin{align*}
  \int_{\R^d} f\, dm_n  =c_n \,.
    \end{align*}
Now, the local uniform convergence of $D_pH(x,Du_n)$ towards $D_pH( x,Du)$  and the stability of viscosity solutions imply that $\phi_n, c_n$ are uniformly bounded and $\phi_n $ locally uniformly converges to  a viscosity solution $\phi$ of the equation
$$
c - \mathcal L \phi   +   (b(x)+ D_pH(x,Du )) \cdott D \phi   = f \qquad  \text{in} \quad    \R^d,
$$  
where $c$ is the unique ergodic constant relative to $f$. Passing to the limit as $m_{n_k} \to m$ we deduce that 
\beq\label{barmeq}
  \int_{\R^d} f\, dm  =c \quad \forall (f,c)\,:\quad \exists \phi\,\,\hbox{sol. of}\quad c - \mathcal L \phi   +   (b(x)+ D_pH(x,Du )) \cdott D \phi   = f \,.
\eeq
This means that $m$ is the unique solution of \rife{mfg_FP_stat}, i.e. $m= T(\mu)$, and the whole sequence  $m_{n} \to m$ in $\cP ( \R^d  )$. Thus $T$ is continuous.
\end{proof}

As in the time dependent case, assuming the Lasry-Lions monotonicity condition implies uniqueness for the stationary problem, as well. The proof is similar to the evolutionary case (Step 4 in Theorem \ref{ex-MFG}), up to using the stationary formulation \rife{barmeq} for the measure $m$  and the uniqueness of viscosity solutions to the ergodic problem, for the first equation. We omit the details. 

\begin{thm}
  Under the assumptions of Theorem \ref{exergo}, and the additional condition  \ref{F1}, there exists at most one  solution $ ( \lambda,u,m ) \in \R \times   ( W^{1,\infty} ( \R^d  ) \cap W_{\text{loc}}^{2,\infty} (\rd) )  \times \cP_k$ of \eqref{eqn:ergodic_MFG}. 
\end{thm}

\section{The exponential turnpike estimate for  fractional MFGs}\label{turnpi}

In this section we prove our main result, namely the exponential turnpike estimate, for a long time horizon $T$,  for solutions of the MFG system \rife{eqn:parabolic_MFG}.
 To do this, we first need to establish preliminary long time estimates 
 for  (a rough version of)  the linearized MFG system.

\subsection{Long time decay estimates for   drift-diffusion operators}

Here we summarize the estimates on linear  Hamilton-Jacobi and Fokker-Planck equations, which are needed to establish the convergence to equilibrium for the MFG system. We start by recalling the  following result   proved  in \cite{porretta2024decay}.
We recall the notation \rife{weisem} and \rife{tvk} to be used in the following.

\begin{thm}\label{decayum}  Let $\cL$ be defined by \rife{def:nonlocal_diffusion}, where \ref{nu'} holds true.  Assume that ${\mathcal B}$ satisfies \rife{bdiss}--\rife{disso}.
Then the following holds:

(a) Let $v$ be a  (viscosity) solution of  
\beq\label{linhj}
\begin{cases}
-\partial_t v - \mathcal{L} v   + {\mathcal B}(t,x) \cdott Dv     =  0 &  \text{ in } ( 0,T ) \times \R^d, \\ 
v ( T,x ) = v_T (x),\qquad  \,&  \text{ in }   \R^d
  \end{cases}
\eeq
such that $|v(t,x)| \leq C (1+ |x|)^\gamma$ for some $\gamma<\sigma$, $C>0$.  

For any $0<k<\sigma$ such that $v_T \in L^\infty(\langle x\rangle^{-k})$, there exist  $K, \omega >0$ 
 (only depending on $\alpha_0, \beta_0, R, d, \sigma, k, c_{\sigma,d}$)  such that $v$ satisfies
\beq\label{decayk}
[v(t)]_{\langle x\rangle^k} \leq K \, e^{-\omega (T-t)} \, [v_T]_{\langle x\rangle^k}, \qquad t\in (0,T).
\eeq

\vskip0.5em
(b) Let $\mu$ be a solution of 
\beq\label{linfp}
\begin{cases}
\partial_t \mu - \mathcal{L}^* \mu   -\dive( {\mathcal B}(t,x) \mu)     =  0 &  \text{ in } ( 0,T ) \times \R^d, \\ 
\mu( 0,x ) = \mu_0 (x),\qquad  \,&  \text{ in }   \R^d
\end{cases}
\eeq
where  $\mu_0 \in  \cM_k(\rd) $, for some  $0<k<\sigma$, and satisfies  $\int_{\R^d} d\mu_0=0$.  
Then there exist   $K, \omega >0$ (as in item (a)) such that 
\beq\label{decaykm}
\|\mu(t)\|_{TV_k}\leq K \, e^{-\omega t} \, \|\mu_0\|_{TV_k}, \qquad t\in (0,T).
\eeq
\end{thm}

\vskip1em

The above estimates say, roughly speaking,  that the drift-diffusion operator generated by linearizing the MFG system (around one solution $(\bar u, \bar m)$) decays exponentially. This is observed in weighted total variation norms for  $m$, and in weighted oscillation estimates for   $u$. In order to   
use those estimates
for the MFG setting, we will need to upgrade and refine those decay estimates. In fact, we need both to consider similar estimates for the nonhomogeneous problems, and to 
extend the decay to $Dv$. This is done by matching \rife{decayk} with the local regularizing effect of Proposition \ref{regeff}.

We recall  (see \cite{hairer2011yet}) the following equivalence of seminorms:
\beq\label{hairer}
\inf\limits_{c\in \R} \|u+c\|_{L^\infty(\langle x\rangle^{-k})}= [u]_{\langle x\rangle^{-k}} 
\eeq
and we start with a kind of  Duhamel's formula for the decay of $v$ and $Dv$.

%

\begin{proposition}\label{lem72}  

Let $\cL$ be defined by \rife{def:nonlocal_diffusion}, where \ref{nu'} holds true.    Assume that ${\mathcal B}\in C(Q_T)$ satisfies \rife{bdiss}--\rife{disso},   that $f\in L^\infty((0,T);L^\infty(\langle x\rangle^{-k})$, and $v_T\in L^\infty(\langle x\rangle^{-k})$ for some $k<\sigma$.

Let $v$ be a  (viscosity) solution of  
\beq\label{linhj}
\begin{cases}
-\partial_t v - \mathcal{L} v + {\mathcal B}(t,x) \cdott Dv    =  f &  \text{ in } ( 0,T ) \times \R^d, \\ 
v (T ) = v_T (x),\qquad  \,&  \text{ in }   \R^d
  \end{cases}
\eeq
such that $|v(t,x)| \leq C (1+ |x|)^\gamma$ for some $\gamma<\sigma$, $C>0$.  

Then there exist  $K, \omega >0$ (only depending on $\alpha_0, \beta_0, R, d, \sigma, k, c_{\sigma,d}$) such that
\beq\label{duhamel-u}
[v(t)]_{\langle x^k\rangle} \leq  K\, e^{-\omega (T-t)}  [v_T]_{\langle x^k\rangle} +  K \int_t^T  e^{-\omega(\tau-t)}[f(\tau)]_{\langle x^k\rangle}d\tau 
\eeq
and
\beq\label{duhamel-Du}
\begin{split}
\| Dv(t)\|_{L^\infty(\langle x\rangle^{-k} )}  & \leq  K \,e^{-\omega (T-t)}[v_T]_{\langle x^k\rangle} 
+  K \int_t^T  e^{-\omega(\tau-t)}[f(\tau)]_{\langle x^k\rangle}d\tau  \\ 
& \qquad 
+ K \left(\sup_{s\in [t,t+1]} [f(s)]_{\langle x^k\rangle}\right)
\qquad \forall t\in (0, T-1)\,.
\end{split}
\eeq
In addition, if  $Dv_T\in L^\infty(\langle x\rangle^{-k}dx)$, we also have
\beq\label{dvT}
\|Dv(t) \|_{L^\infty(\langle x\rangle^{-k})} \leq \|Dv_T \|_{L^\infty(\langle x\rangle^{-k})} + K  (T-t)^{1-\frac1\sigma}\!\sup_{s\in [T-1,T]}[f(s)]_{\langle x^k\rangle} \quad \forall t\in (T-1,T).
\eeq
%
%
\end{proposition}

\begin{remark} We stress that the constant $C$  of the growth condition of $v$ (i.e. such that $|v(t,x)| \leq C (1+ |x|)^\gamma$) does not affect the estimates \rife{duhamel-u}--\rife{dvT}. The growth condition on $v$ is only needed to justify the application of  maximum principle in the class of given solutions. 
\end{remark}

\begin{proof}   Let us show \rife{duhamel-u}. To this purpose, we take a unit measure $\nu$ on $\R^d$ which has positive density and such that $\int_{\R^d}\langle x\rangle^k d\nu <\infty$ (e.g.  $\nu$ is  a normalized Gaussian). For any bounded Borel set $B$  we define
$$
\xi:= \frac1{\langle x\rangle^k}\, \dschi_B - \lambda \, \nu\,, \qquad \hbox{where $\lambda= \int_B \frac1{\langle x\rangle^k}\, dx$.}\,   
$$
We notice that $\xi \in L^1(\langle x\rangle^{k})$,  $\int_{\R^d} \xi\, dx = 0$ (because $\nu$ has unit mass) and we estimate
\beq\label{lem721}
 \|\xi\|_{L^1(\langle x\rangle^{k})} \leq |B|+ \lambda\, \int_{\R^d}\langle x\rangle^k d\nu   \leq (1+ C) |B|
\eeq
for some $C$ independent of $B$.  Then we solve the Fokker-Planck equation 
\beq\label{traslo1}
\begin{cases}
\partial_t \mu - \mathcal{L}^* \mu   -\dive( {\mathcal B}(t,x) \mu)     =  0 &  \text{ in } ( t,T ) \times \R^d, \\ 
\mu( t,x ) = \xi (x),\qquad  \,&  \text{ in }   \R^d
  \end{cases}
\eeq
and we deduce, by duality,
$$
\int_{\R^d} v(t) \,\xi\,dx =\int_{\R^d} v_T\, d\mu(T)   + \int_t^T \int_{\R^d} f\, d\mu(\tau) d\tau\,.
$$
Since $\mu$ has zero average, this implies, for any constant $c_1, c_2\in \R$
\begin{align*}
\int_{\R^d} v(t) \,\xi\,dx  & = \int_{\R^d} (v_T +c_1) \, d\mu(T)  + \int_t^T \int_{\R^d} (f+c_2)\, d\mu(\tau) d\tau
\\ & \leq \|v_T +c_1\|_{L^\infty(\langle x\rangle^{-k})} \|\mu(T)\|_{TV_k} +  \int_t^T   \|f+c_2\|_{L^\infty(\langle x\rangle^{-k})} \|\mu(\tau)\|_{TV_k}  d\tau
\end{align*}
Taking infimum on $c_1, c_2$, and using \rife{hairer},   we deduce
 \begin{align*}
\int_{\R^d} v(t) \,\xi\,dx  &  \leq [v_T]_{\langle x^k\rangle} \|\mu(T)\|_{L^1(\langle x\rangle^{k})} +  \int_t^T  [f(\tau)]_{\langle x^k\rangle}\|\mu(\tau)\|_{TV_k}  d\tau
\\ & \leq \left( K\, e^{-\omega (T-t)}  [v_T]_{\langle x^k\rangle} +  K \int_t^T  e^{-\omega(\tau-t)}[f(\tau)]_{\langle x^k\rangle}d\tau \right)  \|\xi\|_{L^1(\langle x\rangle^{k})}
\end{align*}
where, in the last step,  we used \rife{decaykm} for problem \rife{traslo1}. Replacing $\xi$ by $-\xi$ in the argument above yields to
$$
\left | \int_{\R^d} v(t) \,\xi\,dx \right| \leq K  \left( e^{-\omega (T-t)}  [v_T]_{\langle x^k\rangle} +    \int_t^T  e^{-\omega(\tau-t)}[f(\tau)]_{\langle x^k\rangle}d\tau \right)  \|\xi\|_{L^1(\langle x\rangle^{k})}.
$$
The last term is now estimated with \rife{lem721}, while on the left side we  have
$$
\int_{\R^d} v(t) \,\xi\,dx= \int_{B} (v(t) + c_t ) \frac1{\langle x\rangle^k}\,dx \qquad \hbox{for $c_t= - \int_{\R^d} v(t)d\nu$.}
$$
Hence we deduce
$$
\left | \frac1{|B|}\int_{B} (v(t) + c_t ) \frac1{\langle x\rangle^k}\,dx  \right| \leq   (1+ C) K \left(  e^{-\omega (T-t)}  [v_T]_{\langle x^k\rangle} +   \int_t^T  e^{-\omega(\tau-t)}[f(\tau)]_{\langle x^k\rangle}d\tau \right) \,. 
$$
Since this holds for all Borel sets $B$, we deduce  
$$
\| v(t)+ c_t \|_{L^\infty(\langle x\rangle^{-k})} \leq (1+ C) K \left(  e^{-\omega (T-t)}  [v_T]_{\langle x^k\rangle} +    \int_t^T  e^{-\omega(\tau-t)}[f(\tau)]_{\langle x^k\rangle}d\tau\right)\,.
$$
On account of \rife{hairer}, this implies \rife{duhamel-u}, for a possibly different constant $K$. 

In order to get \rife{duhamel-Du}, we apply Proposition \ref{regeff} (in the time interval $(t, t+1)$), which yields
\beq\label{tt+1}
\| Dv(t)\|_{L^\infty(\langle x\rangle^{-k} )} \leq  K\, \sup_{s\in [t,t+1]} [v(s)]_{\langle x\rangle^k} + K\,  \sup_{s\in [t,t+1]} [f(s)]_{\langle x\rangle^k} 
\eeq
where, as before, $K$ is some generic constant only depending on $\alpha_0, \beta_0, R, d, \sigma, k, c_{\sigma,d}$. Using \rife{duhamel-u}, we have 
$$
\sup_{s\in [t,t+1]} [v(s)]_{\langle x\rangle^k} \leq c_\omega K \left(  e^{-\omega (T-t)}  [v_T]_{\langle x^k\rangle} +    \int_t^T  e^{-\omega(\tau-t)}[f(\tau)]_{\langle x^k\rangle}d\tau\right),
$$
hence \rife{tt+1} implies \rife{duhamel-Du}, for a possibly different $K$.
Finally, the estimate \rife{dvT} follows from \rife{estlipda0} in Remark \ref{lipda0}. 
\end{proof}

It will also be crucial, in order to handle the MFG system, to get corresponding estimates for  the nonhomogeneous Fokker-Planck equation
\beq\label{fpnonhom}
\begin{cases}
\partial_t \mu - \mathcal{L}^* \mu   -\dive( {\mathcal B}(t,x) \mu)     =  \dive (\bar m \,\Phi) &  \text{ in } ( 0,T ) \times \R^d, \\ 
\mu( 0,x ) = \mu_0 (x),\qquad  \,&  \text{ in }   \R^d\,.
  \end{cases}
\eeq
In \rife{fpnonhom}, we assume that $\bar m \in \cP_{2k}(\rd)$ and $\Phi\in L^\infty(\langle x\rangle^{-k})$, for some $k>0$. Solutions of \rife{fpnonhom} are meant in the duality sense as in Definition \ref{def:dsol}.
 
 \begin{proposition}\label{estm} Let $\cL$ be defined by \rife{def:nonlocal_diffusion}, where \ref{nu'} holds true.  Assume that $ \mu_0 \in \cM_k$ for some $k<\sigma$, with $\int_{\R^d} d\mu_0=0$,  that ${\mathcal B}\in C(Q_T)$  satisfies \rife{bdiss}--\rife{disso},  $\bar m \in \cP_{2k}$ and   $\Phi\in L^\infty(\langle x\rangle^{-k})$.  Then,  there exist positive constants $K, \omega >0$ (only depending on $\alpha_0, \beta_0, d, \sigma, k, c_{d,\sigma}, \|\bar m\|_{TV_{2k}}$) such that any  solution $\mu$ of \rife{fpnonhom} satisfies
\beq\label{estm1}
\begin{split}
\|\mu (t)\|_{TV_k} & \leq K \, \left\{ e^{-\omega t} \, \|\mu_0\|_{TV_k} + \de^{1-\frac1\sigma}\sup_{s\in [(t-\de)_+,t]}  \|\Phi(s)\|_{ L^\infty(\langle x\rangle^{-k})} \right.
\\
& \left. \qquad\qquad \qquad\quad + \de^{\frac12-\frac1\sigma} \bigg( \int_0^{(t-\de)_+} \int_{\R^d}  |\Phi(s)|^2\, d\bar m ds \bigg)^{\frac12}\right\}
\end{split}
\eeq
for any $\de>0$.

In addition, for any $\gamma\in (1,\sigma)$ we have (denoting $\gamma'= \frac{\gamma}{\gamma-1}$)
\beq\label{estm2}
\begin{split}
\int_0^T \|\mu(t)\|_{TV_k}^{\gamma'} dt   \leq C \Bigg\{ & \| \mu_0\|_{TV_k}^{\gamma'} 
\\
&  +  \bigg( \sup_{s\in [0,T]} \|\Phi(s)\|_{ L^\infty(\langle x\rangle^{-k})}\bigg)^{\gamma'-2}  \int_0^T \int_{\R^d} |\Phi(t)|^2\, d\bar m dt  \Bigg\}
\end{split}
\eeq
for some constant $C$ depending on $\gamma, \alpha_0, \beta_0, d, \sigma, k, c_{d,\sigma}, \|\bar m\|_{TV_{2k}}$.  
\end{proposition}
 
\begin{proof} As in Proposition  \ref{lem72},  we consider the duality with the adjoint problem 
\beq\label{traslo2}
\begin{cases}
-\partial_t u - \mathcal{L} u   + {\mathcal B}(t,x) \cdott   Du   =  0 &  \text{ in } ( 0,t ) \times \R^d, \\ 
u( t,x ) = \xi (x),\qquad  \,&  \text{ in }   \R^d\,
  \end{cases}
\eeq
for a generic $\xi \in L^\infty(\langle x\rangle^{-k})$.  Let us point out that estimates \rife{duhamel-u}--\rife{duhamel-Du} (read with $f=0$), together with the short time regularizing effect (see Remark \ref{lipda0}) imply that the solution of \rife{traslo2} satisfies
\beq\label{ecchila}
\| Du(s)\|_{L^\infty(\langle x\rangle^{-k} )} \leq  K\, \frac{ e^{-\omega (t-s)}}{(t-s)^{\frac1\sigma}} \,\,  \|\xi\|_{L^\infty(\langle x\rangle^{-k})}\,, \quad \forall s\in (0,t)\,.
\eeq
The duality  between \rife{traslo2} and \rife{fpnonhom} implies
\begin{align*}
\int_{\R^d} \xi \, d\mu(t)  & = \int_{\R^d}  u(0) \, d\mu_0 - \int_0^t \int_{\R^d} \Phi\cdott Du\, d\bar m \\
& \leq \|\mu_0\|_{TV_k} \|u(0)+c\|_{L^\infty(\langle x\rangle^{-k})}- \int_0^t \int_{\R^d} \Phi\cdott Du\, d\bar m, 
\end{align*}
where we used that $\mu_0$ has zero average and $c$ is any constant. Due to \rife{hairer}, optimizing on $c$ and using Theorem \ref{decayum} we get
$$
\int_{\R^d} \xi \, d\mu(t)   \leq  \|\mu_0\|_{TV_k} \, K e^{-\omega t} \|\xi\|_{L^\infty(\langle x\rangle^{-k})}- \int_0^t \int_{\R^d} \Phi\cdott Du\, d\bar m \,.
$$
Now we estimate  the last term, and we split it into two parts, in the case that $t>\de$. We get
\beq\label{calma}
\begin{split}
\int_{\R^d} \xi \, d\mu(t)    \leq \ &
 \|\mu_0\|_{TV_k} \, K e^{-\omega t} \|\xi\|_{L^\infty(\langle x\rangle^{-k} )}
-  \int_{(t-\de)_+}^t \int_{\R^d} \Phi\cdott Du\, d\bar m \\
&  
+ \bigg( \int_0^{(t-\de)_+} \int_{\R^d}  |\Phi|^2\, d\bar m \,  ds \bigg)^{\frac12} \bigg( \int_0^{(t-\de)_+} \int_{\R^d} |Du|^2\, d\bar m   ds \bigg)^{\frac12}
\end{split}
\eeq
Using   \rife{ecchila}, and the fact that $\bar m \in \cP_{2k} $, we estimate
\begin{align*}
\bigg( \int_0^{(t-\de)_+} \int_{\R^d} |Du|^2\, d\bar m  ds \bigg)^{\frac12} & \leq   C \|\xi\|_{L^\infty(\langle x\rangle^{-k})}\|\bar m\|_{TV_{2k}}^{\frac12}
  \bigg( \int_0^{(t-\de)_+}  e^{-2\omega(t-s)}(t-s)^{-2/\sigma}ds \bigg)^{\frac12} 
  \\
  &  \leq C \,e^{-\omega\, \de}\, \de^{\frac12-\frac1\sigma}  \,\  \|\xi\|_{L^\infty(\langle x\rangle^{-k})}\,.
\end{align*}
Here and after, we denote by $C$ possibly different constants only depending on $\alpha_0, \beta_0, d, \sigma$, $k, c_{d,\sigma}, \|\bar m\|_{TV_{2k}}$ (and depending on $\gamma$ when we get later at \rife{estm2}).  We also have, still using \rife{ecchila},
\begin{align*}
\int_{(t-\de)_+}^t \int_{\R^d} \Phi\cdott Du\, d\bar m\,ds & \leq    \int_{(t-\de)_+}^t  \|Du(s)\|_{L^\infty(\langle x\rangle^{-k})}\, \|\Phi(s)\|_{ L^\infty(\langle x\rangle^{-k} )} \int_{\R^d} \langle x\rangle^{2k}d\bar m  \
\\
& \leq  \|\bar m\|_{TV_{2k}}\, K\,  \|\xi\|_{L^\infty(\langle x\rangle^{-k} )} \int_{(t-\de)_+}^t  (t-s)^{-\frac1\sigma} \|\Phi(s)\|_{ L^\infty(\langle x\rangle^{-k} )}  \\
& 
\leq C \, \de^{1-\frac1\sigma} \,\|\xi\|_{L^\infty(\langle x\rangle^{-k} )}\, \sup_{s\in [(t-\de)_+,t]} \|\Phi(s)\|_{ L^\infty(\langle x\rangle^{-k})}  \,.
\end{align*}
Putting together the above estimates, we deduce from \rife{calma}:
\begin{align*}
 \int_{\R^d} \xi \, d\mu(t)    \leq  C\,  \Bigg\{ & e^{-\omega t}  \|\mu_0\|_{TV_k} + \de^{1-\frac1\sigma} \sup_{s\in [(t-\de)_+,t]} \|\Phi(s)\|_{ L^\infty(\langle x\rangle^{-k} )} 
\\
&  + \de^{\frac12-\frac1\sigma} \Big( \int_0^{(t-\delta)_+} \int_{\R^d} |\Phi|^2\,d\bar m ds \Big)^{\frac12}\Bigg\}
\|\xi\|_{L^\infty(\langle x\rangle^{-k})}
\end{align*}
which allows us to deduce \rife{estm1}.

In order to obtain \rife{estm2} we proceed similarly; as before we estimate  by duality (using the solution of \rife{traslo2})
\begin{align*}
&
\int_{\R^d} \xi \, d\mu(t)   = \int_{\R^d}  u(0) \, d\mu_0 - \int_0^t \int_{\R^d} \Phi\cdott Du\, d\bar m
\\
& \,\, \leq C\,  e^{-\omega t}  \|\mu_0\|_{TV_k} \, \|\xi\|_{L^\infty(\langle x\rangle^{-k} )} + \int_0^t  (\|\bar m\|_{TV_{2k}})^{\frac12} \| Du(s)\|_{L^\infty(\langle x\rangle^{-k})}\, \|\Phi(s)\|_{L^2(d\bar m)}\\
& \,\, \leq C\,  e^{-\omega t}  \|\mu_0\|_{TV_k} \, \|\xi\|_{L^\infty(\langle x\rangle^{-k} )} \\
&   \quad + C  \|\xi\|_{L^\infty(\langle x\rangle^{-k} )}  \left\{ \int_{t-1}^t  \|\Phi(s)\|_{L^2(d\bar m)}(t-s)^{-\frac1\sigma}ds
+ \int_0^{t-1}  e^{-\omega(t-s)} \|\Phi(s)\|_{L^2(d\bar m)}
ds  \right\}
\end{align*}
where we used \rife{ecchila}, and we suppose by now that $t>1$.
The above inequality allows us to estimate $\|\mu(t)\|_{TV_k}$, and by H\"older inequality we obtain, for any $\gamma\in (1,\sigma)$ (so that $  (t-s)^{-\frac\gamma\sigma}$ is integrable),
\beq\label{summa}
\begin{split}
\|\mu(t)\|_{TV_k} & \leq C \Bigg\{ e^{-\omega t}  \|\mu_0\|_{TV_k}+  C \bigg(\int_{t-1}^t \Big(  \int_{\R^d}  |\Phi(s)|^2\,   d\bar m \Big)^{\frac{\gamma'}2}ds \bigg)^{\frac1{\gamma'}} 
\\
& 
\qquad \quad  +  C\bigg( \int_0^{t-1}  e^{-\omega(t-s)} \int_{\R^d} |\Phi(s)|^2\,  d\bar m \, ds \bigg)^{\frac12}  \Bigg\}\,.
\end{split}
\eeq
 We take the power $\gamma'$ and we integrate obtaining
 \beq\label{sporca}
 \begin{split}
\int_1^T \|m(t)\|_{TV_k}^{\gamma'}dt  \leq \ &  C  \|\mu_0\|_{TV_k}^{\gamma'} +  C  \int_1^T dt \int_{t-1}^t \left(  \int_{\R^d}  |\Phi(s)|^2\,   d\bar m \right)^{\frac{\gamma'}2}  ds\\
& 
+  C\int_1^T \bigg( \int_0^{t-1}  e^{-\omega(t-s)} \int_{\R^d} |\Phi(s)|^2\,  d\bar m \, ds \bigg)^{\frac{\gamma'}2}   \,.
\end{split}
\eeq
Now we observe that, for $s\in [t-1,t]$,
$$
 \int_{\R^d}   |\Phi(s)|^2\,  d\bar m\leq  \bigg( \sup_{s\in [t-1,t]} \|\Phi(s)\|_{ L^\infty(\langle x\rangle^{-k})} \bigg)^2\, \|\bar m\|_{TV_{2k}} 
$$
and similarly 
$$
\int_0^{t-1}  e^{-\omega(t-s)}  \int_{\R^d} |\Phi(s)|^2\,  d\bar m
\, ds  \leq  C_\omega \,  \bigg( 
\sup_{s\in [0,t-1]} \|\Phi(s)\|_{ L^\infty(\langle x\rangle^{-k})} \bigg)^2
\, \|\bar m\|_{TV_{2k}}  \,.
$$
Therefore, using that $\gamma' \geq 2$ (because $\gamma<\sigma<2$), we deduce from \rife{sporca} that
\begin{align*}
&
\int_1^T \|\mu(t)\|_{TV_k}^{\gamma'}dt  \leq C\,  \|\mu_0\|_{TV_k}^{\gamma'}  \\
& \quad +   C\,  \bigg( \sup_{s\in [0,T]} \|\Phi(s)\|_{ L^\infty(\langle x\rangle^{-k})} \bigg)^{\gamma'-2}  \int_1^T dt \int_{t-1}^t  \int_{\R^d}  |\Phi(s)|^2\, d\bar m 
\\
&    \quad  + C\, \bigg( \sup_{s\in [0,T]} \|\Phi(s)\|_{ L^\infty(\langle x\rangle^{-k})} \bigg)^{\gamma'-2}   \int_1^T dt \int_0^{t-1}  e^{-\omega(t-s)}  \int_{\R^d}  |\Phi(s)|^2\, d\bar m \, ds 
\end{align*}
for some possibly different constant   $  C$ depending also on $\|\bar m\|_{TV_{2k}}$.
Exchanging the order of integration between $s$ and $t$ we conclude that 
\beq\label{da1aT}\begin{split}
\int_1^T \|\mu(t)\|_{TV_k}^{\gamma'} dt   & \leq C\, \left\{ \| \mu_0\|_{TV_k}^{\gamma'} 
+    \bigg( \sup_{s\in [0,T]} \|\Phi(s)\|_{ L^\infty(\langle x\rangle^{-k})}\bigg)^{\gamma'-2}  \int_0^T \int_{\R^d}  |\Phi(s)|^2\, d\bar m \, dt  \right\}\,.
\end{split}
\eeq

We easily estimate the integral for $t\in (0,1)$ in a similar way, since we have (as before in  \rife{summa}),
$$
\|\mu(t)\|_{TV_k}\leq K  e^{-\omega t}  \|\mu_0\|_{TV_k}+   K \bigg(\int_{0}^t \left(  \int_{\R^d}  |\Phi(s)|^2\, d\bar m \right)^{\frac{\gamma'}2}ds \bigg)^{\frac1{\gamma'}}\,.
$$
Hence
\begin{align*}
& \int_0^1 \|\mu(t)\|_{TV_k}^{\gamma'} dt     \leq C\, \| \mu_0\|_{TV_k}^{\gamma'} + C\int_0^1 \int_{0}^t \left(  \int_{\R^d}  |\Phi(s)|^2\, d\bar m \right)^{\frac{\gamma'}2}dsdt \\
& \quad \leq C\, \| \mu_0\|_{TV_k}^{\gamma'} + C \left( \sup_{s\in [0,T]} \|\Phi(s)\|_{ L^\infty(\langle x\rangle^{-k})}\right)^{\gamma'-2}  \int_0^T\int_{\R^d}  |\Phi(s)|^2\, d\bar m\,  dt \,.
\end{align*}
Adding this contribution to \rife{da1aT}, we conclude that 
\rife{estm2} holds true.
\end{proof}


\subsection{Proof of  the main result}

Finally, we are ready to prove the main result of the paper   (cf. Theorem \ref{main0}),   that we restate here for the reader's convenience.
 
\vskip0.4em

\begin{thm}\label{turnpike} Assume hypotheses \ref{nu'}, \ref{B}--\ref{B2}, \ref{H1}--\ref{H5} and \ref{F1}--\ref{F3}. Assume that $m_0\in \cP_k$ for some $k\in (0,\sigma)$, and $u_T\in W^{1,\infty}(\rd)$. Let $(u,m)$ be the solution of \rife{eqn:parabolic_MFG} and let $(\lambda, \bar u, \bar m)$ be the solution of  \rife{eqn:ergodic_MFG}, such that  $\bar m \in \cP_{2k}$. Then there exist   $\omega, L>0$ (independent of $T$) such that 
\beq\label{turnest1}
\|m(t)-\bar m\|_{TV_{k}}+ [u(t)-\bar u]_{\langle x^k\rangle}  +  \|Du(t)-D\bar u \|_{L^\infty(\langle x\rangle^{-k})} \leq L (e^{-\omega t}+ e^{-\omega (T-t)}) 
\eeq
for every $t\in (0,T) $.
\end{thm}

\vskip1em
\begin{remark}\label{constant} The values of $\omega$ and $L$ in \rife{turnest1} depend on the dimension $d$,  on the drift-diffusion constants $\alpha, \beta, c_{\sigma,d}$ appearing in \ref{nu'}--\ref{B}, on the growth constants of the function $H(x,p)$ restricted to $|p|\leq K$ (given by the gradient bound in Theorem \ref{ex-MFG}),  on $\|m_0\|_{TV_k}$, $\|u_T\|_{L^\infty(\langle x\rangle^{-k})}$, $\|Du_T\|_{L^\infty(\langle x\rangle^{-k})}$, and on $\bar m$, in particular through $\| \bar m\|_{TV_{2k}}$.
\end{remark}

\vskip1em

\begin{proof}[{\bf Proof of Theorem \ref{turnpike}.}] \quad We set 
$$
v:= u- \bar u- \lambda(T-t),\quad  \mu:= m-\bar m\,,
$$
and we notice that $(\mu,v)$ solve the system
 \beq\label{ls}
 \begin{cases}
-\partial_t v - \cI[ v] + h(x, Dv)  = f(x,\mu(t)),   &  t\in (0,T),
\\
v(T)= v_T,& 
\\[0.2cm]
\partial_t \mu  - \cI^*[ \mu]- \dive(\mu\, D_ph(x, Dv)) =  \dive(\bar m \Phi(x,Dv)), &  \,t\in (0,T),
\\
\mu(0)=     \mu_0,
& \end{cases}
\eeq
where  $v_T= u_T- \bar u$, $\mu_0= m_0-\bar m$, and where the functions $h(x,p), f(x,\mu), \Phi(x,p)$ are defined by 
\begin{align*}
& h(x,p):= b(x) \cdott p + H(x,D\bar u(x) + p) - H(x,D\bar u(x)),
\\
&  f(x,\mu):= F(x,\bar m   + \mu) - F(x,\bar m ),
\\
& 
\Phi(x,p):=  D_pH (x,D\bar u(x) + p)- D_pH (x, D\bar u(x)) \,.
\end{align*}
We first recall (from the gradient bounds in Theorems 
  \ref{lipest0} and \ref{thm:well_posedness_HJB_ergodic})  that $Du$ is   bounded (uniformly in time) and so is $D\bar u$, which means that $Dv$ is uniformly bounded, and the function $H(x,p)$ can be assumed to be bounded, with bounded second derivatives, and uniformly convex in $p$ due to \ref{H1} and \ref{H5}.  Then, if we exploit  the  duality between $v$ and $\mu$, using monotonicity of $F$   from \ref{F1}   and convexity of $H$  we get
\begin{align*}
-\frac{d}{dt} \, \into  v(t)\,d\mu(t)  & =  \into (F(x,\bar m   + \mu(t)) - F(x,\bar m ))d\mu(t) \\
& \quad - \into (h(x,Dv)- D_ph (x, Dv)\cdott Dv)d\mu(t) 
+ \into \Phi(x,Dv)\cdott Dv\, d\bar m 
 \\ & \geq  \into \Phi(x,Dv)\cdott Dv\, d\bar m  
 \\ & =   \into   \big[ D_pH (x,D\bar u(x) + Dv(x))- D_pH (x, D\bar u(x))\big]\cdott Dv(x)\, d\bar m  
\end{align*}
which yields, due to  \ref{H5},
 \beq\label{ddtT}
-\frac{d}{dt} \, \into v(t)\,d\mu(t)    \geq c_0\into   |Dv|^2  \, d\bar m 
\eeq
for some $c_0>0$. This implies
\beq\label{globT}
 \begin{split}
 & c_0\int_0^T\into   |Dv|^2 \, d\bar m  \leq \into v(0)\, d\mu_0    - \into v_T\, d\mu(T)   \\
 & \quad \leq \|\mu_0\|_{TV_{k}} \, \inf_c\, [\|v(0)+c\|_{L^\infty(\langle x\rangle^{-k})}] + \|\mu(T)\|_{TV_k} \, \inf_c\, [\|v_T+c\|_{L^\infty(\langle x\rangle^{-k})}]
 \\ & \quad = \|\mu_0\|_{TV_k} \,   [v(0)]_{\langle x\rangle^{-k}} + \|\mu(T)\|_{TV_k} \,  [v_T]_{\langle x\rangle^{-k}}
 \end{split}
\eeq
where we used that $\mu$ has zero average and the characterization of the seminorm in \rife{hairer}. 
Here and after,  $k\in (0,\sigma)$ is chosen in a way that $\bar m \in \cP_{2k}$. 

Now we apply Proposition \ref{estm} with 
${\mathcal B}(t,x)=D_ph (x,Dv(t,x))$ and $\Phi= \Phi(x,Dv(t,x))$. Since $D_ph (x,Dv(t,x))= b(x)+ D_pH(x,Du(t,x))$, the fact that $u$ is $C^1$ and globally Lipschitz implies, together with conditions \ref{B}--\ref{B2} and \ref{H2}--\ref{H3}, 
that  ${\mathcal B}$ is a continuous vector field satisfying \rife{bdiss}--\rife{disso}. Moreover, since $D_pH$ is locally Lipschitz in $p$ from \ref{H5}, and $Du, D\bar u$ are bounded, we have 
$$
|\Phi(x,Dv)|\leq  c_1   \, |Dv|
$$
for some $c_1>0$. Hence  we have that $ \Phi(x,Dv)\in  L^\infty(\langle x\rangle^{-k})\cap L^2((0,T);L^2(d\bar m))$, with  
\beq\label{F=B}
\begin{split}
  \into |\Phi(x,Dv)|^2\, \bar m dt & \leq c_1^2\,  \into   |Dv|^2\, d\bar m \,, 
\\
\noalign{\medskip}
\qquad 
\| \Phi(x,Dv)\|_{ L^\infty(\langle x\rangle^{-k})} & \leq c_1 \|Dv\|_{L^\infty(\langle x\rangle^{-k})}\,.
\end{split}
\eeq
We now plug the above quantities in the estimate  \rife{estm1} of Proposition \ref{estm}, and we get
\beq\label{mut}
\begin{split}
\|\mu(t)\|_{TV_{k} } & \leq C \, \Bigg\{ e^{-\omega t} \, \|\mu_0\|_{TV_k} + \de^{1-\frac1\sigma}\sup_{s\in [(t-\de)_+,t]} \|Dv(s)\|_{ L^\infty(\langle x\rangle^{-k})} 
\\  & \qquad\qquad\qquad\qquad\quad\   +  \de^{\frac12-\frac1\sigma}\bigg( \int_0^{(t-\de)_+}\into |Dv|^2\, d\bar m \bigg)^{\frac12}\Bigg\}
\end{split}
\eeq
for every $t\in (0,T)$. Here and after, we denote by $C$ possibly different constants which are independent of $T$, and only depend on the data as explained in Remark \ref{constant}. 

Now we use the equation of $v$, applying Proposition \ref{lem72}. 
 Observe  that  we can write
$$
h(x,Dv)=\tilde{\mathcal B}(t,x)\cdot Dv\ \ \text{for}\ \ \tilde {\mathcal B}(t,x)= b(x) + \int_0^1 D_pH\big(x,(1-s)D\bar u(x) + s Du(t,x)\big)ds.
$$
In particular, $\tilde {\mathcal B}$  satisfies \rife{bdiss}--\rife{disso} due to \ref{B} and \ref{H2}.  So we can apply estimates \rife{duhamel-u} and \rife{duhamel-Du} from Proposition \ref{lem72}, which yield
\begin{align}
\label{brackv}
[v(t)]_{\langle x^k\rangle} &\leq  K\, e^{-\omega (T-t)}  [v_T]_{\langle x^k\rangle} + K \int_t^T  e^{-\omega(s-t)}[f(\cdot,\mu(s))]_{\langle x^k\rangle}ds, \\ 
\label{brackDv}
\| Dv(t)\|_{L^\infty(\langle x\rangle^{-k} )}  & \leq  K \,e^{-\omega (T-t)}[v_T]_{\langle x^k\rangle} 
+  K \int_t^T  e^{-\omega(s-t)}[f(\cdot,\mu(s))]_{\langle x^k\rangle}ds  \\ 
& \qquad\qquad 
+ K \bigg(\sup_{s\in [t,t+1]} [f(\cdot, \mu(s))]_{\langle x^k\rangle}\bigg)
\qquad \forall t\in (0, T-1)\,.\nonumber
\end{align}
In addition, since $Dv_T$ is bounded,
 by \eqref{dvT}   we also have
\beq\label{dvT0}
\|Dv(t) \|_{L^\infty(\langle x\rangle^{-k})} \leq \|Dv_T \|_{L^\infty(\langle x\rangle^{-k})} + C \, \sup_{s\in [T-1,T]}[f(\cdot ,\mu(s))]_{\langle x^k\rangle} \quad \forall t\in (T-1,T).
\eeq
By  condition \ref{F3},  we have
$$
[f(\cdot,\mu(s))]_{\langle x^k\rangle}  \leq C \,\|\mu(s)\|_{TV_k}\,.
$$
Therefore we deduce from \rife{brackDv} and \rife{dvT0} that 
\beq\label{brackDv2}
\|Dv(t) \|_{L^\infty(\langle x\rangle^{-k})}  \leq K \,e^{-\omega (T-t)}[v_T]_{\langle x^k\rangle}  + \|Dv_T \|_{L^\infty(\langle x\rangle^{-k})} + C \sup_{s\in [0,T]} \|\mu(s)\|_{TV_k}\,.
\eeq
for every $t\in [0,T]$. Using this information in \rife{mut} yields
\begin{align*}
\|\mu(t)\|_{TV_{k}}   \leq C \, \Bigg\{& e^{-\omega t} \, \|\mu_0\|_{TV_k} + \de^{1-\frac1\sigma} \left( K \,e^{-\omega (T-t)}[v_T]_{\langle x^k\rangle}  + \|Dv_T \|_{L^\infty(\langle x\rangle^{-k})}\right) 
\\ &   + \de^{1-\frac1\sigma}  \, C  \sup_{s\in [0,T]} \|\mu(s)\|_{TV_k}  
 +  \de^{\frac12-\frac1\sigma}\bigg( \int_0^{(t-\de)_+}\into   |Dv|^2\, d\bar m \bigg)^{\frac12}\Bigg\}\,.
\end{align*}
If we take the sup for $t\in (0,T)$, and we choose a suitable small $\de$, only depending on the constant $C$, we conclude with the preliminary   estimate
\beq\label{prelmu}
\sup_{s\in [0,T]} \|\mu(s)\|_{TV_k}   \leq C \bigg\{\|\mu_0\|_{TV_k} 
+  [v_T]_{\langle x^k\rangle} + \|Dv_T \|_{L^\infty(\langle x\rangle^{-k})} +   \Big( \int_0^T\into  |Dv|^2\, d\bar m\Big)^{\frac12}\bigg\}
\eeq
for some (possibly different) constant $C$. In turn we deduce from \rife{brackv}
\begin{align*}
  [v(0)]_{\langle x^k\rangle}    &\leq  C\, e^{-\omega T}  [v_T]_{\langle x^k\rangle} + C  \, \sup_{s\in [0,T]} \|\mu(s)\|_{TV_k} 
\\ &\leq  C\, \bigg\{\|\mu_0\|_{TV_k} 
+  [v_T]_{\langle x^k\rangle} + \|Dv_T \|_{L^\infty(\langle x\rangle^{-k})} +   \Big( \int_0^T\into   |Dv|^2\, d\bar m \Big)^{\frac12}\bigg\}\,.
\end{align*}
Hereafter, let us set 
$$
M= [v_T]_{\langle x^k\rangle} + \|Dv_T \|_{L^\infty(\langle x\rangle^{-k})}+ \|\mu_0\|_{TV_k}\,.
$$

Using in \rife{globT} the above estimates of $\|\mu(T)\|_{TV_k} $ and $[v(0)]_{\langle x^k\rangle}  $, we obtain
\begin{align*}
c_0\int_0^T\into   |Dv|^2\, d\bar m   &  \leq   C \Big( \|\mu_0\|_{TV_k} + [v_T]_{\langle x^k\rangle}  \Big) \bigg\{M +   \Big( \int_0^T\into   |Dv|^2\, d\bar m\Big)^{\frac12}\bigg\} \\
& \leq C  M \bigg\{M +   \Big( \int_0^T\into   |Dv|^2\, d\bar m\Big)^{\frac12}\bigg\} 
\end{align*}
which yields
\beq\label{prelmdv}
\int_0^T\into   |Dv|^2\, d\bar m    \leq  C \, M^2\,.
\eeq
We can now update the estimate \rife{prelmu}, which gives
$$
 \|\mu(t)\|_{TV_k} \leq C\, M\,,\qquad \forall t\in (0,T)\,.
$$ 
From \rife{brackv} and \rife{brackDv2} we can also deduce a similar estimate for $Dv$ and $v$, namely
\beq\label{prelv}
[v(t)]_{\langle x^k\rangle}+ \|Dv(t) \|_{L^\infty(\langle x\rangle^{-k})} \leq C\, M \qquad \forall t\in (0,T).
\eeq
%

 We now estimate $\mu$ by \rife{estm2} in Proposition \ref{estm} and \rife{F=B},  
\begin{align*}
\int_0^T \|\mu(t)\|_{TV_k}^{\gamma'} dt   \leq C \Bigg\{ & \| \mu_0\|_{TV_k}^{\gamma'} 
\\
&  +  \bigg( c_1 \sup_{s\in [0,T]} \|Dv\|_{L^\infty(\langle x\rangle^{-k})} \bigg)^{\gamma'-2} c_1^2 \int_0^T \int_{\R^d}  |Dv|^2 \, d\bar m   \,dt  \Bigg\}\,.
\end{align*}
Then,  thanks to \rife{prelmdv} and \rife{prelv}, and since $\| \mu_0\|_{TV_k}\leq M$, we deduce 
$$
\int_0^T\|\mu(t)\|_{TV_k}^{\gamma '} dt    \leq  C \, M^{\gamma '}\,.
$$
For any $\tau\in(0,T/2)$,  by mean value theorems for integrals,   this implies that there exist points $\xi_\tau \in [0,\tau/2] $ and $\eta_\tau\in [T- \tau/2, T]$ such that
\beq\label{muxitau}
 \tau\, \|\mu(\xi_\tau)\|_{TV_k}^{\gamma'}   \leq 2C\,  M^{\gamma'}  \,,\quad  \tau \|\mu(\eta_\tau)\|_{TV_k}^{\gamma'}    \leq 2C\, M^{\gamma'} \,.
 \eeq 
 Since $[v(t)]_{\langle x^k\rangle}$   is globally bounded  by \rife{prelv},  we deduce from \rife{ddtT}   (see also \rife{globT}) that 
\beq\label{mdvxitau}
  c_0 \int_{\xi_\tau}^{\eta_\tau} \into   |Dv|^2\, d\bar m      
 \leq \|\mu_{\eta_\tau}\|_{TV_k} \,   [v(\eta_\tau)]_{\langle x\rangle^{-k}} + \|\mu(\xi_\tau)\|_{TV_k} \,  [v_{\xi_\tau}]_{\langle x\rangle^{-k}}  \leq \frac {C}{\tau^{\frac1{\gamma'}}} M^2.
\eeq
We use again \rife{mut},   now  in the interval $(\xi_\tau, \eta_\tau)$.   By \rife{prelv}, \rife{muxitau}, and \eqref{mdvxitau}, we then get  
 \begin{align*}
\|\mu(t)\|_{TV_k} & \leq C \, \left\{ e^{-\omega (t-\xi_\tau)} \, \|\mu(\xi_\tau)\|_{TV_k} + \de^{1-\frac1\sigma}\sup_{[\xi_\tau ,\eta_\tau]} \|Dv(s)\|_{ L^\infty(\langle x\rangle^{-k})}  \right.
\\  & \qquad\qquad\qquad\qquad\qquad\qquad\  \left. +  \de^{\frac12-\frac1\sigma}\left(  \int_{\xi_\tau}^{\eta_\tau} \into  |Dv|^2\, d\bar m   \right)^{\frac12}\right\}
\\
& \leq C \left\{ \frac 1{\tau^{\frac1{\gamma'}} }M +  \de^{1-\frac1\sigma} M + \de^{\frac12-\frac1\sigma}\frac 1 {\tau^{\frac1{2\gamma'}}} M\right\},
\end{align*}
for possibly different constants $C$.
Optimizing this inequality with respect to $\de$ we obtain
\beq\label{xieta}
\|\mu(t)\|_{TV_k} \leq  C \, M \Bigg( \frac 1{\tau^{\frac1{\gamma'}} }+ \frac c {\tau^{\frac1{\gamma'}(1-\frac1\sigma) }}\Bigg) \quad \forall t\in (\xi_\tau, \eta_\tau).
\eeq
Choosing  a suitable large $\tau$ independent of $T$, $\|\mu(t)\|_{TV_k}$ becomes smaller than $\frac12M$ in $(\xi_\tau, \eta_\tau)$, and a fortiori in $( \tau/2, T- \tau /2)$.   

 We   now  further   restrict   the time interval to export a similar estimate to $v$ and $Dv$. 
Suppose in fact that $t\in [\tau, T-\tau]$; if we use \rife{brackv} with terminal   time  
equal to  $\eta_\tau$ we have
\begin{align*}
[v(t)]_{\langle x^k\rangle}  & \leq C\, e^{-\omega (\eta_\tau -t) } [v(\eta_\tau)]_{\langle x^k\rangle}  + C\,    \int_{t}^{\eta_\tau}  e^{-\omega(s-t)}\|\mu(s)\|_{TV_k} ds
\\
& \leq   C\, e^{-\omega (\eta_\tau -t) } M  +      C\, M \Bigg( \frac 1{\tau^{\frac1{\gamma'}} }+ \frac c {\tau^{\frac1{\gamma'}(1-\frac1\sigma) }}\Bigg), 
\end{align*}
where we used \rife{prelv} to estimate the first term and \rife{xieta} for the second one. Since $\eta_\tau -t\geq \frac\tau 2$ for $t\in [\tau, T-\tau]$, we get
$$
[v(t)]_{\langle x^k\rangle}  \leq C\, M \left(  e^{-\omega\, \tau/2}  +     \frac 1{\tau^{\frac1{\gamma'}} }+ \frac c {\tau^{\frac1{\gamma'}(1-\frac1\sigma) }}\right) \qquad \forall t\in [\tau, T-\tau].
$$
We argue similarly  for $Dv$, using \rife{brackDv} starting from $\eta_\tau$ (observe again that $ \eta_\tau-t\geq \frac \tau 2$ is large enough). Using once more \ref{F3} and \rife{xieta} we obtain as before
 $$
 \|Dv(t) \|_{L^\infty(\langle x\rangle^{-k})}   \leq C\, M \left(  e^{-\omega\, \tau/2}  +     \frac 1{\tau^{\frac1{\gamma'}} }+ \frac c {\tau^{\frac1{\gamma'}(1-\frac1\sigma) }}\right) 
 \qquad \forall t\in [\tau, T-\tau].
$$
 Combining these estimates with \rife{xieta}, we conclude that 
there exists $\tau$ (independent of $T$) such that  we have 
\begin{align*}
 & \|\mu(t)\|_{TV_k}+ [v(t)]_{\langle x^k\rangle}  +  \|Dv(t) \|_{L^\infty(\langle x\rangle^{-k})}   \leq \frac12  M \,,\quad \forall t\in [\tau,T-\tau],
\\[0.2cm]
& \qquad \qquad \qquad\qquad \ \hbox{where} \quad M=   [v_T]_{\langle x^k\rangle} + \|Dv_T \|_{L^\infty(\langle x\rangle^{-k})}+ \|\mu_0\|_{TV_k}
\end{align*}
for   every $T>2\tau$. Iterating this estimate yields the exponential rate for the turnpike property, namely  \rife{turnest1} holds  for some $\omega, L>0$.
\end{proof}

%
%
%
%
%

\appendix

\section{H\"older and Lipschitz bounds for fractional Hamilton-Jacobi equations} \label{app:Lipschitz}

In this Appendix we prove different kind of oscillations estimates,  in the form of H\"older and Lipschitz bounds, for solutions of Hamilton-Jacobi equations. In particular, we prove the Lipschitz bound stated in Proposition  \ref{lipest0}.   
For this kind of estimates, which are global in time, a key role is played by the existence of Lyapunov functions. This means, specifically, that the function $\langle x  \rangle^{k}$  (and eventually some truncation if needed)  are  supersolutions for large $|x|$.

\begin{lemma} \label{lemma:supersolution_property}
Assume \ref{B}, \ref{nu'}, and let $\phi ( x ) = \langle x  \rangle^{\gamma}$ with $\gamma\in(0,\sigma)$. For any $L_0 >0$,  there exists $K, \omega_0>0$, 
only depending on  $\alpha, \beta, |b(0)|, \gamma, L_0, \nu, d$, such that
\beq\label{Lambda1}
 \mathcal{L}_{L_0} \phi  ( x ):= - \mathcal{L} \phi (x) 
    +  b ( x )\cdott  D \phi ( x )  - L_{0} | D \phi ( x ) | \geq \omega_0\, \phi (x) - K\qquad \text{ for }  x  \in \rd\,.
\eeq
\end{lemma}
\begin{proof}  We recall that $\langle x  \rangle = \sqrt{1 + | x |^{2} }$. Hence, by definition of $\phi$, we get that 
    $D \phi ( x ) = \gamma  \langle x\rangle ^{\gamma -1 }\frac x{ \langle x\rangle }$
 and
 $D^{2} \phi ( x ) = \gamma \langle x  \rangle^{\gamma -2}(I + ( \gamma -2 ) \frac x {\langle x  \rangle}\otimes \frac x{\langle x  \rangle})$. Thus  by \ref{B}, 
    \begin{align*}
        & b ( x )\cdott  D \phi ( x )  = 
        (b ( x ) - b ( 0 )) \cdott   \frac{x}{\langle x  \rangle } \gamma \langle x  \rangle^{\gamma - 1}  
        +  b ( 0 )\cdott  \frac{x}{\langle x  \rangle } \gamma \langle x  \rangle^{\gamma -1}   \\
        &\geq \gamma \langle x  \rangle^{\gamma-1}  \left( \frac{\alpha | x |^{2} }{\langle x  \rangle } - (\beta + | b ( 0 ) |)   \right),
    \end{align*}
    and
    \begin{align*}
        &\mathcal{L} \phi (x) = \mathcal{L}_{1} ( x, \phi, D \phi ( x ) ) + \mathcal{L}^{1} ( x, \phi, D \phi ( x ) ) \\
        &\leq \| D^2 \phi  \|_{\infty}\int_{| z | \leq 1 } z^2 \nu ( dz )  + \int_{| z | \geq 1 } |\phi ( x+z ) - \phi ( x )| \nu ( dz ) 
        \end{align*}
We have that  $\| D^{2} \phi  \|_{\infty} \leq C$, so that  $\mathcal{L}_{1} ( x, \phi, D \phi ( x ) ) \leq C$ for some constant $C >0$. As for  $\mathcal{L}^{1}$, we use that
 \begin{align*}
 | \phi ( x+z ) - \phi ( x )| \leq | z | \sup_{\tau \in [ 0,1 ]} | D \phi ( \tau z + x ) | \leq \gamma\, |z| \sup_{\tau \in [ 0,1 ]} \langle x+\tau z\rangle^{\gamma-1}.
 \end{align*}
If $\gamma\leq 1$, $ \langle x+\tau z\rangle^{\gamma-1}\leq 1$. If $\gamma>1$, we estimate $\langle x+\tau z\rangle^{\gamma-1}\leq c_{\gamma} (\langle x\rangle^{\gamma-1}+ \langle z\rangle^{\gamma-1})$
for some constant $c_\gamma>1$, possibly depending on $\gamma$, and for any $\tau\in (0,1)$. Hence we get
\begin{align*}
\int_{| z | \geq 1 } |\phi ( x+z ) - \phi ( x )| \nu ( dz )  & \leq c_\gamma \langle x  \rangle^{\gamma-1} \int_{ | z | \geq 1 } | z | \nu ( dz ) +c_\gamma \int_{| z | \geq 1 } \langle z  \rangle^{\gamma} \nu ( dz ) \\ & 
    \leq C ( 1 + \langle x  \rangle^{\gamma-1}  )
\end{align*}
for some constant $C > 0$, where we used assumption \ref{nu'} and that $\sigma>\max(1, \gamma)$.

Going back to $\mathcal{L}_{L_0}$, we then find, for a possibly different $C$, 
    \begin{align*}
        \mathcal{L}_{L_0} [ \phi ] ( x )&=- \mathcal{L} \phi (x) +  b ( x )\cdot  D \phi ( x ) - L_{0} |D \phi ( x ) | \\
        &\geq - C (1 + \langle x  \rangle^{\gamma -1}   ) + \gamma \langle x  \rangle^{\gamma-1}  \left( \frac{\alpha | x |^{2} }{\langle x  \rangle } - (\beta + | b ( 0 ) |+ L_0) \right) 
          \\
        &\geq    \gamma \langle x  \rangle^{\gamma-1}  \big( \alpha | x |   -  (\beta + | b ( 0 ) |+ L_0 + C ) \big)  - C.
    \end{align*}
  We conclude that \rife{Lambda1} holds for $\omega_0< \alpha\,\gamma$ and some $K>0$.
\end{proof}

The next tool is a crucial lemma which  exploits the local diffusivity of the operator in doubling variables arguments. In the case of nonlocal operators  such as   
\rife{def:nonlocal_diffusion},  this kind of argument was  developed  in \cite{barles+al1, barles+al2}. Here we use an improved version, which is needed for long time estimates; this is a slight extension of   
\cite[Lemma 3.4]{porretta2024decay}.  We recall the notation \rife{notL} used for the operator $\cL$.

\begin{lemma} \label{Estimates_non_local_operator}
Let $\mathcal L$ be defined by \rife{def:nonlocal_diffusion}, where $\nu$ satisfies  
\beq\label{coerc}
\exists \lambda >0\,  :\, \quad     \frac{d\nu}{dz}\geq \lambda \, \frac1{|z|^{d+\sigma}}
\qquad \forall z \in B_1\,.
\eeq
Suppose that $(x,y)$ is  a maximum point of the function
$$
({\tilde x}, {\tilde y})\mapsto u(\tilde x)-u(\tilde y) - \zeta(\tilde x,\tilde y)\,,\quad \zeta:= [\vfi(\tilde x)+ \vfi(\tilde y)] \psi(|\tilde x-\tilde y|) + F(\tilde x)+ F(\tilde y)
$$
where  $\vfi$ and $F$ are $C^2$ functions, $\vfi>0$,  and  $\psi:[0,\infty)\to \R_+$ is a positive, increasing and  concave function which is $C^2$  in a neighborhood of $|x-y|$. 

Then, for every $\de\geq 0$ such that $\de\leq \left(\frac{|x-y|}2 \wedge 1\right)$, we have
 \beq\label{nonlocal_precise}
 \begin{split}
&  \mathcal{L} ( x,u(x),D_{x} \zeta(x,y) ) -  \mathcal{L} ( y,u(y),-D_{y} \zeta(x,y) ) \leq  \\
& \qquad \leq \psi(|x-y|) \left\{ \mathcal{L}\vfi(x) + \mathcal{L}\vfi(y)\right\}+ \mathcal{L} F(x)+ \mathcal{L}F(y) \\
&  \quad + 
 \lambda\, C_{d,\sigma}  |x-y|^{2-\sigma} \, \psi'(|x-y|) \left( \sup_{\xi\in x+ B_\de} |D\vfi(\xi)| + \sup_{\xi\in y+ B_\de} |D\vfi(\xi)|\right)  
\\
& \quad 
+   4\lambda  \left( \inf_{\xi\in x+ B_\de} \vfi(\xi) + \inf_{\xi\in y+ B_\de} \vfi(\xi)\right)   \int_0^1 (1-s) \int_{B_\de} \psi''( \xi_{s,z}) |\widehat{x-y}\cdott z|^2  \frac{ dz}{|z|^{d+\sigma}} ds\, 
\end{split}
 \eeq
 where $\xi_{s,z}=|x-y|+ 2s(\widehat{x-y}\cdott z)$, $B_\de:= \{ z\in \R^d\,:\, |z|<\de\}$ and $C_{d,\sigma}$ is a constant only depending on $d,\sigma$.
 \end{lemma}
 
 \vskip1em
 \begin{remark} 
Despite the generality of the statement,  we mostly use it for the case of  constant $\vfi(x)$, providing with 
a simplified   conclusion where only $\psi''$ plays a role. 
 \end{remark}
 
\begin{proof}
We follow the same strategy as in  \cite[Lemma 3.4]{porretta2024decay}, where the result was proved for the case of constant $\vfi$. For the reader's convenience, we give anyway all details here. First we note that, by definition of $\zeta$,
\begin{align*}
D_x\zeta(x,y) & = D\vfi(x)  \psi(|x-y|) + DF(x) +   [\vfi(x)+ \vfi(y)] \psi'(|x-y|) \widehat{x-y} \,
\\
D_y\zeta (x,y)& = D\vfi(y)  \psi(|x-y|) + DF(y) -  [\vfi(x)+ \vfi(y)] \psi'(|x-y|) \widehat{x-y}\,,
\end{align*}
where $\widehat{x-y}:= \frac{x-y}{|x-y|} $. Moreover, from the maximality condition we have 
\beq\label{zz'}
u(x)-u(y) - \zeta(x,y) \geq u(x+z)-u(y+z')- \zeta(x+z, y+z' )\quad \forall z,z'\in \rd\,.
\eeq
Taking $z'=z$ yields
\beq\label{maj1}
\begin{split}
& [u ( x+z ) -u( x )]- [u (y+z ) -u(y )] \\
& \qquad\qquad \leq   \left\{ \vfi(  x+z ) -\vfi(  x )+  \vfi ( y+z ) -\vfi(  y )]  \right\} \psi(|x-y|)
\\
& \qquad\qquad\quad +   \left\{ F(x+z ) - F(  x )+  F(y+z ) -F(y )]  \right\} \qquad \forall z\in \R^d\,.
\end{split}
\eeq
Thanks to assumption \rife{coerc}, for small $|z|$ we can write $\nu =  \lambda \frac{dz}{|z|^{d+\sigma}}+ \mu$, where $\mu$ is a nonnegative measure (and  is itself a Levy measure).  Using that 
\beq\label{splitnu}
\nu= \nu   {1}_{B_\de^c} + \mu  {1}_{B_\de} + \lambda \frac{dz}{|z|^{d+\sigma}}  {1}_{B_\de}\,,
\eeq
we split the integrals in $\mathcal{L}$ into three terms.  For the first two measures, we only use \rife{maj1},  which implies
\beq\label{coerc1}
\begin{split}
& \int_{\R^d} \{ u (  x+z ) -u( x )  -   (D_x\zeta \cdott  z  )   {1}_{| z | \leq 1} \}  \tilde \nu( dz ) 
\\
& \qquad \qquad - \int_{\R^d} \{u ( y+z ) -u(  y ) +   (D_y\zeta\cdott  z)    {1}_{| z | \leq 1} \}\tilde \nu ( dz )
\\ & \leq  \psi(|x-y|)  \int_{\R^d} \left\{ \vfi( x+z ) -\vfi(  x )  - (D\vfi(x)\cdott z )   {1}_{| z | \leq 1}\right\}\tilde \nu(dz)  
\\
& \quad +  \psi(|x-y|)  \int_{\R^d} \left\{ \vfi(  y+z ) -\vfi(  y )  - (D\vfi(y)\cdott z)   {1}_{| z | \leq 1} \right\}\tilde \nu(dz) 
\\ & \quad  +    \int_{\R^d} \left\{ F(  x+z ) -F(  x )  - (DF(x)\cdott z )   {1}_{| z | \leq 1}\right\}\tilde \nu(dz)  
\\
& \quad +    \int_{\R^d} \left\{ F(  y+z ) -F(  y )  - (DF (y)\cdott z)   {1}_{| z | \leq 1} \right\}\tilde \nu(dz) 
\end{split}
\eeq
applied for both $\tilde \nu= \nu   {1}_{B_\de^c}$ and  $\tilde \nu= \mu   {1}_{B_\de}$. 
In the remaining term, we use \rife{zz'} with $z'=Az$, where $A$ is the matrix defined in $\R^d$ as
$$
A:= I_d- 2 (\widehat{x-y}\otimes \widehat{x-y})\,.
$$
Hence we have  
\beq\label{maj2}
u(x)-u(y) - \zeta(x,y) \geq u(x+z)-u(y+Az)- \zeta(x+z, y+Az )\,.
\eeq
Since $A$ is an orthogonal matrix, we have $|Az|=|z|$; then, using the rotational invariance of the  fractional Laplacian kernel, 
we have
\beq\label{rotid}
\begin{split}
& \int_{B_\de} \left\{ u (  x+z ) -u(  x )  -    D_x\zeta \cdott  z\right\}   \frac{dz}{|z|^{d+\sigma}} -
\int_{B_\de} \{u (  y+ z ) -u(  y ) +   D_y\zeta \cdott   z   \} \frac{dz}{|z|^{d+\sigma}} 
\\
&= \int_{B_\de} \left\{ u (  x+z ) -u(  x )  -    D_x\zeta \cdott  z\right\}   \frac{dz}{|z|^{d+\sigma}} -
\int_{B_\de} \{u (  y+Az ) -u(  y ) +   D_y\zeta \cdott A z   \} \frac{dz}{|z|^{d+\sigma}} \,.
\end{split}
\eeq
Now we develop \rife{maj2} with our definition of $\zeta$, which yields
\beq\label{maj2bis}
\begin{split}
& [u ( x+z ) -u(  x )]- [u ( y+Az ) -u(  y )] \\
&  \leq  \psi(|x-y|) [ \vfi(  x+z )  - \vfi(x)+  \vfi ( y+Az ) - \vfi(y)]  
\\
& \quad + 
  [\psi(|x-y + z-Az|) -\psi(|x-y|)] \, [ \vfi(  x+z )   +  \vfi ( y+Az ) ] 
\\
& \qquad\qquad\quad +    \left\{ F(x+z ) -F(x )+  F(y+Az ) -F(y )]  \right\} \qquad \forall z\in \R^d\,.
\end{split}
\eeq
Using \rife{maj2bis} in the right-hand side of \rife{rotid} we estimate
\begin{align*}
& \int_{B_\de} \left\{ u (  x+z ) -u(  x )  -    D_x\zeta \cdott  z\right\}   \frac{dz}{|z|^{d+\sigma}} -
\int_{B_\de} \{u (  y+ z ) -u(  y ) +   D_y\zeta \cdott   z   \} \frac{dz}{|z|^{d+\sigma}} 
\\ &
\leq   \psi(|x-y|) \left\{ \int_{B_\de}   \left\{ \vfi(  x+z ) -\vfi(  x )  - D\vfi(x)\cdott z  \right\} \frac{dz}{|z|^{d+\sigma}}  \right. 
\\ & \qquad\qquad\qquad\quad  \quad \left. +    \int_{B_\de} \left\{ \vfi(  y+Az ) -\vfi(  y )  - D\vfi(y)\cdott Az  \right\} \frac{dz}{|z|^{d+\sigma}}  \right\}
\\ & 
 +   \int_{B_\de} \left\{ F(x+z ) -F(x )  - DF(x)\cdott z  \right\} \frac{dz}{|z|^{d+\sigma}} 
+    \int_{B_\de} \left\{ F(y+Az ) -F(y )  - DF(y) \cdott Az  \right\} \frac{dz}{|z|^{d+\sigma}} 
\\ & 
\quad +    \int_{B_\de}   \, [ \vfi(  x+z )   -\vfi(x)  +  \vfi ( y+Az )-\vfi(y)  ]   \psi'(|x-y|) \widehat{x-y}\cdott(z-Az) \frac{dz}{|z|^{d+\sigma}} 
\\ 
&
 \quad +      \int_{B_\de} \left\{  \psi(|x-y + z-Az|)- \psi(|x-y|) -  \psi'(|x-y|) \widehat{x-y}\cdott(z-Az) \right\} {\mathcal R}_\vfi \frac{dz}{|z|^{d+\sigma}} \,.
\end{align*}
where $ {\mathcal R}_\vfi := [ \vfi(  x+z )   +  \vfi ( y+Az ) ] $.
\vskip0.4em

We multiply the above inequality by $\lambda$ and we combine it with \rife{coerc1} used for  both $\tilde \nu= \nu   {1}_{B_\de^c}$ and  $\tilde \nu= \mu   {1}_{B_\de}$. 
Recalling \rife{splitnu} we have proved that
 \beq\label{alltoge}
 \begin{split}
  &    \qquad  \mathcal{L}   ( x,u(x),D_{x} \zeta(x,y) ) -  \mathcal{L} ( y,u(y),-D_{y} \zeta(x,y) )    \leq  \\
  & \quad  \qquad   \quad  \leq  \psi(|x-y|) \left\{\mathcal{L} \vfi(x)  + \mathcal{L}  \vfi(y)  \right\}+ \mathcal{L}  F(x)  + \mathcal{L}  F(y) 
 \\ & 
  +    \int_{B_\de}   \, [ \vfi(  x+z )   -\vfi(x)  +  \vfi ( y+Az )-\vfi(y)  ]   \psi'(|x-y|) \widehat{x-y}\cdott(z-Az) \frac{\lambda \,dz}{|z|^{d+\sigma}} 
 \\ 
&
  +     \int_{B_\de} \left\{  \psi(|x-y + z-Az|)- \psi(|x-y|) -  \psi'(|x-y|) \widehat{x-y}\cdott(z-Az) \right\}  {\mathcal R}_\vfi  \frac{\lambda\, dz}{|z|^{d+\sigma}} \,.
\end{split}
 \eeq
Now we estimate last two terms. To this goal we notice that the precise choice of $A$ implies
$$
x-y + z-Az= \widehat{x-y} \left(  |x-y| + 2 (\widehat{x-y}) \cdott z  \right)
$$
and moreover, since $|z|\leq \de \leq (\frac{|x-y|}2\wedge 1)$, we have  $0\leq  |x-y| + 2 (\widehat{x-y})\cdott z$.
Hence we estimate, by Taylor's expansion,
\begin{align*}
&   \psi(|x-y + z-Az|)- \psi(|x-y|) -  \psi'(|x-y|) \widehat{x-y}\cdott(z-Az) 
\\
& =    \psi(|x-y| + 2 (\widehat{x-y}\cdott z))- \psi(|x-y|) -  \psi'(|x-y|) 2(\widehat{x-y}\cdott z)  
\\
& = 4 \int_0^1 (1-s) \left( \psi''( |x-y|+ 2s(\widehat{x-y}\cdott z)) |\widehat{x-y}\cdott z|^2\right)ds
\end{align*}
Notice that this latter term is negative since $\psi$ is concave. Using that 
$$
\vfi(  x+z )   +  \vfi ( y+Az )  \geq \inf_{\xi\in x+ B_\de} \vfi(\xi) + \inf_{\xi\in y+ B_\de} \vfi(\xi) \quad \forall z\in B_\de
$$
we estimate
\begin{align*}
& \int_{B_\de} \left\{  \psi(|x-y + z-Az|)- \psi(|x-y|) -  \psi'(|x-y|) \widehat{x-y}\cdott(z-Az) \right\}  {\mathcal R}_\vfi \frac{\lambda\, dz}{|z|^{d+\sigma}} 
\\ & 
\leq  \left( \inf_{\xi\in x+ B_\de} \vfi(\xi) + \inf_{\xi\in y+ B_\de} \vfi(\xi)\right)  \int_0^1 (1-s) \int_{B_\de} \psi''( |x-y|+ 2s(\widehat{x-y}\cdott z)) \frac{|\widehat{x-y}\cdott z|^24\lambda }{|z|^{d+\sigma}} dzds\,.
\end{align*}
Similarly we estimate
\begin{align*}
 & \int_{B_\de}   \, [ \vfi(  x+z )   -\vfi(x)  +  \vfi ( y+Az )-\vfi(y)  ]   \psi'(|x-y|) \widehat{x-y}\cdott(z-Az) \frac{\lambda \,dz}{|z|^{d+\sigma}}  \\
 & \quad  \leq   \left( \sup_{\xi\in x+ B_\de} |D\vfi(\xi)| + \sup_{\xi\in y+ B_\de} |D\vfi(\xi)|\right) \psi'(|x-y|)
 2\lambda   \int_{B_\de} |z|^2 \frac{dz}{|z|^{d+\sigma}} 
 \\
 &\quad  \leq \lambda\, C_{\sigma, d} \left( \sup_{\xi\in x+ B_\de} |D\vfi(\xi)| + \sup_{\xi\in y+ B_\de} |D\vfi(\xi)|\right) \psi'(|x-y|)|x-y|^{2-\sigma}\,.
\end{align*}

We conclude therefore the estimate from \rife{alltoge}, obtaining \rife{nonlocal_precise}.
\end{proof}

Now we prove the global (in time) H\"older and Lipschitz estimates for the Hamilton-Jacobi equation \eqref{eqn:parabolic_HJB}. This is a   regularity and stability  result   coming from both the ellipticity of $\mathcal L$ \ref{nu'} and the confining drift assumption \ref{B}. 

\begin{proposition}\label{lipest} Assume \ref{nu'}, \ref{B}, \ref{H1}, and $f \in C_b (Q_T)$.
Let $u \in  C ( Q_{T} ) \cap L^{\infty} (0,T; L^{\infty} (\langle x \rangle^{-\beta}))$ (for some $\beta<\sigma$) be  a viscosity solution of \eqref{eqn:parabolic_HJB}. 

If $u_0 \in C^{0,\theta}(\R^d)$ for some $\theta\in (0,1]$, then  $u(t) \in C^{0,\theta}(\R^d)$ for $t>0$, and  there is a constant $C>0$ independent of $T>0$, such that
\begin{align*}
        | u  ( t,x ) - u  ( t,y ) | 
        \leq C | x-y |^\theta, \qquad x,y \in \R^{d}, t \in  [ 0,T ),
\end{align*}
where $C$ depends on $\theta, d, \sigma, \alpha, \beta, |b(0)|, C_H$ and on $[u_0]_\theta$ and $C_f= \sup\limits_{t,x,y}[ f(t,x)-f(t,y)]$.
\end{proposition}

\begin{proof}   We proceed in two steps. First we prove  H\"older  estimates for $u$, and then the Lipschitz estimate afterwards. 

Let $\phi(x)= \langle x\rangle^{\tilde \beta}$ for some $ \tilde \beta \in ((\beta\vee 1), \sigma)$, so that $u(t,x)=o(\phi(x))$ as $|x|\to \infty$, uniformly in time.
We define
\begin{align}
    \Phi ( t,x,y ) = u ( t, x ) - u ( t, y )  -  \psi ( |x-y| ) - \varepsilon ( \phi ( x ) + \phi ( y ) ) -\frac \vep{T-t},
    \label{defun:Phi}
\end{align}
where  $\psi$ is an increasing  concave  function, to be defined later. 

Our aim is to show that $\Phi ( t,x,y ) \leq 0$ for all $x,y \in \R^{d}$, $t<T$, and for every  $\vep$ sufficiently small.

{\it Step 1.} \quad   The case $\theta\in(0,1)$.   Here we choose in \rife{defun:Phi}  
$$
\psi ( r ) := C_1(1 - e^{- C_{2} r^{\theta}} )+ L r^\theta 
$$
where $C_1, C_2$ will be large constants to be fixed later, and 
\beq\label{lgeq}
L\geq [u_0]_\theta\,.
\eeq
Assume by contradiction that, for some $\varepsilon > 0$, 
$\sup_{x,y \in \R^d, t\in (0,T) } \Phi ( t,x,y ) > 0$. 
Since $u= o(\phi)$ as $|x|\to \infty$, 
the supremum of $\Phi$ is achieved at some point $ ( t, x , y  )$ with $t<T$, and,  since it is a positive maximum,  it occurs with $x \neq y$ and $t\neq 0$, due to  \rife{lgeq}. 

We denote 
\begin{align}
  \zeta   ( x,y ) :=    \psi ( |x-y| ) + \varepsilon ( \phi ( x ) + \phi ( y ) ),
    \label{defun:varphi}
\end{align}
and we denote hereafter $\widehat{x-y}=\frac{x-y}{| x-y | }$, $r:=|x-y|$. Hence we have
\begin{align} \label{defun:phi_derivative}
    D_{x} \zeta =   \psi' (r)\, \widehat{x-y} + \varepsilon D \phi ( x ) \qquad\text{and}\qquad D_{y} \zeta = -   \psi'(r) \widehat{x-y} + \varepsilon D \phi ( y ).
\end{align}
Using the parabolic version of the Theorem of sums in viscosity solutions' theory (see e.g. \cite[Thm 8.3]{CIL},   \cite[Theorem 2.2]{JaKa05},   and \cite[Corollary 2]{barles2008second}) yields
%
\begin{align} 
    \label{eqn:viscosity_inequalities_1}
    &  \frac \vep{(T-t)^2} - \mathcal{L} ( x, u(x), D_{x} \zeta ) + \mathcal{L} ( y, u(y), -D_{y} \zeta ) \\
\notag &  +   b ( x )\cdott  D_{x} \zeta  -
 b ( y )\cdott ( - D_{y} \zeta )  
+ H ( x, D_{x} \zeta ) - H ( y, -D_{y} \zeta ) \\
\notag &\leq f ( t,x ) - f (t,  y ).
\end{align} 
We estimate the terms in \eqref{eqn:viscosity_inequalities_1}. 
Due to  \ref{B}, \ref{H1} and \eqref{defun:phi_derivative},  we have
\begin{align}
  b ( x )\cdott  D_{x} \zeta  -
 b ( y )\cdott ( - D_{y} \zeta )  
    &\geq     \psi' (r)  (\alpha\, r-\beta)+ 
    \varepsilon (  b ( x ) \cdott D \phi ( x )   + b ( y )\cdott D \phi ( y )    ),
    \label{eqn:dummytag2}\\
    H ( x, D_{x} \zeta ) - H ( y, -D_{y} \zeta ) & \geq
     - C_{H} [ 2+ 2  \psi' (r)+ \varepsilon ( | D \phi ( x ) | + | D \phi ( y ) |   ) ].\nonumber 
\end{align}
Inserting these estimates into \eqref{eqn:viscosity_inequalities_1} leads to
\begin{align} 
    \label{eqn:viscosity_inequalities_2}
   &  \frac \vep{(T-t)^2}  +    \psi' (r)( \alpha \, r -(\beta+ 2C_H) )    - \mathcal{L} ( x,  u, D_{x} \zeta ) + \mathcal{L} ( y, u, -D_{y} \zeta ) \\
    &\quad \leq 2 C_{H}      +  C_f + \varepsilon \Big( C_{H } ( | D \phi ( x ) | + | D \phi ( y ) |   )- b ( x ) \cdott D \phi ( x )   - b ( y )\cdott D \phi ( y )  \Big)  \notag
\end{align}
where $C_f= \sup\limits_{t,x,y}[ f(x)-f(y)]$.

Let us first consider the case that $r > \frac{2(\beta+2C_H)}\alpha$, so that $\alpha \, r-(\beta+ 2C_H)>\frac \alpha2\, r$. 
Using Lemma \ref{Estimates_non_local_operator}  (with $\vfi=\frac12$), $\psi''(r)<0$,   and   $\psi'(r)\geq \theta L r^{\theta-1}$, \rife{eqn:viscosity_inequalities_2} implies
\begin{align*} 
   &  \frac \vep{(T-t)^2}  +  \frac\alpha2\theta\, L \,  r^\theta  \leq 2 C_{H}  +  C_f   - \varepsilon (\mathcal{L} \phi( x)  + b ( x ) \cdott D \phi ( x ) - C_H | D \phi ( x ) | ) 
     \\
     & \qquad \qquad - \varepsilon (  \mathcal{L} \phi( y)  + b ( y )\cdott D \phi ( y ) - C_H | D \phi ( y ) | ))\,.
\end{align*}
Assume that $r\geq r_1> \frac{2(\beta+2C_H)}\alpha$; this implies $|x|\vee|y| \geq r_1/2$, so one between $|x|$ and $|y|$ can be assumed to be large. This means, due   to Lemma \ref{lemma:supersolution_property}, that we can choose $r_1$ sufficiently large so that
\begin{align*}
& (\mathcal{L}\phi ( x)   + b ( x ) \cdott D \phi ( x ) - C_H | D \phi ( x ) | ) 
 \\ & \qquad\qquad  + ( \mathcal{L} \phi( y) + b ( y )\cdott D \phi ( y ) - C_H | D \phi ( y ) | )) \geq 0 \qquad \forall\, x,y: \, |x|\vee|y| \geq r_1/2\,.
\end{align*}
Then we deduce
$$
\frac \vep{(T-t)^2}  +  \frac\alpha2\theta\, L \, \left(\frac{4C_H}\alpha\right)^\theta  \leq 2 C_{H}  +C_f \qquad \forall r\geq r_1
$$
which cannot hold if $L$ is sufficiently large, only depending on $ C_f , C_H, \alpha, \theta$. 

We are left with the case that $r<r_1$ and we go back to \rife{eqn:viscosity_inequalities_2}; here we use again Lemma \ref{Estimates_non_local_operator} (with $\vfi=\frac12$) and Lemma \ref{lemma:supersolution_property} (which implies that 
$\cL_{C_H}\phi \geq -K$ for some $K>0$), and we obtain
\beq\label{hol1}
\begin{split}
&  \frac \vep{(T-t)^2}  - 4\lambda   \int_0^1 (1-s)\int_{B_\de}   \psi'' ( \xi_{s,z}) |\widehat{x-y}\cdott z|^2 \frac{dz}{|z|^{d+\sigma}} ds\\
    &\quad \leq 2 C_{H}      +  C_f  + 2\vep K+ (\beta +2C_H )\, \psi'(r)\;,
    \end{split}
    \eeq
where we have denoted $\xi_{s,z}=  r+ 2s(\widehat{x-y}\cdot z)$.  We choose $\de= (r/4) \wedge 1$, and we estimate   (recall that $\psi''<0$) 
\beq\label{psi1}
 \int_0^1 (1-s)\int_{B_\de}   \psi'' (\xi_{s,z}) |\widehat{x-y}\cdot z|^2 \frac{dz}{|z|^{d+\sigma}} ds   \leq  \int_0^1 (1-s)\int_{B_\de^-}   \psi'' (\xi_{s,z}) |\widehat{x-y}\cdot z|^2 \frac{dz}{|z|^{d+\sigma}} ds
\eeq
where $B_\de^-= \{z\in B_\de\,:\, \widehat{x-y}\cdot z\leq 0\}$.  Since $B_\de^- $ is half of the ball, we have 
\beq\label{halfball}
\int_{B_\de^-}  |\widehat{x-y}\cdot z|^2 \frac{dz}{|z|^{d+\sigma}}  \geq c_{d,\sigma} \de^{2-\sigma}
\eeq
for a constant $c_{d,\sigma}$ only depending on $d, \sigma$. This allows us to estimate
\beq\label{psi2}
\begin{split}
\psi'(r) & \leq \frac2{c_{d,\sigma}} \int_0^1 (1-s)\int_{B_\de^-} \frac{\psi'(r)}{\de^{2-\sigma}}  |\widehat{x-y}\cdot z|^2 \frac{dz}{|z|^{d+\sigma}}ds   
\\
& \leq C_{d,\sigma} \int_0^1 (1-s)\int_{B_\de^-} \frac{\psi'(\xi_{s,z}) }{(\xi_{s,z} \wedge 1)^{2-\sigma}}  |\widehat{x-y}\cdot z|^2 \frac{dz}{|z|^{d+\sigma}}ds
\end{split}
\eeq
where we used the monotonicity of $\psi'$ and the fact that $r/2\leq \xi_{s,z}\leq r$ for $z$ in $B_\de^-$.
Here and after  $C_{d,\sigma}$ denotes  a possibly different constant only depending on $d, \sigma$. Collecting \rife{psi1}--\rife{psi2} we get
\beq\label{prehol1}
\begin{split}
& 4\lambda   \int_0^1 (1-s)\int_{B_\de}   \psi'' (\xi_{s,z}) |\widehat{x-y}\cdot z|^2 \frac{dz}{|z|^{d+\sigma}} ds+ (\beta+ 2C_H) \, \psi'(r)
\\
& \leq  \int_0^1 (1-s)\int_{B_\de^-} \left\{ 4\lambda \psi''(\xi_{s,z})+ c \frac{\psi'(\xi_{s,z})}{(\xi_{s,z}\wedge 1)^{2-\sigma}}\right\} |\widehat{x-y}\cdot z|^2 \frac{dz}{|z|^{d+\sigma}}ds
\end{split}
\eeq
where  $c$ is a constant depending on $\beta, C_H, d, \sigma$. Now we specify to our choice of $\psi$, and we compute  for $\frac r2\leq \xi\leq r\;(<r_1)$, 
\begin{align*}  
4\lambda \psi''(\xi)+ c \frac{\psi'(\xi )}{(\xi\wedge 1)^{2-\sigma}}  
& = - C_1 \,  \theta \, C_2\, e^{-C_2\xi^\theta}\xi^{\theta-2} 
\left\{4\lambda(1-\theta)+ 4\lambda \theta C_2 \xi^\theta -  c (\xi^{\sigma-1}\vee \xi) \right\}  
\\
& \qquad -  \theta  L \xi^{\theta-2} \left\{ 4\lambda(1-\theta) -  c\, (\xi^{\sigma-1}\vee \xi)\right\}\,.
\end{align*}
Here we first fix $C_2$ sufficiently large so that  $4\lambda \theta C_2 \xi^\theta -  c (\xi^{\sigma-1}\vee \xi) \geq 0$ for any $r\in (0,r_1)$. 
Having fixed $C_2$ this way, we get
\begin{align*}
4\lambda \psi''(\xi)+ c \frac{\psi'(\xi )}{(\xi\wedge 1)^{2-\sigma}}  & \leq  - 4\lambda(1-\theta) C_1 \,  \theta \, C_2\, e^{-C_2\xi^\theta}\xi^{\theta-2}  + \theta  L \xi^{\theta-2} c\, (\xi^{\sigma-1}\vee \xi)
\\
& \leq   \theta    \xi^{\theta-2}  ( - 4\lambda(1-\theta) C_1  \, C_2\, e^{-C_2r_1^\theta}+ c\, L r_1)
\end{align*}
%
since $\xi \leq r_1$.  
Finally, we choose $C_1$ large enough (depending on $L, C_2, r_1$ and the other parameters) so that last term is  negative  
and we get, for some $\vep _0>0$,
$$
4\lambda \psi''(\xi)+ c \frac{\psi'(\xi )}{(\xi\wedge 1)^{2-\sigma}}   \leq -\vep_0\, C_1 \,  \xi^{\theta-2} \,. 
$$
We insert now this estimate in \rife{prehol1} (where, we recall, $r/2\leq \xi_{s,z}\leq r$) and we deduce   
\begin{align*}
&  4\lambda   \int_0^1 (1-s)\int_{B_\de}   \psi'' ( \xi_{s,z}) |z|^2 \frac{dz}{|z|^{d+\sigma}} ds+ (\beta + 2C_H) \, \psi'(r)
\\
& \leq  -  \vep_0\, C_1 \, (r/2)^{\theta-2} \int_0^1 (1-s)\int_{B_\de^-}  |\widehat{x-y}\cdot z|^2 \frac{dz}{|z|^{d+\sigma}}ds  
\\
& \leq - \tilde \vep_0\, C_1\, r^{\theta-\sigma} 
\end{align*}
for some $\tilde \vep_0$ possibly depending on all parameters but not on  $C_1$.  Plugging this estimate in \rife{hol1} we finally get
$$
\frac \vep{(T-t)^2}   +   \tilde \vep_0\, C_1\, r^{\theta-\sigma} \leq 2 C_{H}      +  C_f  + 2\vep K\,,\qquad \forall r\leq r_1
$$
which  since $\theta-\sigma<0$   is impossible for a possibly  larger choice of $C_1$.  The contradiction  proves that 
\begin{align*}
    u ( t, x ) - u (t, y ) & \leq  \psi ( |x-y| ) + \varepsilon ( \phi ( x ) + \phi ( y ) )+\frac \vep{T-t}
    \\
    & \leq (C_1C_2 + L)|x-y|^\theta+ \varepsilon ( \phi ( x ) + \phi ( y ) )+\frac \vep{T-t}
\end{align*}
which after sending $\varepsilon\to0$ gives  the desired estimate, and the conclusion for $\theta\in (0,1)$.

\vskip1em

{\it Step 2.} \quad  In order to get the Lipschitz estimate ($\theta=1$ in the statement of the Proposition), we   start with the following remark. We notice  that, if $u_0$ satisfies
 $$
 | u_0 ( x ) - u_0 ( y ) | \leq K \left(| x-y |^{\theta}\vee|x-y|\right) , \quad x,y \in \R^d , 
 $$
 then $u(t,x)$ satisfies a similar estimate, uniformly in time. This is obtained as in Step 1 by modifying $\psi$ into $\psi(r)=  C_1(1 - e^{- C_{2} r^{\theta}} )+ L r$.
 As before, the term with $L$ is used for $r\geq r_1$ and is harmless for $r<r_1$. This remark says that, if $u_0$ is Lipschitz, we already know that
\beq\label{inter}
 | u ( t,x ) - u  ( t,y ) | \leq K  \, (1+ |x-y|) \qquad \forall t,x,y\,.  
\eeq
Hence, we only need to prove an estimate for small $|x-y|$.

To this purpose,  we  choose in \rife{defun:Phi} $\psi$ as
\begin{align}
\psi ( r ) = C_1(r-\rho r^{1+\theta}) \,,\quad r\in [0, r_0], \,\,\hbox{with $r_0^\theta= \frac1{2\rho}$.}
    \label{defun:psi_2}
\end{align}
for some $\rho>0$ and  $\theta\in(0,1)$ to be fixed later, and we claim that   
\beq\label{phirr0}
\sup_{0<r<r_0}  \Phi(r) \leq 0\,.
\eeq
We observe that $\psi''\leq0$, and  by choice of $r_0$,  we have
$$
\psi(r) \geq C_1 r (1-\rho r_0^\theta)\geq \frac 12 C_1 r \quad  \text{and}\quad\psi'(r)\geq C_1(1-(1+\theta)\rho r_0^\theta) 
>0\quad   \forall r\leq r_0\,.
$$
In particular, due to \rife{inter} and the fact that $u_0$ is Lipschitz, if we choose $C_1$ conveniently large  we have $\Phi\leq 0$ for $t=0$ and for $r=r_0$. This means that, if  \rife{phirr0} does not hold, then $\Phi$ has a positive local maximum, attained at some point $(t,x,y)$ with $t\in (0,T)$ and $x\neq y$, $|x-y|<r_0$.

Then we proceed as in Step 1;  using the viscosity inequalities we obtain \rife{hol1} and  \rife{prehol1}. Namely, with the same notation as before, we get 
\beq\label{lip1}
\begin{split}
&  \frac \vep{(T-t)^2}  \leq 2 C_{H}      +  C_f  + 2\vep K
\\
& \quad + \int_0^1 (1-s)\int_{B_\de^-} \left\{ 4\lambda \psi''(\xi_{s,z})+ c\, \frac{\psi'(\xi_{s,z})}{(\xi_{s,z}\wedge 1)^{2-\sigma}}\right\} | \widehat{x-y}\cdot z|^2 \frac{dz}{|z|^{d+\sigma}}ds
\end{split}
\eeq
with $r=|x-y|$ and $\de= (r/4) \wedge 1$, where $\xi_{s,z}=  r+ 2s(\widehat{x-y}\cdot z)$ satisfies  $r/2\leq \xi_{s,z}\leq r$.  When we specialize to our choice of $\psi$, we have:
\begin{align*}
\quad & 4\lambda \psi''(\xi )+ c\,  \frac{\psi'(\xi )}{(\xi \wedge 1)^{2-\sigma}}  =    \xi^{\theta-1}C_1 \left[ -  4\lambda \theta\rho (1+\theta) + c\,  (  
\xi^{\sigma-1-\theta}\vee \xi^{1-\theta})(1-(1+\theta)\rho \xi^\theta)  \right].
\end{align*}
Hence, for $\theta< \sigma-1$,  choosing $\rho$ sufficiently large we deduce
$$
4\lambda \psi''(\xi )+ c\,  \frac{\psi'(\xi )}{(\xi \wedge 1)^{2-\sigma}}\leq  - \vep_0\, C_1\,  \xi^{\theta-1}
\qquad \forall \xi\leq r_0
$$
for some constant $\vep_0>0$ (independent of $C_1$). Therefore we deduce from \rife{lip1}
\begin{align*}
\frac \vep{(T-t)^2} & \leq 2 C_{H}      +  C_f  + 2\vep K - \vep_0\, C_1\, \int_0^1 (1-s)\int_{B_\de^-} \xi_{s,z}^{\theta-1}| \widehat{x-y}\cdott z|^2 \frac{dz}{|z|^{d+\sigma}}ds
\\ &  \leq 2 C_{H}      +  C_f  + 2\vep K -  \tilde \vep_0 \, C_1 \, r^{\theta+1-\sigma}   \qquad   \forall r\leq r_0\,,
\end{align*}
for a possibly different constant $ \tilde \vep_0$, where we used that $r/2\leq \xi_{s,z}\leq r$ and \rife{halfball}. 
Since $\theta< \sigma-1$,  the above inequality implies 
$$
\frac \vep{(T-t)^2}\leq 2 C_{H}      +  C_f  + 2\vep K -  \tilde \vep_0 \, C_1 \, r_0^{\theta+1-\sigma} 
$$
and  we get a contradiction choosing $C_1$ sufficiently large. Hence we deduce that \rife{phirr0} must hold true. Letting $\vep \to 0$ in the definition of $\Phi$, we proved that
\begin{align*}
    u ( t, x ) - u (t, y ) \leq  C_1\, |x-y|   \qquad \forall x,y,t\,: |x-y|<r_0
\end{align*}
which, jointly with \rife{inter},  concludes with the desired Lipschitz estimate.
\end{proof}


In a similar way,  we obtain an oscillation bound for the truncated problem \eqref{eqn:parabolic_HJB_cutoff}, which will be used in   Section \ref{HJB}. 

\begin{thm} \label{thm:parabolic_Lipschitz_cutoff}
  Assume \ref{nu'}, 
  \ref{B}, \ref{H1}, and  $f \in C_b (Q_T)$, 
    and  $u_{0} \in W^{1,\infty} ( \R^d ) $. 
    Let $\vep\in(0,1)$, $R>1$, 
    and $u_{\vep,R} \in  C_b ( Q_{T} )$ be a viscosity solution of  \eqref{eqn:parabolic_HJB_cutoff} such that 
    \begin{align}\label{bnd-soln}
    \|u_{\vep,R}\|_{L^\infty(Q_T)}\leq \|u_0\|_{\infty} + \tilde C T \qquad \text{for} \qquad \tilde C>0 \  \text{independent of $\vep,R$.}
    \end{align}
 
Then for $ \gamma  \in (1, \sigma)$,  there  exists a constant $C>0$, independent of $\vep,T,R$, and a constant $c_T$ (independent of $\vep,R$ but possibly depending on $T$) such that
\begin{align*}
 |  u_{\vep,R} ( t,x ) - u_{\vep,R}  ( t,y ) |   \leq C | x-y | + \frac{c_T}{\langle R \rangle^{\gamma}} ( \langle x \rangle^{\gamma} + \langle y \rangle^{\gamma} ), \qquad x ,y \in B_{R}, t \in  [ 0,T ).
    \end{align*}
\end{thm}
\nc
\begin{proof}    
 As  in the  proof of Proposition \ref{lipest},   we consider the function
 \begin{align*}
 \Phi_{ \vep, R} (t,x,y) := u_{ \vep,  R} (t,x) - u_{ \vep,  R} (t,y) - C \psi (x-y) - \epsilon_R ( \phi (x) + \phi (y)) - \frac{\tilde\vep}{(T-t)^2},
\end{align*}
    where $\phi(x)=\langle x\rangle^{\gamma}$ with $ \gamma<\sigma$, and where 
    $$
    \epsilon_R := \frac{2 }{\langle R \rangle^{ \gamma }}(\|u_0\|_{\infty} + \tilde C T).
    $$ 
This choice of $ \epsilon_R $ implies that $\Phi_{ \vep,  R}\leq 0$ for  $(x,y) \not \in [B(0,R)\times B(0,R)]^c$.  Hence, if  $\Phi_{ \vep,  R}$ has a positive maximum, this is attained at points 
$(t,x,y)$ with $x,y \in B(0,R)$; since we have   $b_{\vep,R}=b_\vep$   on $B(0,R)$, we can follow the proof of  Proposition \ref{lipest}. We obtain that $\Phi_{ \vep,  R}\leq 0$, for the choice of $\psi$ used before.  Letting $ \tilde\vep  \to 0$ we conclude with the desired estimate. 
\end{proof}

We now give a local Lipschitz estimate for   the Hamilton-Jacobi equation   
\rife{eqn:parabolic_HJB}. This is now independent of the behavior at infinity   (hence independent from the confining drift assumption \ref{B}),   so we shorten notations, including both the drift term and the right-hand side in the Hamiltonian. 

\begin{thm}\label{thm:parabolic_local_Lipschitz_cutoff}
Assume \ref{nu'}, and let $\hat H(t,x,p)$ be a   continuous   
function  satisfying
\beq\label{hath}
|\hat H(t,x,p)| \leq \hat C(1+ |p|)\,, \qquad \forall (t,x,p)\in (0,T)\times \Omega \times \R^d\,.
\eeq
Let  $u$ be  a viscosity solution of the problem
$$
    \begin{cases}
        \partial_t u - \mathcal{L} u (x) + \hat H ( t, x , Du ) =0 , &\text{in } ( 0,T ) \times \Omega,  \\
        u ( 0, \cdot ) = u_{0}, & \text{in } \Omega,  
    \end{cases}
$$
where $\Omega\subseteq \R^d$ is an  open domain in $\R^d$. Let   $B $ be a compact convex domain contained in $\Omega$. 

Then there  exist constants $ C,\delta>0$,  such that
\beq\label{lipB}
        | u  ( t,x ) - u  ( t,y ) | \leq C (1+\tilde\omega_{\de, B}(u))  | x-y | \qquad \hbox{for $x,y \in B$},
    \eeq
    where $C$ depends on $\frac1t,d,\sigma$, and $\hat C$; and 
$$
\tilde\omega_{\de, B}(u):= \sup_{\scriptsize\begin{array}{c}s\in(\frac t2 , t),x,y \in  B_{2},\\|x-y|<\delta \end{array}}\!\!\!\!  |u  (s,x) - u  (s,y)|\,,
$$ 
where $B_2=\{x: \textup{dist}(x,B)\leq 2\}$ and $\de\leq 1$ only depends on $d,\sigma, \hat C$. 
\vskip0.4em
In addition, if $u_0$ is locally Lipschitz continuous, the estimate holds up to  $t=0$ and we have  
\beq\label{lipB2}
       \sup_{t\in(0,T)} | u  ( t,x ) - u  ( t,y ) | \leq C (1+\omega_{\de, B,T}(u)) | x-y |, \qquad \hbox{for $x,y \in B$},
    \eeq
where the constant $C$ in \rife{lipB2}  only depends on $d,\sigma, \hat C, Lip_{B_1}(u_0)$; and  
$$
\omega_{\de, B,T}(u):= \sup_{\scriptsize\begin{array}{c}s\in[0,T],x,y \in  B_{2},\\|x-y|<\delta \end{array}}\!\!\!\!  |u  (s,x) - u  (s,y)|\,.
$$\nc
\end{thm}

\begin{proof} The proof is inspired by \cite[Theorem 4.19]{porretta2013global}, and many details are similar to  Step 2 of Proposition \ref{lipest}; however,  in this proof we shall not exploit the  confining effect of the drift, but only the local regularizing effect of the diffusion.
 
\smallskip

 {\it Step 1:}\quad Let $\delta > 0$, 
 $x_{0}\in B$, $t_{0}\in(0,T)$, and 
 define 
     \begin{align*}
         \Delta = \big\{ ( t,x,y ) \in ( 0,T ) \times B( x_{0},\delta ) \times B( x_{0}, 2 \delta):
         | x-y | < \delta, \ \frac{t_{0}}{2} < t <   t_{0}   \wedge T \big\},
    \end{align*}
    and the function
    \begin{align*}
        \Phi_{\epsilon} ( t,x,y ) = &\ u  ( t,x ) - u  ( t,y ) - K \psi ( |x-y| ) - L \phi ( x-x_{0} )
        - 
        C_{0} ( t-t_{0} )^{2} - \frac{\varepsilon}{t_0-t},
    \end{align*}
    where $K$, $L$, and $C_{0}$ are constants to be determined, and 
    $\phi (x) = \langle x \rangle^{\beta}-1$ with $1<\beta < \sigma$.  Note that $B(x_0,2\delta)\subset   B_2 $ 
    and that $\phi\geq0$ is  smooth and radially increasing with a unique strict minimum $\phi(0)=0$, so there is $c_\phi>0$ such that $\phi(x)\geq c_\phi|x|^2$ for $|x|<1$. 
    As before, we take $\theta \in ( 0, 1)$ sufficiently small and $\rho >0 $ sufficiently large (to be fixed below in dependence of $\theta$)  and we consider 
    \begin{align*}
        \psi ( r) =r - \rho r^{1+ \theta}, 
    \end{align*}
    which is increasing and concave for $r \leq \big( \frac{1}{\rho (1+ \theta)} \big)^{\frac{1}{\theta}}=:\delta_0$ $(<1)$. Assume $2\delta<\delta_0$.   

We claim that $\Phi_{\epsilon} (t,x,y) \leq 0$ for all  $( t,x,y ) \in \Delta$, and 
     arguing by contradiction we suppose that 
     $\sup_{\Delta} \Phi_{\epsilon} ( t,x,y ) >0$. In this case, since $\Phi_\epsilon \to -\infty$ uniformly as $t\to t_0^-$,  we have that $\sup_{\Delta} \Phi_{\epsilon} $ must be attained for some $t<t_0$. 
    Since
    \begin{align*}
          0< \Phi_{\epsilon} ( t,x,y ) \leq  \omega_{\delta,B,t_0} (u)  - K \psi ( x-y ) - L \phi ( x-x_{0} ) - 
        C_{0} ( t-t_{0} )^{2} - \frac{\epsilon}{t_0-t},
    \end{align*}
    by choosing
    \begin{align*}
     L = \frac{  \omega_{\delta, B,t_0} (u  )}{c_\phi \delta^2} , \qquad C_{0} = \frac{4 \omega_{\delta, B,t_0} (u  ) }{t_{0}^{2}}, 
    \qquad K \geq \frac{  \omega_{\delta, B,t_0} (u  ) }{\psi ( \delta )},
    \end{align*}
   the maximum of $\Phi_{\epsilon}$ in $\overline{\Delta}$ cannot be attained
    when $t = \frac{t_{0}}{2}$, $| x-y | = \delta $, or  $| x-x_{0} | = \delta$  since  $\Phi_\epsilon\leq 0$ in those cases. Moreover, if $| y-x_{0} | = 2\delta $, then $| x-x_{0} | \geq |y-x_0|-|x-y|\geq \delta $
    and hence the maximum can  not either be attained when $| y-x_{0} | = 2\delta$.
    Thus, the maximum is attained in the interior of $\Delta$. Furthermore, we can exclude  that 
    $x = y$, otherwise $\Phi_{\epsilon} \leq 0$.  
    
    Let $ ( t , x , y  )  $ denotes such a maximum  point of $\Phi_{\epsilon}$. Then
    setting $\zeta (t,x,y) = K \psi ( x-y ) + L \phi ( x-x_{0} ) + 
        C_{0} ( t-t_{0} )^{2} + \frac{\epsilon}{t_0-t}$ and
     subtracting sub- and supersolution inequalities (cf. \cite[Theorem 2.2]{JaKa05}), we have
    \begin{align}
    \label{geco}
    \begin{split}
        \frac{\epsilon}{ ( t_0-t  )^{2}} + 2 C_{0} ( t -t_{0} ) 
        - \mathcal{L} ( x, u, D_{x} \zeta ) + \mathcal{L} ( y, u, -D_{y} \zeta ) &\\
   + \hat H ( x, D_{x} \zeta ) - \hat H ( y, -D_{y} \zeta ) &\leq 0.
    \end{split}
    \end{align}  
Let us now estimate the different terms.  
By \rife{hath}, we have
 \begin{align*}
     \hat H ( x, D_{x} \zeta ) - \hat H ( y, -D_{y} \zeta ) &  \geq
     - \hat C  \big(2+ 2K \psi' + L  | D \phi ( x-x_0 ) | \big) \\
     & \geq  - \hat C  \big(2+ 2K   + L c_\beta\delta \big),
\end{align*}    
where we used that $\psi'\leq 1$ and $| D \phi (x) | \leq \beta \langle x\rangle^{\beta-1} |x|$.  
%

We estimate the nonlocal terms  using as before Lemma \ref{Estimates_non_local_operator}; in particular (as in Step 2 of Proposition \ref{lipest}), in our function $\psi$ we can choose  any   $\theta<\sigma-1$ and $\rho$   large enough   such that  we get
    \begin{align*}
        \mathcal{L} ( x, u, D_{x} \zeta ) - \mathcal{L} ( y, u , - D_{y} \zeta ) 
        \leq L \, \mathcal{L}\phi ( x-x_0)   - 
        K\Lambda \de^{\theta+1 -\sigma },
    \end{align*}
for some constant $\Lambda > 0$.  
Combining the above two estimates, we deduce from \rife{geco}
\begin{align*}
      &
      K(\Lambda \delta^{- \tilde \theta }  - 2\hat C )    \leq 
C_{0} t_{0} + L \, \mathcal{L} \phi( x-x_0)   + \hat C  \big(2+   L c_\beta \delta \big).
\end{align*}
where $\tilde \theta= \sigma-1-\theta >0 $.
Now we choose first 
\begin{align}\label{delta}
\delta=\min\Big(\delta_0/2, 
(\frac{2\hat C +1}{\Lambda})^{-\frac1{\tilde \theta}}\Big),
\end{align}
so that $\Lambda \delta^{- \tilde{\theta} }  - 2 \hat C \geq1,$ and we get
\begin{align*}
K &  \leq 
C_{0} t_{0} + L \, \mathcal{L} \phi( x-x_0)   + \hat C  \big(2+   L  c_\beta \delta \big)
\\
& \leq c\, \omega_{\delta,B,t_0} (u) \left( \frac1{t_0} + \frac1{\de^2}(\mathcal{L} \phi( x-x_0)  +  c_\beta \delta)\right) + 2\hat C\;, 
\end{align*}
where we used the initial choice of $C_0$ and $L$. 
Choosing $K$ sufficiently large (depending on $\hat C$, $\omega_{\delta,B,t_0} (u)$, $\de$ and $\frac1{t_0}$), we get a  contradiction.
Hence $\sup_{\Delta}\Phi_{\epsilon}\leq 0$.  Letting $\vep \to 0$, we have proved that
\begin{align*}
u  ( t,x ) - u  ( t,y )  & \leq  K \psi ( |x-y| ) + L \phi ( x-x_{0} ) + C_{0} ( t-t_{0} )^{2} \\
&  \leq K  
|x-y|+ L \phi ( x-x_{0} ) + C_{0} ( t-t_{0} )^{2} \;.
\end{align*}
We now let $t\to t_0$, we take  $x=x_0$, and we obtain
%
\begin{align*}
    |u (t_0,x_0) - u(t_0,y)| \leq   K   |y-x_0| \qquad\text{in}\qquad B(x_0, 2\delta). 
\end{align*}
Note that $\delta$ is independent of $t_0,x_0, B,u$, and $   K  $ is independent of $x_0\in B$. Moreover, $  K   \to\infty$ as $t_0\to0$ and remains bounded as $t_0\to\infty$. 

However, if $u_0$ is Lipschitz, then one can take $C_0=0$ in the above proof, and let $t$ vary in $(0,T)$. The conclusion will hold up to $t=0$ provided the constant $K$ is larger than the Lipschitz constant of $u_0$.
\smallskip

{\it Step 2:}\quad Since $x_0$,  in Step 1,  was an arbitrary point in $B$, we deduce   \rife{lipB} in a standard way by using the convexity of $B$.

Let $x,y\in B$ and $t_0>0$.
Take $s_0=0 < s_1 < \ldots < s_N = 1$, such that
for $x_i :=(1-s_i)x + s_i y $ ($x_i$ is on the line through $x$ and $y$) and $|x_i - x_{i+1}| \leq \delta/2$. 
Then $x_{i+1}\in B(x_i,\delta)$ for every $i$, and we can use (repeatedly) step 1 with $x_0=x_i$ and $x=x_{i+1}$ to conclude that
\begin{align*}
   \quad|u (t_0,x) - u (t_0,y) | \leq \sum_{i=0}^N |u (t_0,x_i) - u(t_0,x_{i+1}) | \leq\sum_{i=0}^N K |x_i-x_{i+1}|= K |x-y|.\qquad \ \qedhere
\end{align*}
\end{proof}

We conclude this section 
by showing  that the  Hamilton-Jacobi equation  \rife{eqn:parabolic_HJB} enjoys local-in-time (but  global-in-space) regularizing effects. 

\begin{proposition}\label{regeff} Assume \ref{nu'},  \ref{H1} and that $b\in C((0,T)\times \R^d)$ satisfies
\beq\label{bloc}
\exists \,\, \beta>0\,:\,\,\,  (b(t,x)-b(t,y))  \cdot (x-y) \geq  -\beta \, |x-y |   \quad \forall x,y \in \R^d:\, |x-y|\leq 1 \, ,  t>0
\,.
\eeq 
Let $u$ be a  (viscosity) solution of \rife{eqn:parabolic_HJB} such that $|u(t,x)| \leq C (1+ |x|)^\gamma$ for some $\gamma<\sigma$, $C>0$.  
Given  $t_0\in (0,1)$ and $\theta   \in(0,1] $, 
$k<\sigma$, there exists  $K >0$ (depending on  $\beta$, $|b(0)|$, $c_H$, $c_{\sigma,d}$, $d$, $\sigma$, $k$)  such that $u$ satisfies, for every $x,y\in \R^d\,:|x-y|\leq 1$,
\beq\label{to}
 |u(t_0,x)- u(t_0,y)| \leq K \, \frac{([u]_{\langle x\rangle^k}+ t_0 (C_H+  [f]_{\langle x\rangle^k}) )}{ t_0^{ \theta / \sigma} }\,    [\langle x\rangle^k+ \langle y\rangle^k]\,   |x-y|^\theta\,,
\eeq
where $[u]_{\langle x\rangle^k}:= \sup_{\frac{t_0}2<t<  t_0 }\, \,  [u(t)]_{\langle x\rangle^k} $, $[f]_{\langle x\rangle^k}:= \sup_{\frac{t_0}2<t<  t_0 }\, \,  [f(t)]_{\langle x\rangle^k}$,  and $[\cdot]_{\langle x\rangle^k}$ is defined in \eqref{weisem}. 
\end{proposition}

\begin{proof}  
We  follow the approach of \cite{porretta2013global} (where this estimate is proved for local operators) with the help of Lemma \ref{Estimates_non_local_operator} to handle the  nonlocal diffusion.   We take  $q\in (\gamma,\sigma)$ so that $u=o( \langle x\rangle^q)$ as $|x|\to \infty$.  We call $\vfi(t,x)= e^{\lambda t}\langle x\rangle^k$, $\Phi(t,x)= e^{\lambda t}\langle x\rangle^q$,  and  we consider 
\begin{align*}
M_0:  =  \sup_{\Delta}   &  \Big\{ u(t,x)-u(t,y)  - (\vfi(t,x)+ \vfi(t,y) ) \left[ K \psi(|x-y|) + C_0 (t-t_0)^2 \right]  \\
& \qquad\qquad\qquad\qquad\qquad  \qquad   - \vep (\Phi(t,x)+ \Phi(t,y) ) - \frac\vep{(t-t_0)} \Big\}
\\
\intertext{where}
& \Delta= \{(x,y,t)\in \R^{2d}\times\R :\,\,  |x-y|\leq \de\,\,,  t\in [\frac{t_0}2,   t_0 )\},
\end{align*}
with $\de\leq 1$, and we claim that $M_0\leq 0$ (for suitable choices of $K, C_0, \de$ and a concave increasing function $\psi$). 

As usual we reason by contradiction, assuming $M_0>0$. In this case   the $\sup$ is attained at a point $(t,x,y)$ such that $t<t_0$. 
Moreover, $t>\frac {t_0}2$ and $|x-y|\neq\de$, provided  
\beq\label{scel}
C_0\, t_0^2 e^{\la t_0/2}> 4\,   [u]_{\langle x\rangle^k} \qquad \hbox{and} \qquad  K\psi(\de) e^{\la t_0/2} >   [u]_{\langle x\rangle^k}\,.
\eeq
This choice implies that 
 $(t,x,y)$ is a local maximum point  
where  we can use the viscosity formulations. 
In a similar way as in  the proof of  Proposition \ref{lipest}, then we obtain 
\begin{align} 
    \label{visco_dance}
    & \partial_t \zeta - \mathcal{L} ( x, u(x), D_{x} \zeta ) + \mathcal{L} ( y, u(y), -D_{y} \zeta ) +   b ( t,x )\cdott  D_{x} \zeta  -
 b (t, y )\cdott ( - D_{y} \zeta )  \\
\notag &  
+ H ( x, D_{x} \zeta ) - H ( y, -D_{y} \zeta )  \leq f ( t,x ) - f (t,  y )
\end{align} 
where 
$$
\zeta=(\vfi(t,x)+ \vfi(t,y) ) \left[ K \psi(|x-y|) + C_0 (t-t_0)^2 \right]  +\vep (\Phi(t,x)+ \Phi(t,y) )\,.
$$
We make use of  Lemma \ref{Estimates_non_local_operator}, observing that, for $\delta <\frac12$, we have 
$$
\sup_{\xi\in x+ B_\de} |D\vfi(\xi)| \leq c_k e^{\lambda t} \langle x\rangle^{k-1}\leq c_k \vfi(t,x) \,,  
\quad \inf_{\xi\in x+ B_\de} \vfi(\xi)\geq c_k^{-1} e^{\lambda t}  \langle x\rangle^k=c_k^{-1} \vfi(t,x),
$$
for some constant $c_k$ only depending on $k$. Hence, from Lemma \ref{Estimates_non_local_operator} we get, for some constants $c_0, c_1$,
\begin{align*}
& \mathcal{L} ( x, u(x), D_{x} \zeta ) - \mathcal{L} ( y, u(y), -D_{y} \zeta )  
\\ & \leq\left[ K\psi(r) + C_0 (t-t_0)^2 \right] \left(  \cL\vfi(t,x) + \cL\vfi(t,y)\right) + \vep \left( \cL \Phi(t,x)+  \cL \Phi(t,y)\right) 
\\ & \quad + c_0 \,  K  \left[ \vfi(t,x)+ \vfi(t,y)\right] \,   r^{2-\sigma} \,\psi'(r)
\\ & \quad + c_1\,   K  \left[ \vfi(t,x)+ \vfi(t,y)\right]    \int_0^1 (1-s)\int_{B }   \psi'' ( \xi_{r,\sigma})) |\widehat{x-y}\cdott z|^2 \frac{dz}{|z|^{d+\sigma}} ds
\end{align*}
where $r=| x-y|$, $B=\{z\,:\, |z|<\frac{|x-y|}4\}$ and $\xi_{r,\sigma}= r+ 2s(\widehat{x-y}\cdott z)$.

As for the first order terms, using condition \rife{bloc} and \ref{H1} we have
\begin{align*}
& b (t, x )\cdott  D_{x} \zeta  -
 b ( t,y )\cdott ( - D_{y} \zeta )  
+ H ( x, D_{x} \zeta ) - H ( y, -D_{y} \zeta )   \\ & \geq - K\, (\beta+ C_H)  \psi'(r) (\vfi(t,x)+ \vfi(t,y) ) - 2C_H
\\ &\quad+ \left[ K\psi(r) + C_0 (t-t_0)^2 \right] \\
&\qquad\qquad \cdot\Big( b(t,x)\cdott D\vfi(t,x)- C_H |D\vfi(t,x)| + b(t,y) \cdott D\vfi(t,y) - C_H |D\vfi(t,y)| \Big)  \\
&\quad + \vep \Big( b(t,x)\cdott D\Phi(t,x)- C_H |D\Phi(t,x)| + b(t,y) \cdott D\Phi(t,y)- C_H |D\Phi(t,y)|  \Big). 
\end{align*}
Putting together all the above estimates, and computing $\partial_t \zeta$, we deduce from \rife{visco_dance}
\begin{align*}
& 
\frac\vep{(t_0-t)^2} + 
 2C_0(t-t_0) (\vfi(t,x)+ \vfi(t,y) )  \\
 & \quad+  \vep  \Big(( \partial_t \Phi + \cL_{C_H}\Phi)(t,x)+ (\partial_t \Phi+\cL_{C_H}\Phi)(t,y)\Big)
\\
& \quad+ \left[ K\psi(r) + C_0 (t-t_0)^2 \right] \Big( (\partial_t\vfi + \cL_{C_H}\vfi)(t,x)+ (\partial_t \vfi+\cL_{C_H}\vfi)(t,y)\Big) 
\\
&\leq   2C_H + K\,    \left[ \vfi(t,x)+ \vfi(t,y)\right]  \left\{ (\beta+ C_H)  + c_0 \,  r^{2-\sigma} \right\} \psi'(r)  
\\
& \quad + c_1\,   K  \left[ \vfi(t,x)+ \vfi(t,y)\right]   \int_0^1 (1-s)\int_{B }   \psi'' ( \xi_{r,\sigma}) |\widehat{x-y}\cdott z|^2 \frac{dz}{|z|^{d+\sigma}} ds  \notag
                \\
                & \quad   + \left( f(t,x)- f(t,y)\right),
\end{align*}
where we recall the notation $\cL_{C_H}$ from  \rife{Lambda1}  (now with $b=b(t,x)$).  Since assumption \rife{bloc} implies 
$b(t,x)\cdot x \geq - ( |b(t,0)| + \beta) |x|$,  we estimate, as in the computations of 
Lemma \ref{lemma:supersolution_property}, that $\cL_{C_H}\langle x\rangle^k\geq - c\langle x\rangle^k$ for some $c$  depending on $\beta, |b(0)|, c_H, c_{\sigma,d}, d, \sigma, k$. Hence we can choose  $\lambda$   so that  $\partial_t \vfi + \cL_{C_H}\vfi \geq 0$ in $\R^d$, for all $t\in (0,1)$. Similarly, we can assume that, for the same $\lambda$, it also holds $\partial_t \Phi+ \cL_{C_H}\Phi(t,x) \geq 0$.  Having fixed $\lambda$ this way, we can drop the contributions of the Lyapunov functions and we get
\begin{align*}
& 
\quad 2C_0(t-t_0) (\vfi(t,x)+ \vfi(t,y) )  \leq   2C_H + K\,    \left[ \vfi(t,x)+ \vfi(t,y)\right]  \left\{ (\beta+ C_H)  + c_0 \,  r^{2-\sigma} \right\} \psi'(r)  
\\
& \qquad + c_1\,   K  \left[ \vfi(t,x)+ \vfi(t,y)\right]   \int_0^1 (1-s)\int_{B }   \psi'' ( \xi_{r,\sigma}) |\widehat{x-y}\cdott z|^2 \frac{dz}{|z|^{d+\sigma}} ds  \notag
                \\
                & \qquad\qquad\qquad   + \left( f(t,x)- f(t,y)\right)
\end{align*}
which implies (using $\vfi \geq 1$ and $f(t,x)- f(t,y)\leq  [f]_{\langle x\rangle^k}\left[ \vfi(t,x)+ \vfi(t,y)\right]$)
\begin{align*}
& 
\quad 2C_0(t-t_0)    \leq   2C_H + [f]_{\langle x\rangle^k} + K\,  \left\{ (\beta+ C_H)  + c_0 \,  r^{2-\sigma} \right\} \psi'(r)  
\\
& \qquad + c_1\,   K     \int_0^1 (1-s)\int_{B }   \psi'' ( \xi_{r,\sigma}) |\widehat{x-y}\cdott z|^2 \frac{dz}{|z|^{d+\sigma}} ds \,.
\end{align*}
The  term with $\psi'$ can be estimated as we did in \rife{psi2}; hence,  we obtain
\begin{align}\label{bnd_concl}
\begin{split}
2C_0(t-t_0) & \leq   2C_H + [f]_{\langle x\rangle^k}  
\\ & +  c_1\, K  \left\{ \int_0^1 (1-s)\int_{B^- }   (\psi'' ( \xi_{r,s})+ c \, \frac{\psi'(\xi_{r,s})}{\xi^{2-\sigma}_{r,s}}) |\widehat{x-y}\cdott z|^2 \frac{dz}{|z|^{d+\sigma}} ds  \right\}
\end{split}
\end{align}
for some $c$ independent of $K, \delta$, where  $B^-= \{z\in B\,:\, z\cdot \widehat{x-y}\leq 0\}$.  Recall that $\xi_{r,s}\in[\frac r2,r]$ when $z\in B^-$. 

Hereafter, one can proceed with convenient choices of $\psi$. Namely, for $\theta<1$ we choose $\psi(r)= r^\theta$ and we get, for possibly different $\tilde c_1, \tilde c$, 
$$
 2C_0(t-t_0) \leq 2C_H + [f]_{\langle x\rangle^k} - \tilde c_1\, K\, \theta  (  (1-\theta)  r^{\theta- \sigma} - \tilde c r^{\theta-1})  \,.
$$
Notice that $\theta-\sigma< \theta-1$. Since $2C_0(t-t_0) \geq - C_0 t_0$ and since $r\leq \de$, the above inequality cannot hold  if $\de$ is sufficiently small, $K$ sufficiently big, and $K \, \de^{\theta-\sigma} \geq n\, (C_0\, t_0+ 2C_H + [f]_{\langle x\rangle^k}) $ for some large $n$.  On account of \rife{scel}, we obtain a contradiction by choosing 
$$
C_0= 4\,  \, \frac{ [u ]_{\langle x\rangle^k}}{t_0^2}\,, \quad K\gtrsim   \,   ([u ]_{\langle x\rangle^k}+ t_0 (2C_H + [f]_{\langle x\rangle^k})) \de^{- \theta} \quad \hbox{and} \quad \de \lesssim \left(\frac{t_0}{ n}\right)^{\frac1\sigma}.
$$
Indeed, in this case one also gets 
\begin{align*}
K \, \de^{\theta-\sigma}  & \gtrsim    ([u ]_{\langle x\rangle^k}+ t_0  (2C_H + [f]_{\langle x\rangle^k})) \de^{-\sigma}  = ( C_0\,   \frac{t_0}4+    (2C_H+ [f]_{\langle x\rangle^k})) t_0\de^{-\sigma}  \\ &  \gtrsim n\, ( C_0\,   \frac{t_0}4+  2C_H + [f]_{\langle x\rangle^k})
\end{align*}
and the assertion is proved. 
 
Once we have proved that $M_0\leq 0$, we let $\vep\to 0$ and we read the estimate for $t=t_0$, obtaining \rife{to} due to the choice of $K,\de$. 

The case $\theta=1$ can be proved  choosing $\psi$ as the solution of the ODE
$$
\begin{cases}
\psi''(r)+ c \frac{\psi' (r)}{r^{2-\sigma}} = -1 &  r\in (0,\de)\\
\psi(0)=0\,,\,\,\psi'(\de)=0\,.
\end{cases}
$$
Notice that $2-\sigma<1$, so $\frac1{r^{2-\sigma}}$ is integrable in $(0,1)$,  hence we can solve for $\psi'$   to see that it is bounded and positive. Then $\psi''<-1$ by the equation. Using these bounds, the boundary conditions for $\psi$, and Taylor expansions, one gets \nc
(this is similar to  \cite[Thm 3.3]{porretta2013global}, estimates (3.14)--(3.18)):
\beq\label{psiode}
  \psi'(0)\leq \tilde c \, \de \,,\qquad \psi(r)\leq \psi'(0)r\leq \tilde c\,\de\,r\,,\qquad \psi(\de)\geq \frac1{\tilde c} \, \de^2\,,
\eeq
for some constant $\tilde c>0$.  Then by \eqref{bnd_concl}, we have 
$$
 2C_0(t-t_0) \leq 2C_H + [f]_{\langle x\rangle^k} - \tilde c_3\, K\, r^{2-\sigma}
 \,.
 $$
 To get a contradiction as above,  
one needs $\de$  sufficiently small, $K$ sufficiently big, with $K \de^{2-\sigma}\geq n\, (C_0\, t_0+ 2C_H + [f]_{\langle x\rangle^k}) $ for some large $n$.  Due to \rife{scel} and \rife{psiode}, the conclusion follows by taking
$$
C_0= 4\,  \, \frac{ [u ]_{\langle x\rangle^k}}{t_0^2}\,, \quad K\gtrsim   \,   ([u ]_{\langle x\rangle^k}+ t_0  (2C_H+  [f]_{\langle x\rangle^k}) ) \de^{- 2} \quad \hbox{and} \quad \de \lesssim \left(\frac{t_0}{ n}\right)^{\frac1\sigma}.
$$
Then, after letting $\vep \to 0$ and taking $t=t_0$, one gets the estimate
$$
u(t,x)-u(t,y)  \leq    (\vfi(t,x)+ \vfi(t,y) ) K  \psi(|x-y| ) \leq (\vfi(t,x)+ \vfi(t,y) ) \tilde c \, \de\, K |x-y|
$$
where we used \rife{psiode}. Recalling the choices of $K,\de$ yields \rife{to}.
%
\end{proof}

\begin{remark}\label{lipda0} 
 When   $\theta=1$, \rife{to} provides a gradient estimate. E.g.   if $C_H=0$,
$$
\|Du(t)\|_{L^\infty(\langle x\rangle^{-k})}\leq \frac K{t^{\frac1\sigma}} \, \left( \sup_{s\in [\frac t2, t]} [u(s)]_{\langle x\rangle^k} \right)   + K\, t^{1-\frac1\sigma} \left( \sup_{s\in [\frac t2, t]} [f(s)]_{\langle x\rangle^k} \right)
$$
for $t\in (0,1)$,  which shows the classical short time regularizing effect for linear $\sigma$-diffusive operators.  

Of course, a similar proof as above can work up to $t=0$ if the initial datum $u_0$ is locally Lipschitz continuous, with a weighted global bound. In this case, there is no need of time localization in the above proof, and one obtains e.g. the estimate
\beq\label{estlipda0}
\|Du(t)\|_{L^\infty(\langle x\rangle^{-k})}\leq  \|Du_0\|_{L^\infty(\langle x\rangle^{-k})} + K \left( \sup_{s\in [0, 1]} [f(s)]_{\langle x\rangle^k} \right)\,,\qquad \forall t\in (0,1).
\eeq
\end{remark}



\end{document}